\documentclass[a4paper,reqno]{amsart}

\usepackage{amssymb}
\usepackage{amstext}
\usepackage{amsmath}
\usepackage{amscd}
\usepackage{amsthm}
\usepackage{amsfonts}
\usepackage{bbm}
\usepackage{enumerate}
\usepackage{graphicx}
\usepackage{latexsym}
\usepackage{mathrsfs}
\usepackage{mathtools}
\usepackage[all]{xy}
\xyoption{all}

\usepackage{pstricks}
\usepackage{lscape}
\usepackage{comment}

\newtheorem{theorem}{Theorem}[section]

\newtheorem{corollary}[theorem]{Corollary}
\newtheorem{lemma}[theorem]{Lemma}
\newtheorem{proposition}[theorem]{Proposition}
\newtheorem{definition-proposition}[theorem]{Definition-Proposition}

\theoremstyle{definition}
\newtheorem{definition}[theorem]{Definition}

\newcommand{\kk}{{\mathbbm{k}}}
\newcommand{\q}{{\mathfrak{q}}}
\newcommand{\ok}{{\mathfrak{o}}}
\newcommand{\tk}{{\mathfrak{t}}}
\newcommand{\B}{{\mathcal{B}}}
\newcommand{\Ext}{\mbox{Ext}}
\newcommand{\Tor}{\mbox{Tor}}
\newcommand{\I}{\mbox{\rm Id}}
\newcommand{\Hom}{\mbox{\rm Hom}}

\newcommand{\ch}{\mbox{\rm Char}}

\newcommand{\Imm}{\mbox{\rm Im}}
\newcommand{\Ker}{\mbox{\rm Ker}}
\newcommand{\rank}{\mbox{\rm rank}}
\newcommand{\w}{\mbox{$\wedge^{\ast}$}}
\newcommand{\boti}{\,\bar{\otimes}\,}

\numberwithin{equation}{section}

\usepackage{paralist}

\begin{document}
\title[BV structure on Hochschild cohomology of zigzag algebra]
{Batalin-Vilkovisky structure on Hochschild cohomology of zigzag
algebra of type $\widetilde{\mathbf{A}}_{1}$}
\thanks{The authors are supported by NSFC(Nos.11771122,\, 11801141 and 11961007). }

\author{Bo Hou,\ \ Jin Gao }
\address{School of Mathematics and Statistics, Henan University,
Kaifeng 475001, PR China.}
\email{bohou1981@163.com, \ \ gaojin528@163.com.}

\maketitle
\begin{abstract}
In this paper, we study the Batalin-Vilkovisky structure on the Hochschild cohomology
of quantum zigzag algebras $A_{\q}$ of type $\widetilde{\mathbf{A}}_{1}$.
We first calculate the dimensions of Hochschild homology groups and Hochschild
cohomology groups of $A_{\q}$. Based on these computations, we determine the Hochschild
cohomology ring of $A_{\q}$, and give the Batalin-Vilkovisky
operator and the Gerstenhaber bracket on Hochschild
cohomology ring of $A_{\q}$ explicitly.
\end{abstract}

\bigskip
\section{Introduction}\label{intro}
Let $\Lambda$ be an algebra (associative with
unity) over a field $\kk$. Denote by
$\Lambda^{e}:=\Lambda\otimes_{\kk}\Lambda^{op}$ the enveloping
algebra of $\Lambda$. Then the $i$-th Hochschild homology and
Hochschild cohomology of $\Lambda$ are identified with the
$\kk$-spaces (see \cite{CE})
$$
HH_{i}(\Lambda)=\Tor^{\Lambda^{e}}_{i}(\Lambda, \Lambda),\qquad\qquad
HH^{i}(\Lambda)=\Ext^{i}_{\Lambda^{e}}(\Lambda, \Lambda),
$$
respectively.
The Hochschild homology and Hochschild cohomology of an algebra
are subtle variants of associative algebras and have played a
fundamental role in representation theory of artin algebras.
Hochschild homology is closely related to the
oriented cycle and the global dimension of algebras;
Hochschild cohomology is closely related to
simple connectedness, separability and deformation theory.
The cohomology ring $HH^{\ast}(\Lambda)=
\bigoplus_{i\geq0}HH^{i}(\Lambda)$ is a graded commutative algebra
under the cup product, and is a graded Lie algebra under the Gerstenhaber bracket
(see \cite{Ger}).

During several decades, a new structure in Hochschild theory
has been extensively studied in topology and mathematical physics,
and recently this was introduced into algebra, the
Batalin-Vilkovisky structure (BV structure for short).
A BV structure exists only on Hochschild cohomology of certain
special classes of algebras. For example, Tradler has founded that
the Hochschild cohomology algebra of a finite-dimensional
symmetric algebra is a BV algebra \cite{T}; Lambre, Zhou Zimmermann, and
independently Volkov have showed that the Hochschild cohomology ring
of a Frobenius algebra with semisimple Nakayama automorphism is
a BV algebra by different methods, and generalized Tradler's result
(see \cite{LZZ} and \cite{Vo}).
On the other hand, Ginzburg has proved that there is a BV algebraic
structure on the Hochschild cohomology of a Calabi-Yau algebra \cite{Gi};
Kowalzig and Kr$\ddot{a}$hmer have generalized Ginzburg's conclusion to twisted
Calabi-Yau algebras \cite{KK}.

For some special algebra classes, we already know that there is
a BV algebraic structure on its Hochschild cohomology ring,
but it is very difficult to describe this structure concretely.
Up to now, there has been little research in this field, mainly focusing
on group algebras and local algebras. The BV algebraic structure on the
Hochschild cohomology of a class of truncated polynomial algebras
was calculated by Yang \cite{Y};
In \cite{Iv} and \cite{V}, the BV algebraic structure on the
Hochschild cohomology of some class of local algebras of generalized
quaternion type is described over a field of characteristic two.
In \cite{LZ}, the authors have given a description of the BV
structure on the Hochschild cohomology ring for symmetric group
of degree 3 over $\mathbb{F}_{3}$; Ivanov, Ivanov, Volkov, and
Zhou have computed the BV structure on the Hochschild
cohomology ring of the group algebra $\kk Q_{8}$ over an algebraically
closed field of characteristic two \cite{IIVZ}; Volkov has calculated
the BV structure on the Hochschild cohomology ring of a family of
self-injective algebra of tree type $D_{n}$ \cite{Vol}. Recently, Angel and
Duarte have studied the BV structure on the Hochschild
cohomology ring of finitely generated abelian groups \cite{AD}.
In this and a subsequent paper, we consider a class of
non-group non-local algebras, that is, zigzag algebras. More precisely,
in this paper, we shall deal with (quantum) zigzag algebras.

Zigzag algebras were introduced by Huerfano and Khovanov
in their categorification of the adjoint representation of simply-laced
quantum groups \cite{HK}. Such algebras appear in various places in modern mathematics,
especially in categorification (see \cite{Su}, \cite{X}, \cite{HK},
\cite{KS}, \cite{KMS}, \cite{EK} etc). In \cite{ET}, the authors have studied some
algebraic properties of zigzag algebras and certain generalization of them.
The Hochschild cohomology and the Hochschild cohomology ring modulo
nilpotence of zigzag algebras of type $\widetilde{\mathbf{A}}$ have
been studied in \cite{ST}, \cite{ST1} and \cite{PS}.
In this paper, we consider a broader class of algebras, the quantum
zigzag algebras of type $\widetilde{\mathbf{A}}_{1}$.
These algebras are Koszul self-injective special biserial algebras
and play an important role in representation theory.
Here, the dimensions of Hochschild (co)homology groups, the cup product
and the Gerstenhaber bracket product on the quantum zigzag algebras
of type $\widetilde{\mathbf{A}}_{1}$ are clearly described.
Moreover, the BV structure on the Hochschild cohomology ring of zigzag algebra of type
$\widetilde{\mathbf{A}}_{1}$ is given.

This article is organized as follows. In the second section,
we review the definitions of Hochschild homology and
cohomology, cup product, Gerstenhaber bracket product and
BV algebra. In the third section, we provide a
minimal projective bimodule resolution of quantum zigzag algebras
$A_{\q}$ of type $\widetilde{\mathbf{A}}_{1}$, and
by using the language of closed paths, we calculate the $\kk$-dimensions
of Hochschild homology groups and cyclic homology
groups of $A_{\q}$. Here, we give a positive answer to Han's conjecture for $A_{\q}$.
In the fourth section, we give a explicit basis of each degree
of Hochschild cohomology groups of $A_{\q}$ by the parallel paths.
In the fifth section, we prove that the product of cohomology rings
of $A_{\q}$ is essentially the connection of parallel paths.
By using previous calculations,
the structure of Hochschild cohomology ring of algebras $A_{\q}$
and Hochschild cohomology ring modulo nilpotence are clearly depicted.
An positive answer to Snashall-Solberg conjecture is given for $A_{\q}$.
In the final section, we construct two comparison morphisms between the
minimal projective bimodule resolution given in Section 3 and the
reduced bar resolution of $A_{\q}$. Using these comparison morphisms,
and applying Tradler's construction to the zigzag algebra $A_{-1}$,
we get the BV operator and Gerstenhaber bracket on Hochschild cohomology
ring of $A_{-1}$. If $\q\neq-1$, $A_{\q}$ is not a symmetric algebra,
but a Frobenius algebra with semisimple Nakayama automorphism.
Using the bilinear form constructed by Volkov in \cite{Vo},
we give an exact description of the BV operator and Gerstenhaber bracket
on Hochschild cohomology ring of $A_{\q}$ case by case.

Throughout this paper, we fix $\kk$ a field and often write $\otimes$
in place of $\otimes_{\kk}$ for brevity.

\bigskip
\section{Hochschild (co)homology of associative algebra}\label{(co)homology}

The cohomology theory of associative algebras was introduced by
Hochschild (see \cite{Ho}).
Let $\Lambda$ be an associative algebra over a field $\kk$. The Hochschild
cohomology $HH^{\ast}(\Lambda)$ of $\Lambda$ has a very rich structure.
In this section, we recall the cup product, the Gerstenhaber bracket and
Batalin-Vilkovisky structure in Hochschild cohomology.

For an associative $\kk$-algebra $\Lambda$, there is a
projective bimodule resolution ${\mathbb{B}}=(B_{m}, d_{m})$ of $\Lambda$
as following:
$$
\xymatrix@C=1.5em{
\cdots \ar[r] & \Lambda^{\otimes(m+2)}\ar[r]^{d_{m}}
& \Lambda^{\otimes(m+1)}\ar[r]&\cdots\ar[r]
& \Lambda^{\otimes(3)}\ar[r]^{d_{1}}& \Lambda^{\otimes(2)}\ar[r]^{d_{0}}
& \Lambda\ar[r]& 0  }
$$
where $d_{0}$ is the multiplication map, $B_{m}=A^{\otimes (m+2)}$ for $m\geq 0$,
and $d_{m}$ is defined by
$$
d_{m}(a_{0}\otimes a_{1}\otimes\cdots\otimes a_{m+1})=
\sum_{i=0}^{m}(-1)^{i}
a_{0}\otimes\cdots\otimes a_{i-1}\otimes a_{i}a_{i+1}\otimes
a_{i+2}\otimes \cdots\otimes a_{m+1},
$$
for any $a_{0}, a_{1}, \cdots, a_{m+1}\in\Lambda$.
This is called the bar resolution of $\Lambda$.

Let $\{e_{1},e_{2},\cdots,e_{l}\}$ be a complete set of primitive
orthogonal idempotents of $\Lambda$, $E$ the subalgebra of $\Lambda$
generated by $\{e_{1},e_{2},\cdots,e_{l}\}$.
Denote by $\bar{\Lambda}=\Lambda/E$, the quotient $\kk$-module,
and $\bar{B}_{m}=\Lambda\otimes_{E}
\bar{\Lambda}^{\otimes_{E} m}\otimes_{E} \Lambda$.
Then the quotients $\bar{B}_{m}$ constitute a complex
$\bar{{\mathbb{B}}}=(\bar{B}_{m}, \bar{d}_{m})$,
where the differential $\bar{d}_{m}$ induced from $d_{m}$,
for all $m\geq0$. The complex $\bar{{\mathbb{B}}}$ is also a
projective bimodule resolution of $\Lambda$, which is called the
reduced bar resolution of $\Lambda$.

Applying functor $\Hom_{\Lambda^{e}}(-, \Lambda)$ to the complex
${\mathbb{B}}$, we get a complex
$\Hom_{\Lambda^{e}}({\mathbb{B}}, \Lambda)$. Note that for each $m\geq 0$,
$\Hom_{\Lambda^{e}}(B_{m}, \Lambda)\cong\Hom_{\kk}(\Lambda^{\otimes m}, \Lambda)$,
we can use the complex ${\mathbb{C}}=(C^{m},\; \delta^{m})$ to calculate
the Hochschild cohomology of $\Lambda$, where
$C^{m}=\Hom_{\kk}(\Lambda^{\otimes m}, \Lambda)$, and
$$
\begin{aligned}
\delta^{m}(f)(a_{1}\otimes\cdots\otimes a_{m+1})
=&a_{1}f(a_{2}\otimes\cdots\otimes a_{m+1})\\
&+\sum_{i=1}^{m}(-1)^{i}f(a_{1}\otimes\cdots\otimes a_{i-1}\otimes
a_{i}a_{i+1} \otimes a_{i+2}\otimes\cdots\otimes a_{m+1})\\
&+(-1)^{m+1}f(a_{1}\otimes \cdots \otimes a_{m})a_{m+1},
\end{aligned}
$$
for any $f\in\Hom_{\kk}(\Lambda^{\otimes m}, \Lambda)$ and
$a_{1}\otimes \cdots\otimes a_{m+1}\in\Lambda^{\otimes m+1}$.

The cup product $\alpha\sqcup\beta\in C^{m+l}(\Lambda)
=\Hom_{\kk}(\Lambda^{\otimes (m+l)}, \Lambda)$
for $\alpha\in C^{m}(\Lambda)$  and $\beta\in C^{l}(\Lambda)$
is given by
$$
(\alpha\sqcup\beta)(a_{1}\otimes\cdots\otimes a_{m+l})=
\alpha(a_{1}\otimes\cdots\otimes a_{m})
\beta(a_{m+1}\otimes\cdots\otimes a_{m+l}).
$$
This cup product induces a well-defined product in Hochschild cohomology
$$
\sqcup:\quad HH^{m}(\Lambda)\times HH^{l}(\Lambda)\longrightarrow
HH^{m+l}(\Lambda),
$$
which turns the graded $\kk$-vector space $HH^{\ast}(\Lambda)=
\bigoplus_{i\geq0}HH^{i}(\Lambda)$ into a graded commutative algebra.

Besides addition and multiplication, there is another binary operation
on $HH^{\ast}(\Lambda)$, which is called Gerstenhaber bracket.
Let $\alpha\in C^{m}(\Lambda)$ and $\beta\in C^{l}(\Lambda)$.
If $m, l\geq 1$, then for $1\leq i\leq m$, define
$\alpha\widehat{\circ}_{i}\beta\in C^{l+m-1}(\Lambda)$ by
$$
\begin{aligned}
&(\alpha\widehat{\circ}_{i}\beta)(a_{1}\otimes\cdots\otimes a_{m+l-1})\\
=&\alpha(a_{1}\otimes\cdots\otimes a_{i-1}\otimes
\beta(a_{i}\otimes\cdots\otimes a_{i+l-1})\otimes a_{i+l}\otimes
\cdots \otimes a_{m+l-1});
\end{aligned}
$$
if $m\geq 1$ and $l=0$, then $\beta\in\Lambda$ and for
$1\leq i\leq m$, define
$$
(\alpha\widehat{\circ}_{i}\beta)(a_{1}\otimes\cdots\otimes a_{m-1})
=\alpha(a_{1}\otimes\cdots\otimes a_{i-1}\otimes \beta
\otimes a_{i}\otimes \cdots \otimes a_{m-1});
$$
for any other case, $\alpha\widehat{\circ}_{i}\beta=0$.
Now we can define the Gerstenhaber bracket. Let
$$
\alpha\widehat{\circ}\beta
=\sum_{i=1}^{m}(-1)^{(l-1)(i-1)}\alpha\widehat{\circ}_{i}\beta,
$$
and
$$
[\alpha,\; \beta]=\alpha\widehat{\circ}\beta-
(-1)^{(m-1)(l-1)}\beta\widehat{\circ}\alpha.
$$
The above $[\ \;,\;\ ]$ induces a well-defined graded Lie bracket in
Hochschild cohomology
$$
[\ \;,\;\ ]:\quad HH^{m}(\Lambda)\times HH^{l}(\Lambda)
\longrightarrow HH^{m+l-1}(\Lambda)
$$
This graded Lie bracket is usually called the Gerstenhaber bracket in
$HH^{\ast+1}(\Lambda)$. It is well-known that
$(HH^{\ast}(\Lambda),\; \sqcup,\; [\ \;,\;\ ])$ is a Gerstenhaber
algebra (see \cite{Ger}). That is, the following conditions hold:
\begin{itemize}
\item[(1)] $(HH^{\ast}(\Lambda),\; \sqcup)$ is an associative algebra;

\item[(2)] $(HH^{\ast+1}(\Lambda),\; [\ \;,\;\ ])$ is a graded Lie algebra
with bracket $[\ \;,\;\ ]$ of degree $-1$;

\item[(3)] $[f\sqcup g,\; h]=[f,\; h]\sqcup g+(-1)^{|f|(|h|-1)}f\sqcup[g,\; h]$,
where $|f|$ denotes the degree of $f$.
\end{itemize}
What we want to explain here is that if we use the reduced bar resolution to replace
the bar resolution, we can also get the  Gerstenhaber algebraic structure on
$HH^{\ast}(\Lambda)$ by using the same formula to define the cup product and
the Gerstenhaber bracket.

If there is an operator on Hochschild cohomology
which squares to zero and together with the cup product
can express the Lie bracket, then it is a Batalin-Vilkovisky algebra.
Let us review the definition of Batalin-Vilkovisky algebra
(see, for example \cite{T}).

\begin{definition} A Batalin-Vilkovisky algebra is a Gerstenhaber algebra
$(\Lambda^{\bullet},\; \sqcup,\; [\ \,,\,\ ])$ together with an operator
$\Delta: \Lambda^{\bullet} \rightarrow\Lambda^{\bullet-1}$ of degree
$-1$ such that $\Delta\circ\Delta=0$ and
$$
[a,\, b]=-(-1)^{(|a|-1)|b|}\Big(\Delta(a\sqcup b)-\Delta(a)\sqcup b
-(-1)^{|a|}a\sqcup \Delta(b)\Big)
$$
for homogeneous elements $a, b\in\Lambda^{\bullet}$.
\end{definition}

It seems that the definition of Batalin-Vilkovisky algebra is different
from the conventional one, see \cite{Ge}, but they are isomorphic. This
definition is chosen here to facilitate the calculation of the BV operator.
For any associative $\kk$-algebra with unity, In \cite{Ger},
the author proved that $(HH^{\ast}(\Lambda),\; \sqcup,\; [\ \;,\;\ ])$
is always a Gerstenhaber algebra. However, for a given algebra,
obtain this structure concretely, that is, detailed describe the
cup product and Gerstenhaber bracket product is very difficult.
The BV operator $\Delta$ does not always exist for the Hochschild
cohomology ring $HH^{\ast}(\Lambda)$ of an algebra $\Lambda$.
So far, we only know that there is a BV operator on Hochschild
cohomology ring of a few kinds of algebras.
Fortunately, for the zigzag algebras of type $\widetilde{\mathbf{A}}$
all of these algebraic structures on Hochschild cohomology can be clearly
depicted in this and a future paper \cite{Hou}.

\bigskip
\section{Hochschild homology groups of $A_{\q}$}\label{homology}

In this section, we construct a minimal projective bimodule resolution
of the quantum zigzag algebras $A_{\q}$, and give a equivalent
description of the homology complex obtained by this minimal
projective bimodule resolution. Furthermore, the dimensions of Hochschild
homology groups of $A_{\q}$ are given explicitly.

Recall that the zigzag algebra of type $\widetilde{\mathbf{A}}_{1}$
is given by the quiver $Q$:
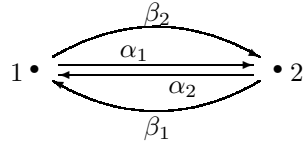
$$
\setlength{\unitlength}{1.6mm}
\begin{picture}(0,12)
\put(-10,6){\circle*{0.6}}    \put(10,6){\circle*{0.6}}
\qbezier(-8.5,7)(0,12)(8.5,7)  \put(7.5,7.5){\vector(2,-1){1}}
\put(-8,6.3){\vector(1,0){16}}
\qbezier(-8.5,5)(0,0)(8.5,5)   \put(-7.5,4.5){\vector(-2,1){1}}
\put(8,5.5){\vector(-1,0){16}}
\put(-3,7){$\alpha_{1}$}   \put(1,4){$\alpha_{2}$}
\put(-1,10){$\beta_{2}$}    \put(-1,0.5){$\beta_{1}$}
\put(-12,5){$1$} \put(11,5){$2$}
\end{picture}
$$
with relations
$\{\alpha_{1}\alpha_{2},\; \alpha_{2}\alpha_{1},\;
\beta_{1}\beta_{2},\; \beta_{2}\beta_{1},\;
\alpha_{1}\beta_{1}-\beta_{2}\alpha_{2},\;
\alpha_{2}\beta_{2}-\beta_{1}\alpha_{1}\}$,
where the composition of paths are from left to right.
Here we consider a broader class of algebras, the quantum zigzag algebras.
The quantum zigzag algebras $A_{\q}$ of type $\widetilde{\mathbf{A}}_{1}$,
are given by the quotient algebras $A_{\q}=\kk Q/I$,
where the ideal $I$ of $\kk Q$ are generated by
$$
R:=\{\alpha_{1}\alpha_{2},\ \ \alpha_{2}\alpha_{1},\ \
\beta_{1}\beta_{2},\ \ \beta_{2}\beta_{1},\ \
\alpha_{1}\beta_{1}+\q\beta_{2}\alpha_{2},\ \
\alpha_{2}\beta_{2}+\q\beta_{1}\alpha_{1}\}.
$$
Here we always assume $\q\neq0$. It is easy to see
that the set $R$ is just a (noncommutative) quadratic
Gr\"{o}bner basis of $I$. Therefore, $\Lambda_{\q}$ is
a Koszul algebra for each $\q\neq 0$ (see \cite{GH}).
Note that $A_{\q}$ is just the zigzag algebra whenever $\q=-1$,
we often record the zigzag algebra of type $\widetilde{\mathbf{A}}_{1}$
as $A_{-1}$.
Denote by $e_{1}$, $e_{2}$ the primitive orthogonal idempotents
corresponding to the vertices $1$ and $2$, respectively. Let
$$
\B:=\{e_{i},\, \alpha_{i},\, \beta_{i},\, \alpha_{i}\beta_{i}\mid i=1, 2\}.
$$
Then $\B$ is a $\kk$-basis of $A_{\q}$, and so that $\dim_{\kk}A_{\q}=8$.

We now construct a minimal projective bimodule resolution for algebra
$A_{\q}$ using the approach of \cite{GHM}. Firstly, setting
$$
\begin{aligned}
&F^{0}:=\left\{f^{0}_{(1,0)}=e_{1},\ \ f^{0}_{(2,0)}=e_{2}\right\}; \\
&F^{1}:=\left\{f^{1}_{(1,1)}=\alpha_{1},\ \ f^{1}_{(2,1)}=\alpha_{2},\ \
f^{1}_{(2,0)}=\beta_{1},\ \ f^{1}_{(1,0)}=\beta_{2}\right\}.
\end{aligned}
$$
For $m\geq2$, we define inductively the set $F^{m}=\left\{f^{m}_{(i,j)}\mid
0\leq j\leq m,\, i=1, 2\right\}$ by
$$
f^{m}_{(i,j)}=\alpha_{i}f^{m-1}_{(i+1,j-1)}+\q^{j}\beta_{i-1}f^{m-1}_{(i-1,j)}.
$$
Then $|F^{m}|=2(m+1)$, and $f^{m}_{(i,j)}=\q^{m-j}f^{m-1}_{(i,j-1)}
\alpha_{i-m-1}+f^{m-1}_{(i,j)}\beta_{i-m}$, where $f^{m-1}_{(i,j)}
=f^{m-1}_{(i',j)}$ if $i\equiv i' \mod2$,
$\alpha_{i}=\alpha_{i'}$ if $i\equiv i' \mod2$,
$\beta_{j}=\beta_{j'}$ if $j\equiv j' \mod2$, and $f^{m-1}_{(i,j)}=0$
if $j<0$ or $j>m-1$.

For any path $p\in Q$, we denote by $\ok(p)$ and $\tk(p)$ the originals and
terminus of $p$. Recall that a non-zero element $x=\sum^{s}_{i=1}a_{i}p_{i}\in\kk Q$,
where $a_{i}\in\kk$ and $p_{i}$ is a path in $Q$,
is said to be uniform if there exist vertices $u, v\in Q_{0}$
such that $\ok(p_{i})=u$ and $\tk(p_{i})=v$ for all paths $p_{i}$.
It is easy to see that elements $f^{m}_{(i,j)}$ are uniform.
Thus for each $f\in F^{m}$, we denote by $\ok(f)$ and $\tk(f)$ the common
originals and terminus of all the paths occurring in $f$,
and always identify $\ok(f)$ and $\tk(f)$ with their corresponding idempotents.

Let
$$
P_{m}:=\bigoplus_{f\in F^{m}}A_{\q}\ok(f)\otimes\tk(f)A_{\q}.
$$
Define $d_{1}:P_{1}\rightarrow P_{0}$ by
$$
d_{1}\left(\ok(f)\otimes\tk(f)\right)=f\otimes\tk(f)-\ok(f)\otimes f,
$$
for $f\in F^{1}$. Whenever $m\geq2$, the differential
$d_{m}:P_{m}\rightarrow P_{m-1}$ is given by:
$$
\begin{aligned}
d_{m}\left(\ok(f^{m}_{(i,j)})\otimes\tk(f^{m}_{(i,j)})\right)
=&\alpha_{i}\otimes\tk(f^{m-1}_{(i+1,j-1)})
+(-1)^{m}\q^{m-j}\ok(f^{m-1}_{(i,j-1)})\otimes\alpha_{i-m-1}\\
&\quad+\q^{j}\beta_{i-1}\otimes\tk(f^{m-1}_{(i-1,j)})
+(-1)^{m}\ok(f^{m-1}_{(i,j)})\otimes\beta_{i-m}.
\end{aligned}
$$

\begin{proposition} \label{pro3.1}
The complex
${\mathbb{P}}=(P_{m},d_{m})$:
$$
\xymatrix@C=2em{
\cdots \ar[r]& P_{m+1}\ar[r]^{d_{m+1}} & P_{m}\ar[r]&\cdots\ar[r]
& P_{2}\ar[r]^{d_{2}}& P_{1}\ar[r]^{d_{1}}& P_{0}\ar[r]^{d_{0}}
& A_{\q}\ar[r]& 0  }
$$
is a minimal projective bimodule resolution of $A_{\q}$, where
$d_{0}$ is the multiplication map.
\end{proposition}

\begin{proof} Now we consider the minimal projective
bimodule resolution of $A_{\q}$ constructed in \cite[section
9]{BK}. Let $X=\{\alpha_{1}, \alpha_{2}, \beta_{1}, \beta_{2}\}$.
Since $A_{\q}$ is a Koszul algebra for each $\q$, we only need to prove that
$F^{m}$ is a $\kk$-basis of the $\kk$-vector space $K_{m} :=
\bigcap\limits_{s+r=m-2}X^{s}RX^{r}$.

Note that $XK_{m-1}\cap K_{m-1}X\subset K_{m}$, for all $m$, and for
all $0\leq j\leq m,\, i=1, 2$, we have $f^{m}_{(i,j)}\in K_{m}$
by induction on $m$. Denote
by $I^{\perp}$ the ideal of $\kk(Q^{op})$ generated by
$$
R^{\perp}=\left\{\q\beta_{1}^{op}\alpha_{1}^{op}-
\alpha_{2}^{op}\beta_{2}^{op},\ \ \q\beta_{2}^{op}\alpha_{2}^{op}-
\alpha_{1}^{op}\beta_{1}^{op}\right\}.
$$
Then, the algebra $\kk(Q)^{op}/I^{\perp}$ is isomorphic to the Yoneda
algebra $E(A_{\q})$ of $A_{\q}$, since $A_{\q}$ is
Koszul (cf. \cite[Theorem 2.10.1]{BGS}). Therefore, the Betti
number of a minimal projective resolution of $A_{\q}$ over
$A_{\q}^{e}$ is $\{b_{m}=2(m+1)\}_{m\geq0}$. Thus $\dim_{\kk}K_{m}=2(m+1)$ for all $m\geq0$.
Note that the elements in $F^{m}$ is $\kk$-linearly independent, we get
the elements in $F^{m}$ is a $\kk$-basis of $K_{m}$.

Finally, by \cite[section 9]{BK} and \cite{GHM}, we get the
differential $d_{m}$ is given as above.
\end{proof}

Let $X$ and $Y$ be two sets of uniform elements in $\kk Q$. Then one can define
$$
X\odot Y=\{(p, q)\in
X\times Y \mid{\mathfrak{t}}(p)={\mathfrak{o}}(q) \mbox{ and }
{\mathfrak{t}}(q)={\mathfrak{o}}(p)\},
$$
and denote by
$\kk(X\odot Y)$ the vector space spanned by the
elements in $X\odot Y$. A pair of uniform elements
$(p,q)$ in $\kk Q$ is called closed if
$(p,q)\in\kk Q\odot \kk Q$.

Consider the set $\B\odot F^{m}$, then we have
$$
\B \odot F^{m}=\left\{\begin{array}{ll}
\left\{(\alpha_{i}, f^{m}_{(i+1,j)}),\; (\beta_{i}, f^{m}_{(i,j)})\mid
i=1, 2,\; 0\leq j\leq m\right\},  & \mbox{if }m\mbox{ is odd}; \\
\left\{(e_{i}, f^{m}_{(i,j)}),\; (\alpha_{i}\beta_{i}, f^{m}_{(i,j)})\mid
i=1, 2,\; 0\leq j\leq m\right\},
& \mbox{if }m\mbox{ is even}.
\end{array}\right.
$$
That is $|\B\odot F^{m}|=4(m+1)$.

Applying functor $A_{\q}\otimes_{A_{\q}^{e}}-$ to the
minimal projective bimodule resolution ${\mathbb{P}}=(P_{n},d_{n})$,
we get a Hochschild homology complex of the algebra $A_{\q}$.
Now, we use vector spaces $\kk(\B\odot F^{n})$ to give a presentation
of this Hochschild homology complex.

\begin{lemma} \label{lem3.2}
As complexes, $A_{\q}\otimes_{A_{\q}^{e}}{\mathbb{P}}\cong{\mathbb{N}}$,
where the complex ${\mathbb{N}}=(N_{m}, \tau_{m})$, $N_{m}=\kk(\B\odot F^{m})$
and differential $\tau_{m}: N_{m}\rightarrow N_{m-1}$ is given by:
for any $(b, f^{m}_{(i,j)})$ in $\kk(\B\odot F^{m})$,
$$
\begin{aligned}
\tau_{m}(b, f^{m}_{(i,j)})=(b\alpha_{i}, &f^{m-1}_{(i+1,j-1)})
+(-1)^{m}\q^{m-j}(\alpha_{i-m-1}b, f^{m-1}_{(i,j-1)})\\
&+\q^{j}(b\beta_{i-1}, f^{m-1}_{(i-1,j)})
+(-1)^{m}(\beta_{i-m}b, f^{m-1}_{(i,j)}).
\end{aligned}
$$
\end{lemma}

\begin{proof}
Let $E$ be the maximal semisimple
subalgebra of $A_{\q}$. Then one can check that
$$
A_{\q}\otimes_{A_{\q}^e}P_{m}= A_{\q}\otimes_{E^{e}}
\bigoplus_{f\in F^{n}}(\ok(f)\otimes\tk(f))
\cong\bigoplus_{\alpha,\beta\in\{e_{1}, e_{2}\}}
\alpha A_{\q}\beta\otimes\beta F^{m}\alpha.
$$
Thus $A_{\q}\otimes_{A_{\q}^e}P_{m}\cong N_{m}$ as $\kk$-vector spaces.
Moreover, from the isomorphisms above, we have the commutative diagram
$$
\xymatrix{
\cdots\ar[r]&A_{\q}\otimes_{A_{\q}^{e}}P_{m}\ar[r]^{1\otimes d_{m}}
\ar[d]^{\varphi^{m}} &A_{\q}\otimes_{A_{\q}^{e}}P_{m-1}\ar[r]
\ar[d]^{\varphi^{m+1}} &\cdots\\
\cdots\ar[r]&\kk(\B\odot F^{m})\ar[r]^{\tau_{m}}
&\kk(\B\odot F^{m-1})\ar[r]&\cdots .}
$$
So differential $\tau_{m}$ can be induced by $d_{m}$ in the
minimal projective resolution ${\mathbb{P}}$.
\end{proof}

By the definition of the Hochschild homology, $HH_{m}(A_{\q})
=\Ker\tau_{m}/ \Imm\tau_{m+1}$ and the isomorphism above, we have
$$
\begin{aligned}
\dim_{\kk}HH_{m}(A_{\q})&=\dim_{\kk}\Ker\tau_{m}
-\dim_{\kk}\Imm\tau_{m+1}\nonumber\\
&=\dim_{\kk}N_{m}-\dim_{\kk}\Imm\tau_{m}-\dim_{\kk}\Imm\tau_{m+1}.
\end{aligned}
$$
Consequently, to calculate the dimensions of Hochschild homology
groups of $A_{\q}$, we only need to determine $\dim_{\kk}\Imm\tau_{m}$
for all $m\geq1$, since $\dim_{\kk}N_{m}=|\B\odot F^{m}|$.
For $m=1, 2$, by the descriptions of the differentials $\tau_{m}$ in
Lemma \ref{lem3.2}, direct calculation shows that
$$
\dim_{\kk}\Imm\tau_{1}=\left\{\begin{array}{ll}
1, &\quad \mbox{if } \q=\pm1; \\
2, &\quad \mbox{otherwise},
\end{array}\right.
\quad\mbox{ and }\quad
\dim_{\kk}\Imm\tau_{2}=\left\{\begin{array}{ll}
3, &\quad \mbox{if } \q=\pm1; \\
4, &\quad \mbox{otherwise}.
\end{array}\right.
$$
Therefore, we have
$$
\dim_{\kk}HH_{0}(A_{\q})=\left\{\begin{array}{ll}
3, &\quad \mbox{if } \q=\pm1; \\
2, &\quad \mbox{otherwise},
\end{array}\right.
\quad\mbox{ and }\quad
\dim_{\kk}HH_{1}(A_{\q})=\left\{\begin{array}{ll}
4, &\quad \mbox{if } \q=\pm1; \\
2, &\quad \mbox{otherwise}.
\end{array}\right.
$$
For $\Imm\tau_{m}$, $m\geq2$, we first express it as a matrix, and make a detailed
analysis of it. Define an order on $\B$ by setting $e_{1}\prec e_{2}\prec\alpha_{1}\prec
\alpha_{2}\prec\beta_{1}\prec\beta_{2}\prec\alpha_{1}\beta_{1}\prec\alpha_{2}\beta_{2}$,
and denote $\B=\B_{1}\cup\B_{2}\cup\B_{3}\cup\B_{4}$, where $\B_{1}=\{e_{1}, e_{2}\}$,
$\B_{2}=\{\alpha_{1}, \alpha_{2}\}$, $\B_{3}=\{\beta_{1}, \beta_{2}\}$ and
$\B_{1}=\{\alpha_{1}\beta_{1}, \alpha_{2}\beta_{2}\}$.
Then we can define an order on $\B\odot F^{m}$ by
$$
\begin{aligned}
(b, f^{m}_{(i,j)})\prec (b', f^{m}_{(i',j')})\quad \mbox{ if }
l< k, &\mbox{ or }l=k \mbox{ but }j<j',\\
&\mbox{ or }l=k, j=j' \mbox{ but }b<b',
\end{aligned}
$$
for any $(b,f^{m}_{(i,j)})$, $(b',f^{m}_{(i',j')})\in\B\odot F^{m}$ with
$b\in\B_{l}$ and $b'\in\B_{k}$.
We still denote by $\tau_{m}$ the matrix of the differentials $\tau_{m}$
under the ordered bases above. Through the detailed analysis of
matrix $\tau_{m}$, the dimension of each degree homology groups
of algebra $A_{\q}$ will be given explicitly.

Denote by $A_{i}$ and $B_{i}$ the $2\times 2$ matrices
${\small\left(\begin{array}{cc}
\q^{i}&\q^{-1}\\   \q^{-1}&\q^{i}
\end{array}\right)}$,
${\small\left(\begin{array}{cccccccc}
1&\q^{i}\\    \q^{i}&1
\end{array}\right)}$
respectively. For any positive integer $l, k$ and $l\times k$
matrix $A$, we denote $\bar{A}$ the $2l\times k$ matrix
${\small\left(\begin{array}{cc}
A\\ 0    \end{array}\right)}$,
denote $\widehat{A}$ the $2l\times k$ matrix
${\small\left(\begin{array}{cc}
0\\ A    \end{array}\right)}$.
Then, following from the descriptions of the
differentials $\tau_{m}$ in Lemma \ref{lem3.2}, we obtain

\begin{itemize}
\item[(1)] if $m$ is odd,
$\tau_{m}=\left(\begin{array}{cc}
0&C_{m}\\     0&D_{m}
\end{array}\right)_{4(m+1)\times 4m}$, where
$C_{m}={\rm diag}\left\{A_{0}, A_{1}, \cdots, A_{m-2},
\bar{A}_{m-1}\right\}$ and
$D_{m}={\rm diag}\left\{-\widehat{A}_{m-1}, -A_{m-2},
\cdots,-A_{1}, -A_{0}\right\}$ are $2(m+1)\times 2m$ matrices;

\item[(2)] if $m$ is even,
$\tau_{m}=\left(\begin{array}{cc}G_{m}&H_{m}\\  0&0
\end{array}\right)_{4(m+1)\times 4m}$, where
$G_{m}={\rm diag}\left\{\widehat{B}_{m-1}, B_{m-2}, \cdots, B_{1},
B_{0}\right\}$ and
$H_{m}={\rm diag}\left\{B_{0}, B_{1}, \cdots, B_{m-2},
\bar{B}_{m-1}\right\}$ are $2(m+1)\times 2m$ matrices.
\end{itemize}

Therefore, we get the following results.

\begin{lemma} \label{lem3.3}
For the differential $\tau_{m}$, $m\geq2$, we have
\begin{itemize}
\item[(1)] if $\q$ is not a root of unity, $\rank\tau_{m}=2m$;

\item[(2)] if $\q=\pm 1$, $\rank\tau_{m}=\left\{\begin{array}{ll}
m, &\quad \mbox{ if }m\mbox{ is odd}; \\
m+1, &\quad \mbox{ if }m\mbox{ is even};
\end{array}\right.$

\item[(3)] if $\q$ is a primitive $s$-th $(s>2)$ root of unity,
$$
\rank\tau_{m}=\left\{\begin{array}{ll}
2m-2l+1, &\quad \mbox{if }s\mbox{ is odd and }m=2ls-1; \\[-0.3em]
 & \qquad\mbox{or }s\mbox{ is even and }m=ls-1;\\
2m, &\quad \mbox{otherwise}.
\end{array}\right.
$$
\end{itemize}
\end{lemma}

\begin{proof}
If $m$ is odd, then
$$
\rank\tau_{m}=\rank{\small\left(\begin{array}{cc}
A_{0}\\ A_{m-1} \end{array}\right)}
+\rank{\small\left(\begin{array}{cc}
A_{1}\\ A_{m-2} \end{array}\right)}+\cdots+
\rank{\small\left(\begin{array}{cc}
A_{m-1}\\ A_{0} \end{array}\right)}.
$$
Note that $\rank{\small\left(\begin{array}{cc}
A_{i}\\ A_{m-1-i} \end{array}\right)}=1$ if and only if
$\q=\pm 1$, or $\q$ is a primitive $s$-th $(s>2)$ root of unity,
$s|m+1$ and $s|2i+2$, we have
\begin{itemize}
\item[(1)] if $\q$ is not a root of unity, $\rank\tau_{m}=2m$;

\item[(2)] if $\q=\pm 1$, $\rank\tau_{m}=m$;

\item[(3)] if $\q$ is a primitive $s$-th $(s>2)$ root of unity,
$$
\rank\tau_{m}=\left\{\begin{array}{ll}
2m-2l+1, &\quad \mbox{if }s\mbox{ is odd and }m=2ls-1, \\[-0.3em]
 & \qquad\mbox{or }s\mbox{ is even and }m=ls-1;\\
2m, &\quad \mbox{otherwise}.
\end{array}\right.
$$
\end{itemize}

If $m$ is even, then
$$
\rank\tau_{m}=2\rank B_{0}
+\rank{\small\left(\begin{array}{cc}
B_{m-1}& B_{1} \end{array}\right)}
+\rank{\small\left(\begin{array}{cc}
B_{m-2}& B_{2} \end{array}\right)}+\cdots+
\rank{\small\left(\begin{array}{cc}
B_{1}& B_{m-1} \end{array}\right)}.
$$
Note that $\rank B_{0}=1$ and $\rank{\small\left(\begin{array}{cc}
B_{m-i}& B_{i} \end{array}\right)}=1$ if and only if
$\q=\pm 1$, or $\q$ is a primitive $s$-th $(s>2)$ root of unity,
$s|m$ and $s|2i$, we have
\begin{itemize}
\item[(1)] if $\q$ is not a root of unity, $\rank\tau_{m}=2m$;

\item[(2)] if $\q=\pm 1$, $\rank\tau_{m}=m+1$;

\item[(3)] if $\q$ is a primitive $s$-th $(s>2)$ root of unity,
$$
\rank\tau_{m}=\left\{\begin{array}{ll}
2m-2l+1, &\quad \mbox{if }s\mbox{ is odd and }m=2ls-1, \\[-0.3em]
 & \qquad\mbox{or }s\mbox{ is even and }m=ls-1;\\
2m, &\quad \mbox{otherwise}.
\end{array}\right.
$$
\end{itemize}
Thus, we obtain this lemma.
\end{proof}

Now we can give the main results of this section.

\begin{proposition} \label{pro3.4}
Let $A_{\q}$ be the quantum zigzag algebra $A_{\q}$ of type
$\widetilde{\mathbf{A}}_{1}$. Then for $m\geq2$, we have
\begin{itemize}
\item[(1)] if $\q$ is not a root of unity, $\dim_{\kk}HH_{m}(A_{\q})=2$;

\item[(2)] if $\q=\pm 1$, $\dim_{\kk}HH_{m}(A_{\q})=2m+2$;

\item[(3)] if $\q$ is a primitive $s$-th $(s>2)$ root of unity,
$$
\dim_{\kk}HH_{m}(A_{\q})=\left\{\begin{array}{ll}
2l+1, &\quad \mbox{if }s\mbox{ is odd, and }m=2ls-2\mbox{ or }m=2ls,  \\[-0.3em]
 & \qquad\mbox{or }s\mbox{ is even, and }m=ls-2\mbox{ or }m=ls;\\
4l, &\quad \mbox{ if }s\mbox{ is odd and }m=2ls-1, \\[-0.3em]
 & \qquad\mbox{or }s\mbox{ is even and }m=ls-1;\\
2, &\quad \mbox{otherwise}.
\end{array}\right.
$$
\end{itemize}
\end{proposition}

\begin{proof}
Thank to the formula
$$
\dim_{\kk}HH_{m}(A_{\q})
=\dim_{\kk}N_{m}-\dim_{\kk}\Imm\tau_{m}-\dim_{\kk}\Imm\tau_{m+1}
$$
and results given in Lemma \ref{lem3.3},
by direct calculation, we obtain the proposition.
\end{proof}

Denote by $HC_{m}(A_{\q})$ the $m$-th cyclic homology group of
$A_{\q}$.  Whenever \ch$\kk=0$, using the close relationship
between $HC_{m}(A_{\q})$ and $HH_{m}(A_{\q})$ given in \cite{Lo},
we can give the dimensions of cyclic homology groups of algebras $A_{\q}$.

\begin{corollary} \label{cor3.5}
Let $A_{\q}$ be the quantum zigzag algebras of type
$\widetilde{\mathbf{A}}_{1}$ and \ch$\kk=0$. Then we have
\begin{itemize}
\item[(1)] if $\q$ is not a root of unity, $\dim_{\kk}HC_{m}(A_{\q})=2$;

\item[(2)] if $\q=\pm 1$, $\dim_{\kk}HC_{m}(A_{\q})
=\left\{\begin{array}{ll}
m+2, &\quad \mbox{if }m\mbox{ is odd}; \\
m+3, &\quad \mbox{if }m\mbox{ is even};
\end{array}\right.$

\item[(3)] if $\q$ is a primitive $s$-th $(s>2)$ root of unity,
$$
\dim_{\kk}HC_{m}(A_{\q})=\left\{\begin{array}{llll}
4l-1, &\quad \mbox{if }s\mbox{ is odd and }m=2ls-2, \\[-0.3em]
 & \qquad\mbox{or }s\mbox{ is even and }m=ls-2;\\
3, &\quad \mbox{if }s\mbox{ is odd and }m=2ls-1, \\[-0.3em]
 & \qquad\mbox{or }s\mbox{ is even and }m=ls-1;\\
2l, &\quad \mbox{if }s\mbox{ is odd and }m=2ls, \\[-0.3em]
 & \qquad\mbox{or }s\mbox{ is even and }m=ls;\\
2, &\quad \mbox{otherwise}.
\end{array}\right.
$$
\end{itemize}
\end{corollary}

\begin{proof}
By \cite[Theorem 4.1.13]{Lo}, we have
$$\begin{aligned}
\dim_{\kk}HC_{m}(A_{\q})-\dim_{\kk}HC_{m}(\kk^2)=
-&(\dim_{\kk}HC_{m-1}(A_{\q})-\dim_{\kk}HC_{m-1}(\kk^2))\\
+&(\dim_{\kk}HH_{m}(A_{\q})-\dim_{\kk}HH_{m}(\kk^2)).
\end{aligned}
$$
Thus
$\dim_{\kk}HC_{m}(A_{\q})-\dim_{\kk}HC_{m}(\kk^{2})=
\sum\limits_{i=0}^{m}(-1)^{m-i}(\dim_{\kk}HH_{i}(A_{\q})
-\dim_{\kk}HH_{i}(\kk^{2}))$.
Moreover, it is well-known that
$$
\dim_{\kk}HH_{i}(\kk^2)=\left\{\begin{array}{ll}{2,} &\quad\mbox{if $i=0$}; \\
{0,}&\quad\mbox{if $i\geq 1$},\\
\end{array}\right.
\quad\mbox{ and }\quad
\dim_{\kk}HC_{i}(\kk^2)=\left\{\begin{array}{ll}{2,}
&\quad\mbox{if $i$ is even}; \\  {0,}&\quad\mbox{if $i$ is odd}.\\
\end{array}\right.
$$
Therefore, by Proposition \ref{pro3.4}, we obtain this corollary.
\end{proof}

For any finite-dimensional $\kk$-algebra $\Lambda$, we denote by
$$
hh.\dim\Lambda:={\rm inf}\{l\in{\mathbb{Z}}\mid \dim_{\kk}
HH_{m}(\Lambda)=0 \mbox{ for all }m>l\}
$$
and $gl.\dim\Lambda$ the Hochschild homology dimension
and global dimension of $\Lambda$, respectively. Then,
by the results of Proposition \ref{pro3.4}, we have

\begin{corollary} \label{cor3.4}
$gl.\dim A_{\q}=\infty=hh.\dim A_{\q}$
\end{corollary}

Dieter Happel in \cite{Ha} asked the following question: if the
Hochschild cohomology groups $HH^{n}(\Lambda)$ of a finite-dimensional
algebra $\Lambda$ over a field $k$ vanish for all
sufficiently large $n$, is the global dimension of $\Lambda$ finite?
The paper \cite{BGMS} have given a negative answer by a class of four dimensional
algebras $\Lambda_{4}:=\kk\langle x, y\rangle/(x^{2}, xy-\q yx, y^{2})$,
where $\q$ is not a root of unity.

In \cite{Han}, Han conjectured that the homology of Happel's
question would always hold, namely that a finite-dimensional algebra
whose higher Hochschild homology groups vanish must be of finite
global dimension. It is known that Han's conjecture holds for
many types of algebra. Our results show that the algebra $A_{\q}$
also provide a positive answer to Han's conjecture.

\bigskip
\section{Hochschild cohomology groups of $A_{\q}$}\label{coh}

In this section, using the language of parallel paths,
we give a equivalent description of the cohomology
complex obtained by the minimal projective bimodule resolution
$\mathbb{P}$. Further, we give a $\kk$-basis of Hochschild cohomology
groups of $A_{\q}$ on each degree.

Let $X$ and $Y$ be two sets of uniform elements in $\kk Q$, we define
$$
X /\!\!/ Y:=\left\{(p,q)\in X\times Y
\mid{\mathfrak{o}}(p)={\mathfrak{o}}(q) \mbox{ and
}{\mathfrak{t}}(p)={\mathfrak{t}}(q)\right\},
$$
and denote by $\kk(X /\!\!/ Y)$ the vector space spanned by the
elements in $X /\!\!/ Y$, and call a pair of uniform elements
$(p,q)$ in $\kk Q$ is parallel if $(p,q)\in\kk Q /\!\!/ \kk Q$.

Consider the sets $\B /\!\!/ F^{m}$, then we have
$$
\B /\!\!/ F^{m}=\left\{\begin{array}{ll}
\left\{(\alpha_{i}, f^{m}_{(i,j)}), (\beta_{i}, f^{m}_{(i+1,j)})\mid
i=1, 2,\; 0\leq j\leq m\right\},  & \mbox{if }m\mbox{ is odd}; \\
\left\{(e_{i}, f^{m}_{(i,j)}), (\alpha_{i}\beta_{i}, f^{m}_{(i,j)})\mid
i=1, 2,\; 0\leq j\leq m\right\},
& \mbox{if }m\mbox{ is even}.
\end{array}\right.
$$
Thus $|\B /\!\!/ F^{m}|=4(m+1)$.

We now define the complex ${\mathbb{L}}=(L^{m}, \sigma^{m})$ by the set
$\B /\!\!/ F^{m}$ as following: firstly, let
$$
L^{m}=\kk(\B /\!\!/ F^{m})
$$
for all $m\geq0$; secondly, define the differential
$\sigma^{m}: L^{m-1}\rightarrow L^{m}$ by
$$
\begin{aligned}
\sigma^{m}(b, f^{m-1}_{(i,j)})=(\alpha_{i-1}b, &f^{m}_{(i-1,j+1)})
+(-1)^{m}\q^{m-j-1}(b\alpha_{i-m+1}, f^{m}_{(i,j+1)}) \\
&+\q^{j}(\beta_{i}b, f^{m}_{(i+1,j)})
+(-1)^{m}(b\beta_{i-m}, f^{m}_{(i,j)}).
\end{aligned}
$$
Applying functor $\Hom_{A_{\q}^{e}}(-,\; A_{\q})$ to
the minimal projective bimodule resolution $\mathbb{P}$, we get
a Hochschild cohomology complex of $A_{\q}$. Next lemma,
we show that the complex ${\mathbb{L}}$ give a presentation
of this Hochschild cohomology complex.

\begin{lemma} \label{lem4.1}
$\Hom_{A_{\q}^{e}}({\mathbb{P}}, A_{\q})\cong{\mathbb{L}}$ as complexes.
\end{lemma}

\begin{proof} It is easy to see that
$$
\Hom_{A_{\q}^{e}}\left(P_{m},A_{\q}\right)\cong
\bigoplus_{f\in F^{m}}\Hom_{A_{\q}^{e}}\left(A_{\q}\ok(f)
\otimes\tk(f)A_{\q}, A_{\q}\right)\cong\bigoplus_{f\in F^{m}}
\ok(f)A_{\q}\tk(f)\cong L^{m}
$$
as $\kk$-vector spaces.
The corresponding isomorphism
$\varphi_{m}: L^{m}\rightarrow\Hom_{A_{\q}^{e}}\left(P_{m},
A_{\q}\right)$ is given by $(a, f)\mapsto\xi_{(a, f)}$,
where $\xi_{(a, f)}\left(\ok(g)\otimes\tk(g)\right)$ is $a$ if $f=g$
and is 0 otherwise.
Then we have the following commutative diagram:
$$
\xymatrix{
\cdots \ar[r]&\kk\left(\B/\!\!/ F^{m}\right)\ar[r]^{\sigma^{m+1}}
\ar[d]^{\varphi^{m}}
&\kk\left(\B/\!\!/ F^{m+1}\right)\ar[r]\ar[d]^{\varphi^{m+1}}&\cdots\\
\cdots \ar[r]&\Hom_{A_{\q}^{e}}\left(P_{m},
A_{\q}\right)\ar[r]^{d_{m+1}^{\ast}}
&\Hom_{A_{\q}^{e}}\left(P_{m+1}, A_{\q}\right)\ar[r]&\cdots .}
$$
Therefore, the isomorphism of complexes is obtained.
\end{proof}

We now  give a basis of Hochschild cohomology
groups of $A_{\q}$ by the complex $\mathbb{L}$. By the
definition, $HH^{m}(A_{\q})=\Ker\sigma^{m+1}/\Imm\sigma^{m}$,
we need to determine $\Imm\sigma^{m}$ and $\Ker\sigma^{m}$ for all $m\geq1$.
Now, we will determine $\Imm\sigma^{m}$ by considering the corresponding matrix
of $\sigma^{m}$ over an ordered basis of $L^{m-1}$.
Let
$$
(b,f^{m}_{(i,j)})\prec (b',f^{m}_{(i',j')})\quad \mbox{ if }
j<j', \mbox{ or }j=j'\mbox{ but }b\prec b',
$$
for any
$(b,f^{m}_{(i,j)}), (b',f^{m}_{(i',j')})\in\B/\!\!/ F^{m}$.
We still denote by $\sigma^{m}$ the matrix of $\sigma^{m}$ under the
ordered basis $\B/\!\!/ F^{m}$.
Then we have

$$
\sigma^{m}=\left\{\begin{array}{ll}
{\rm diag}\left\{\widehat{C}_{0},C_{1},C_{2},\cdots,C_{m-2},
\bar{C}_{m-1}\right\},  & \mbox{ if }m\mbox{ is odd}; \\
{\rm diag}\left\{\widehat{D}_{-1},D_{0},D_{1},\cdots,D_{m-2},
\bar{D}_{m-1}\right\},
& \mbox{ if }m\mbox{ is even},
\end{array}\right.
$$
where $4\times4$ matrices
$C_{i}={\small\left(\begin{array}{cc}   A_{i}&-A_{m-i-1}\\ 0&0
\end{array}\right)}$,
$D_{i}={\small\left(\begin{array}{cc}   0&B_{m-i-2}\\ 0&B_{i}
\end{array}\right)}$,
$2\times4$ matrices
$\widehat{D}_{-1}=\bar{D}_{m-1}=
{\small\left(\begin{array}{cc}    0&B_{-1}\end{array}\right)}$,
$4\times6$ matrices
$\widehat{C}_{0}={\small\left(\begin{array}{cc}    0&C_{0}
\end{array}\right)}$,
$\bar{C}_{m-1}={\small\left(\begin{array}{cc}   C_{m-1}&0
\end{array}\right)}$,
and $2\times2$ matrices
$A_{i}={\small\left(
\begin{array}{cc}
\q^{i}&-1\\   -1&\q^{i}
\end{array}\right)}$,
$B_{i}={\small\left(
\begin{array}{cccccccc}
1&-\q^{i}\\
-\q^{i}&1
\end{array}\right)}$.

Thus, direct computations show that
$$
\begin{aligned}
&HH^{0}(A_{\q})\cong\kk\left\{\begin{split}
&(\alpha_{1}\beta_{1}, f^{0}_{(1,0)})+\q^{-1}(\alpha_{2}\beta_{2}, f^{0}_{(2,0)}),\\
&(\alpha_{1}\beta_{1}, f^{0}_{(1,0)})-\q^{-1}(\alpha_{2}\beta_{2}, f^{0}_{(2,0)}),\\
&(e_{1}, f^{0}_{(1,0)})+(e_{2}, f^{0}_{(2,0)})
\end{split}\right\};\\
&HH^{1}(A_{\q})\cong\left\{\begin{array}{lll}
\kk\left\{\begin{split}
&(\alpha_{1}, f^{1}_{(1,0)})+\q^{-1}(\alpha_{2}, f^{1}_{(2,0)}),\\
&(\beta_{1}, f^{1}_{(2,0)})+(\beta_{2}, f^{1}_{(1,0)}),\\
&(\alpha_{1}, f^{1}_{(1,1)})+(\alpha_{2}, f^{1}_{(2,1)}),\\
&(\beta_{1}, f^{1}_{(2,1)})+\q^{-1}(\beta_{2}, f^{1}_{(1,1)})
\end{split}\right\},  & \mbox{if }\q=\pm1; \\
\kk\left\{\begin{split}
&(\alpha_{1}, f^{1}_{(1,1)})+(\alpha_{2}, f^{1}_{(2,1)}),\\
&(\beta_{1}, f^{1}_{(2,0)})+(\beta_{2}, f^{1}_{(1,0)})
\end{split}\right\},
& \mbox{otherwise}.
\end{array}\right.
\end{aligned}
$$
Here, we use the same notation for the corresponding cohomology classes.
For the higher degree Hochschild cohomology groups of $A_{\q}$,
We will discuss the value of $\q$ case-by-case.
Note that $1\leq\rank C_{i}\leq 2$, $1\leq\rank D_{i}\leq 2$,
$\rank C_{i}=1$ if and only if $\q=\pm1$, or $\q$ is a primitive
$s$-th $(s>2)$ root of unity and satisfying $s|m-1$ and $s|2i$; and $\rank D_{i}=1$
if and only if one of $\q=\pm1$, or $\q$ is a primitive $s$-th $(s>2)$
root of unity and satisfying $s|m-2$ and $s|2i$. Hence, we have the following
propositions.

\begin{proposition} \label{prop4.2}
If $\q$ is not a root of unity and $m\geq2$. Then
$$
HH^{m}(A_{\q})\cong\left\{\begin{array}{ll}
\kk\left\{(\alpha_{1}\beta_{1}, f^{2}_{(1,1)})
+(\alpha_{2}\beta_{2}, f^{2}_{(2,1)})\right\},
& \mbox{if }m=2; \\
0,   & \mbox{if }m>2.
\end{array}\right.
$$
\end{proposition}

\begin{proof}
Let us discuss the parity of m as follows. If $m>2$ is odd,
then $\rank\sigma^{m}=\rank C_{0}+\rank C_{1}+\cdots+\rank C_{m-1}$.
Note that $\rank C_{i}=2$ since $\q$ is not a root of unity, we get
$\rank\sigma^{m}=2m$. If $m>2$ is even, then $\rank\sigma^{m}=
\rank\widehat{D}_{-1}+\rank D_{0}+\cdots+\rank D_{m-2}
+\rank\bar{D}_{m-1}$.
Note that $\rank\widehat{D}_{-1}=\rank D_{i}=\rank\bar{D}_{m-1}=2$
since $\q$ is not a root of unity, we get
$\rank\sigma^{m}=2m+2$. By the definition,
$HH^{m}(A_{\q})=\Ker\sigma^{m+1}/\Imm\sigma^{m}$, and so that
$$
\begin{aligned}
\dim_{\kk}HH^{m}(A_{\q})&=\dim_{\kk}\Ker\sigma^{m+1}
-\dim_{\kk}\Imm\sigma^{m}\\
&=\dim_{\kk}L^{m}-\rank\sigma^{m+1}-\rank\sigma^{m}.
\end{aligned}
$$
Because $\dim_{\kk}L^{m}=4(m+1)$, we get $\dim_{\kk}HH^{m}(A_{\q})=0$
for $m>2$. Finally, when $m=2$, it is easy to see that
$\rank\sigma^{2}=5$, and $\dim_{\kk}HH^{2}(A_{\q})=1$.
By the definition of the differential $\sigma^{2}$, it is not hard
to get that $(\alpha_{1}\beta_{1}, f^{2}_{(1,1)})
+(\alpha_{2}\beta_{2}, f^{2}_{(2,1)})$ is basis of $HH^{2}(A_{\q})$.
\end{proof}

Notably, if $\q$ is not a root of unity, the algebra $A_{\q}$ also
give a negative answer for Happel's question.
By \cite{ZH}, $A_{\q}$ is just the $\mathbb{Z}_{2}$-Galois covering
of the four dimension algebra $\Lambda_{4}$ given in \cite{BGMS}, if
\ch$\kk\nmid 2$. Next, we consider the basis elements of
$HH^{m}(A_{\q})$ whenever $\q=\pm1$.

\begin{proposition} \label{prop4.3}
If $\q=\pm1$ and $m\geq2$. Then
$$
HH^{m}(A_{\q})\cong\left\{\begin{array}{lllllll}
\kk\left\{\begin{aligned}
&(\beta_{1}, f^{m}_{(2,j)})+\q^{j}(\beta_{2}, f^{m}_{(1,j)}),\\
&(\alpha_{1}, f^{m}_{(1,j)})+\q^{j-1}(\alpha_{2}, f^{m}_{(2,j)})
\end{aligned}\right\}_{0\leq j\leq m},
& \mbox{if }m\mbox{ is odd}; \\
\kk\left\{\begin{split}
&(\alpha_{1}\beta_{1}, f^{m}_{(1,j)})+\q^{j-1}(\alpha_{2}\beta_{2}, f^{m}_{(2,j)}), \\
&(e_{1}, f^{m}_{(1,j)})+\q^{j}(e_{2}, f^{m}_{(2,j)})
\end{split}\right\}_{0\leq j\leq m},
& \mbox{if }m\mbox{ is even}.
\end{array}\right.
$$
\end{proposition}

\begin{proof}
If $m\geq2$ is odd, then $\rank\sigma^{m}=m$, since $\rank C_{i}=1$
for all $0\leq i\leq m-1$. Consider the action of the
differential $\sigma^{m}$ on $L^{m-1}$, we have
$$
\begin{aligned}
\Imm\sigma^{m}&=\kk\left\{\q^{j}(\beta_{1}, f^{m}_{(2,j)})-
(\beta_{2}, f^{m}_{(1,j)})-\q^{j}(\alpha_{1}, f^{m}_{(1,j+1)})
+(\alpha_{2}, f^{m}_{(2,j+1)})\right\}_{0\leq j\leq m-1};\\
\Ker\sigma^{m}&=\kk\left\{(\alpha_{1}\beta_{1}, f^{m-1}_{(1,j)}),\ \
(e_{1}, f^{m-1}_{(1,j)})+\q^{j}(e_{2}, f^{m-1}_{(2,j)}),\ \
(\alpha_{2}\beta_{2}, f^{m-1}_{(2,j)})\right\}_{0\leq j\leq m-1}.
\end{aligned}
$$
If $m\geq2$ is even, then $\rank\sigma^{m}=m+1$,
Since $\rank\widehat{D}_{-1}=\rank D_{i}=\rank\bar{D}_{m-1}=1$
for all $0\leq i\leq m-2$. By the definition of the differential $\sigma^{m}$, we have
$$
\begin{aligned}
\Imm\sigma^{m}&=\kk
\left\{(\alpha_{1}\beta_{1}, f^{m}_{(1,j)})
-\q^{j-1}(\alpha_{2}\beta_{2}, f^{m}_{(2,j)})\right\}_{0\leq j\leq m};\\
\Ker\sigma^{m}&=\kk\left\{\begin{split}
&\left\{\begin{aligned}
&(\beta_{1}, f^{m-1}_{(2,j)})+\q^{j}(\beta_{2}, f^{m-1}_{(1,j)}),\\
&(\alpha_{1}, f^{m-1}_{(1,j)})
+\q^{j-1}(\alpha_{2}, f^{m-1}_{(2,j)})\\
\end{aligned}\right\}_{0\leq j\leq m-1},\\
&\left\{(\alpha_{1}, f^{m-1}_{(1,j+1)})-(\beta_{1}, f^{m-1}_{(2,j)})
\right\}_{0\leq j\leq m-2}
\end{split}\right\}.
\end{aligned}
$$
Finally, Using the definition of $HH^{m}(A_{\q})$, that is,
$HH^{m}(A_{\q})\cong\Ker\sigma^{m+1}/\Imm\sigma^{m}$, we can get
a basis of $HH^{m}(A_{\q})$.
\end{proof}

Similar to the discussion of Proposition \ref{prop4.3},
we can get a $\kk$-basis of $HH^{m}(A_{\q})$ when
$\q$ is a primitive $s$-th $(s>2)$ root of unity as following.

\begin{proposition} \label{prop4.4}
If $\q$ is a primitive $s$-th $(s>2)$ root of unity and $m\geq2$.

(1) Whenever $s$ is odd,
$$
HH^{m}(A_{\q})=\left\{\begin{array}{lllll}
\kk\left\{(e_{1}, f^{m}_{(1,j)})+(e_{2}, f^{m}_{(2,j)})
\right\}_{j\in\{ts\}_{t=0}^{2l}},
& \mbox{if }m=2ls; \\
\kk\left\{\begin{split}
&\{(\beta_{1}, f^{m}_{(2,j)})+(\beta_{2}, f^{m}_{(1,j)})\}_{j\in\{ts\}_{t=0}^{2l}},\\
&\{(\alpha_{1}, f^{m}_{(1,j)})+(\alpha_{2}, f^{m}_{(2,j)})\}_{j\in\{ts+1\}_{t=0}^{2l}}
\end{split}\right\},
& \mbox{if }m=2ls+1;\\
\kk\left\{(\alpha_{1}\beta_{1}, f^{m}_{(1,j)})
+(\alpha_{2}\beta_{2}, f^{m}_{(2,j)})
\right\}_{j\in\{ts+1\}_{t=0}^{2l}},
& \mbox{if }m=2ls+2;\\
0, & \mbox{otherwise}.
\end{array}\right.
$$

(2) Whenever $s$ is even,
$$
HH^{m}(A_{\q})=\left\{\begin{array}{lllll}
\kk\left\{(e_{1}, f^{m}_{(1,j)})+\q^{j}(e_{2}, f^{m}_{(2,j)})
\right\}_{j\in\{\frac{ts}{2}\}_{t=0}^{2l}},
& \mbox{if }m=ls; \\
\kk\left\{\begin{split}
&\{(\beta_{1}, f^{m}_{(2,j)})+(\beta_{2},
f^{m}_{(1,j)})\}_{j\in\{\frac{ts}{2}\}_{t=0}^{2l}},\\
&\{(\alpha_{1}, f^{m}_{(1,j)})+(\alpha_{2},
f^{m}_{(2,j)})\}_{j\in\{\frac{ts}{2}+1\}_{t=0}^{2l}}
\end{split}\right\}, & \mbox{if }m=ls+1;\\
\kk\left\{(\alpha_{1}\beta_{1}, f^{m}_{(1,j)})
+(\alpha_{2}\beta_{2}, f^{m}_{(2,j)})
\right\}_{j\in\{\frac{ts}{2}+1\}_{t=0}^{2l}},
& \mbox{if }m=ls+2;\\
0, & \mbox{otherwise}.
\end{array}\right.
$$
\end{proposition}

\begin{proof}
(1) In this case, $s$ is odd. If $m=2ls$, we get $\rank\sigma^{m}=2m+2$,
and $\rank\widehat{D}_{-1}=\rank D_{i}=\rank\overline{D}_{m-1}=2$
for all $0\leq i\leq m-2$. By the definition of the differential $\sigma^{m}$, we have
$$
\begin{aligned}
\Imm\sigma^{m}&=\kk\left\{(\alpha_{1}\beta_{1}, f^{m}_{(1,j)}),\, (\alpha_{2}\beta_{2},
f^{m}_{(2,j)})\right\}_{0\leq j\leq m};\\
\Ker\sigma^{m}&=\kk\left\{\begin{split}
&\q^{j}(\beta_{1},f^{m}_{(2,j)})+(\alpha_{2},f^{m}_{(2,j+1)}),\\
&(\beta_{2},f^{m}_{(1,j)})+q^{m-1-j}(\alpha_{1},f^{m}_{(1,j+1)})\}\\
\end{split}\right\}_{ 0\leq j\leq m-2}.
\end{aligned}
$$
If $m=2ls+1$, then $\rank\sigma^{m}=2m-2l-1$, and $\rank C_{i}=2$, $\rank C_{ts}=1$
for all $0\leq i\leq m-1$, $0\leq t\leq 2l$ and $i\neq ts$.
Thus we have
$$
\begin{aligned}
\Imm\sigma^{m}&=\kk\left\{\begin{split}
&\left\{\q^{j}(\beta_{1}, f^{m}_{(2,j)})-(\beta_{2},f^{m}_{(1,j)})
-\q^{m-1-j}(\alpha_{1}, f^{m}_{(1,j+1)})+(\alpha_{2},f^{m}_{(2,j+1)})
\right\}_{0\leq j\leq m-1,\, j\neq ts},\\
&\left\{-(\beta_{1},f^{m}_{(2,j)})+\q^{j}(\beta_{2}, f^{m}_{(1,j)})
+(\alpha_{1}, f^{m}_{(1,j+1)})-\q^{m-1-j}(\alpha_{2}, f^{m}_{(2,j+1)})
\right\}_{0\leq j\leq m-1,\,j\neq ts},\\
&\left\{(\beta_{1}, f^{m}_{(2,j)})-(\beta_{2}, f^{m}_{(1,j)})
-(\alpha_{1}, f^{m}_{(1,j+1)})+(\alpha_{2}, f^{m}_{(2,j+1)})
\right\}_{\{j= ts\}_{t=0}^{2l}}\end{split}\right\};\\
\Ker\sigma^{m}&=\kk\left\{\begin{split}
&\left\{(\alpha_{1}\beta_{1}, f^{m-1}_{(1,j)}),\ \
(\alpha_{2}\beta_{2}, f^{m-1}_{(2,j)})\right\}_{0\leq j\leq m-1},\\
&\left\{(e_{1}, f^{m-1}_{(1,j)})+(e_{2}, f^{m-1}_{(2,j)})\right\}_{\{j=ts\}^{2l}_{t=0}}
\end{split}\right\}.
\end{aligned}
$$
If $m=2ls+2$, then $\rank\sigma^{m}=2m+2-2l-1$,
and $\rank\widehat{D}_{-1}=\rank D_{i}=\rank\overline{D}_{m-1}=2$, $\rank D_{ts}=1$
for all $0\leq i\leq m-2$, $0\leq t\leq 2l$ and $i\neq ts$. Hence, we have
$$
\begin{aligned}
\Imm\sigma^{m}&=\kk\left\{\begin{split}
&\left\{(\alpha_{1}\beta_{1},f^{m}_{(1,j)})-(\alpha_{2}\beta_{2},f^{m}_{(2,j)})
\right\}_{\{j=ts+1\}_{t=0}^{2l}},\\
&\left\{(\alpha_{1}\beta_{1},f^{m}_{(1,j)}),\ \ (\alpha_{2}\beta_{2}, f^{m}_{(2,j)})
\right\}_{\{ 0\leq j\leq m,\, j\neq ts\}}\end{split}\right\};\\
\Ker\sigma^{m}&=\kk\left\{\begin{split}
&\left\{(\alpha_{1},f^{m-1}_{(1,j)})+(\alpha_{2},f^{m-1}_{(2,j)})
\right\}_{\{j=ts+1\}^{2l}_{t=0}},\\
&\left\{(\beta_{1},f^{m-1}_{(2,j)})+(\beta_{2},f^{m-1}_{(1,j)})
\right\}_{\{j=ts\}^{2l}_{t=0}},\\
&\left\{(\beta_{1},f^{m-1}_{(2,j)})-(\alpha_{1},f^{m-1}_{(1,j+1)})
\right\}_{\{j=ts\}^{2l}_{t=0}},\\
&\left\{\begin{split}
&(\beta_{2},f^{m}_{(1,j)})-\q^{m-1-j}(\alpha_{1},f^{m}_{(1,j+1)}),\\
&\q^{j}(\beta_{1},f^{m}_{(2,j)})+(\alpha_{2},f^{m-1}_{(2,j+1)})
\end{split}\right\}_{ 0\leq j\leq m-2,\,j\neq ts}\end{split}\right\}.
\end{aligned}
$$
If $m=2ls+3$ , then $\rank\sigma^{m}=2m$, and $\rank C_{i}=2$
for all $0\leq i\leq m-1$. Thus, we have
$$
\begin{aligned}
\Imm\sigma^{m}&=\kk\left\{\begin{aligned}
&\left\{\q^{j}(\beta_{1}, f^{m}_{(2,j)})-(\beta_{2}, f^{m}_{(1,j)})
-q^{m-1-j}(\alpha_{1}, f^{m}_{(1,j+1)})+(\alpha_{2}, f^{m}_{(2,j+1)})
\right\}_{0\leq j\leq m-1},\\
&\left\{-(\beta_{1}, f^{m}_{(2,j)})+\q^{j}(\beta_{2}, f^{m}_{(1,j)})
+(\alpha_{1}, f^{m}_{(1,j+1)})-\q^{m-1-j}(\alpha_{2}, f^{m}_{(2,j+1)})
\right\}_{0\leq j\leq m-1}\end{aligned}\right\},\\
\Ker\sigma^{m}&=\kk\left\{(\alpha_{1}\beta_{1}, f^{m-1}_{(1,j)}),\ \
(\alpha_{2}\beta_{2}, f^{m-1}_{(2,j)})\right\}_{0\leq j\leq m-1}.
\end{aligned}
$$
Note that $HH^{m}(A_{\q})\cong\Ker\sigma^{m+1}/\Imm\sigma^{m}$, we obtain
a basis of $HH^{m}(A_{\q})$ in the proposition. The proof of (2) is similar,
so we won't repeat it here.
\end{proof}

At the end of this section, by comparing the dimensions of
Hochschild homology groups and Hochschild cohomology groups
of algebra $A_{\q}$, we can get that if $\q=\pm1$, then
$$
\dim_{\kk}HH^{m}(A_{\q})=\dim_{\kk}HH_{m}(A_{\q})=2m+2,
$$
for any $m\geq0$.

\bigskip
\section{Hochschild cohomology ring of $A_{\q}$}\label{ring}

In this section, the cup product of the cohomology ring
$HH^{\ast}(A_{\q})$ is described by the parallel paths, and
so that the ring structure of $HH^{\ast}(A_{\q})$
and $HH^{\ast}(A_{\q})/{\mathcal{N}}$ are given explicitly.

Recall that, for an arbitrary finite-dimensional $\kk$-algebra
$\Lambda$, the Hochschild cohomology ring $HH^{\ast}(\Lambda)$ is
defined to be
$HH^{\ast}(\Lambda):=\bigoplus_{m\geq0}HH^{m}(\Lambda)$, whose
multiplication is given by the multiplication induced by the Yoneda
product. It is well known that the Yoneda product of
$HH^{\ast}(\Lambda)$ coincides with the cup product defined on the
cohomology of $\Hom_{\Lambda^{e}}({\mathbb{B}}, \Lambda)$, where
${\mathbb{B}}$ is the standard projective $\Lambda^{e}$-resolution
of $\Lambda$. Gerstenhaber showed that $HH^{\ast}(\Lambda)$ under
the cup product
$$
\sqcup: \quad HH^{m}(\Lambda)\times HH^{l}(\Lambda)\longrightarrow
HH^{m+l}(\Lambda)
$$
is graded commutative (see \cite{Ger}).

In \cite{SW}, Siegel and Witherspoon proved that any projective
$\Lambda^{e}$-resolution $\mathbb{X}$ of $\Lambda$ gives rise to the
cup product. They showed that there exists a chain map
$\bigtriangleup:{\mathbb{X}}\rightarrow{\mathbb{X}}
\otimes_{\Lambda}{\mathbb{X}}$
lifting the identity, which is unique up to homotopy, and the cup
product of two elements $\eta$ in $HH^{m}(\Lambda)$ and $\xi$ in
$HH^{n}(\Lambda)$ can be defined by the composition of the maps
$$
{\mathbb{X}}\xrightarrow[]{\ \bigtriangleup\
}{\mathbb{X}}\otimes_{\Lambda}{\mathbb{X}} \xrightarrow[]{\
\eta\otimes\xi \ }\Lambda\otimes_{\Lambda}\Lambda \xrightarrow[]{\
\nu\ }\Lambda,
$$
where $\nu$ is the natural isomorphism.

Here, we will use the minimal projective bimodule resolution
$\mathbb{P}=(P_{m},d_{m})$ of $A_{\q}$ which is constructed in
Section \ref{homology}, to give the cup product of $HH^{\ast}(A_{\q})$.
First recall that the tensor complex
${\mathbb{P}}\otimes_{A_{\q}}{\mathbb{P}}:=({\mathscr{P}}_{m},
b_{m})$ is given by
$$
{\mathscr{P}}_{m}:=\bigoplus_{i+j=m}
P_{i}\otimes_{A_{\q}}P_{j},
$$
and the differential $b_{m}: {\mathscr{P}}_{m}\rightarrow
{\mathscr{P}}_{m-1}$ is given by
$$
b_{m}=\sum_{i=0}^{m-1}\left((-1)^{i}{\rm id}\otimes
d_{m-i}+d_{i+1}\otimes {\rm id}\right),
$$
for all $m\geq1$. It is well known that
${\mathbb{P}}\otimes_{A_{\q}}{\mathbb{P}}$ is also a
projective bimodule resolution of
$A_{\q}\cong A_{\q}\otimes_{A_{\q}}A_{\q}$.

Now we define a family of $(A_{\q})^{e}$-morphisms
$\{\bigtriangleup_{m}: P_{m}\rightarrow {\mathscr{P}}_{m}\}_{m\geq0}$ as
follows:
$$
\begin{aligned}
&\bigtriangleup_{m}\left(\ok(f^{m}_{(i,j)})\otimes
\tk(f^{m}_{(i,j)})\right)\\
&=\sum_{s=0}^{m}\sum_{j'=0}^{N}\q^{(s-j')(j-j')}
\left(\ok(f^{s}_{(i,j')})\otimes\tk(f^{s}_{(i,j')})\right)
\widehat{\otimes}\left(\ok(f^{m-s}_{(i-s,j-j')})\otimes
\tk(f^{m-s}_{(i-s,j-j')})\right),
\end{aligned}
$$
where $N=\min\{s, j\}$, and
$\widehat{\otimes}:=\otimes_{A_{\q}}$.

\begin{lemma} \label{lem5.1}
The morphism
$\bigtriangleup:=(\bigtriangleup_{m})_{m\geq0}$ satisfies the
following commutative diagram
$$
\xymatrix{
\cdots \ar[r]& P_{m}\ar[r]^{d_{m}}\ar[d]^{\bigtriangleup_{m}}
& P_{m-1}\ar[r]\ar[d]^{\bigtriangleup_{m-1}} & \cdots\ar[r]
& P_{1}\ar[r]^{d_{1}}\ar[d]^{\bigtriangleup_{1}} &
P_{0}\ar[r]^{d_{0}}\ar[d]^{\bigtriangleup_{0}}
& A_{\q}\ar[r]\ar@{=}[d] & 0\; \\
\cdots \ar[r]& {\mathscr{P}}_{m}\ar[r]^{b_{m}}&
{\mathscr{P}}_{m-1}\ar[r]
& \cdots\ar[r]& {\mathscr{P}}_{1}\ar[r]^{b_{1}} &
{\mathscr{P}}_{0}\ar[r]^{b_{0}}
& A_{\q}\ar[r] & 0,}
$$
where $b_{0}=\nu\circ(d_{0}\otimes d_{0})$, $d_{0}$ is the multiplication map,
$\nu: A_{\q}\otimes_{A_{\q}}A_{\q}\rightarrow A_{\q}$
is the natural isomorphism.
\end{lemma}

\begin{proof} Firstly, it is easy to see that
$d_{0}=b_{0}\circ\bigtriangleup_{0}$. Secondly, for $n=1$, we have
$\bigtriangleup_{0}\circ d_{1}=b_{1}\circ\bigtriangleup_{1}$.
Indeed, for each $\alpha_{i}$,
$$
\begin{aligned}
b_{1}\circ\bigtriangleup_{1}\left(\ok(\alpha_{i})\otimes\tk(\alpha_{i})\right)
&=b_{1}\big((e_{i}\otimes e_{i})\widehat{\otimes}(e_{i}\otimes e_{i+1})
+(e_{i}\otimes e_{i+1})\widehat{\otimes}(e_{i+1}\otimes e_{i+1})\big)\\
&=(e_{i}\otimes e_{i})\widehat{\otimes}(\alpha_{i}\otimes e_{i+1})
-(e_{i}\otimes e_{i})\widehat{\otimes}(e_{i}\otimes\alpha_{i})\\
&\qquad +(\alpha_{i}\otimes e_{i+1})\widehat{\otimes}(e_{i+1}\otimes e_{i+1})
-(e_{i}\otimes\alpha_{i})\widehat{\otimes}(e_{i+1}\otimes e_{i+1})\\
&=\bigtriangleup_{0}\left(\alpha_{i}\otimes\tk(\alpha_{i})
-\ok(\alpha_{i})\otimes \alpha_{i}\right)\\
&=\bigtriangleup_{0}\circ d_{1}\left(\ok(\alpha_{i})\otimes\tk(\alpha_{i})\right).
\end{aligned}
$$
Similarly, one can check that $\bigtriangleup_{0}\circ d_{1}
\left(\ok(\beta_{i})\otimes\tk(\beta_{i})\right)=b_{1}\circ\bigtriangleup_{1}
\left(\ok(\beta_{i})\otimes\tk(\beta_{i})\right)$ for all $\beta_{i}\in F^{1}$.
Thus $\bigtriangleup_{0}\circ d_{1}=b_{1}\circ\bigtriangleup_{1}$.

Finally, let $m\geq2$. For any $a\in F^{m}$, $0\leq s\leq m$, we
denote by $\bigtriangleup_{m-1}\circ d_{m}\left(\ok(a)\otimes\tk(a)\right)_{s}$,
$b_{m}\circ \bigtriangleup_{m}\left(\ok(a)\otimes\tk(a)\right)_{s}$ the $s$-th
direct summand of $\bigtriangleup_{m-1}\circ d_{m}\left(\ok(a)\otimes\tk(a)\right)$
and $b_{m}\circ \bigtriangleup_{m}\left(\ok(a)\otimes\tk(a)\right)$ respectively,
that is, $\bigtriangleup_{m-1}\circ d_{m}\left(\ok(a)\otimes\tk(a)\right)_{s}\subseteq
P_{s}\widehat{\otimes}P_{m-s-1}$ and $b_{m}\circ
\bigtriangleup_{m}\left(\ok(a)\otimes\tk(a)\right)_{s}\subseteq
P_{s}\widehat{\otimes}P_{m-s-1}$. Then we have
$$
\begin{aligned}
&b_{m}\circ\bigtriangleup_{m}\left(\ok(f^{m}_{(i,j)})\otimes
\tk(f^{m}_{(i,j)})\right)_{s}\\
=&\sum_{j'=0}^{N}\q^{(s-j')(j-j'-1)}\left(\alpha_{i}\otimes
\tk(f^{s}_{(i+1,j')})\right)\widehat{\otimes}
\left(\ok(f^{m-s-1}_{(i-s+1,j-j'-1)})\otimes
\tk(f^{m-s-1}_{(i-s+1,j-j'-1)})\right)\\
&+\sum_{j'=0}^{N}\q^{(s-j')(j-j')+j}\left(\beta_{i+1}\otimes
\tk(f^{s}_{(i-1,j')})\right)\widehat{\otimes}
\left(\ok(f^{m-s-1}_{(i-s-1,j-j')})\otimes
\tk(f^{m-s-1}_{(i-s-1,j-j')})\right)\\
&+\sum_{j'=0}^{N}(-1)^{m}\q^{(s-j')(j-j'-1)+m-j}\left(\ok(f^{s}_{(i,j')})
\otimes\tk(f^{s}_{(i,j')})\right)\widehat{\otimes}
\left(\ok(f^{m-s-1}_{(i-s,j-j'-1)})\otimes\alpha_{i+m-1}\right)\\
&+\sum_{j'=0}^{N}(-1)^{m}\q^{(s-j')(j-j')}\left(\ok(f^{s}_{(i,j')})\otimes
\tk(f^{s}_{(i,j')})\right)\widehat{\otimes}\left(\ok(f^{m-s-1}_{(i-s,j-j')})
\otimes\beta_{i-m}\right)\\
=&\bigtriangleup_{m-1}\circ d_{m}\left(\ok(f^{m}_{(i,j)})
\otimes\tk(f^{m}_{(i,j)})\right)_{s},
\end{aligned}
$$
where $N:=\min\{j,s\}$. Therefore, we obtain the commutative diagram.
\end{proof}

Now, for any $m\geq0$ and $\eta:=(a, f)\in L^{m}=\kk(\B /\!\!/
F^{m})$, we identify it with its image $\varphi_{m}(\eta)$ under the
isomorphism $\varphi_{m}: L^{m}\rightarrow\Hom_{A_{\q}^{e}}
(P_{m}, A_{\q})$ which is given in Section \ref{coh}. By the morphism
$\bigtriangleup:=(\bigtriangleup_{m})_{m\geq0}$, the following theorem will give a
description of the cup product using the parallel paths. In fact, it
shows that the cup product is essentially given by concatenation of
paths.

\begin{proposition} \label{pro5.2}
Suppose $\eta:=(a, f)\in
HH^{m}(A_{\q})$ and $\xi:=(a', f')\in HH^{l}(A_{\q})$,
where $a, a'\in \B$, $f\in F^{m}$ and $f'\in F^{l}$. Then
$$
\eta\sqcup\xi=\left\{\begin{array}{lll}\q^{(m-j)j'}\left(aa',
f^{m+l}_{(i,j+j')}\right), & \quad\mbox{ if } f=f^{m}_{(i,j)}\mbox{ and
} f'=f^{l}_{(i-m,j')};\\
0, & \quad\mbox{ otherwise}.\end{array}\right.
$$
\end{proposition}

\begin{proof}
Let $f=f^{m}_{(i,j)}$ and
$f'=f^{l}_{(i',j')}$. Since the cup product of $\eta$ and $\xi$ is
given by the composition of the maps $ {\mathbb{X}}\xrightarrow[]{\
\bigtriangleup\ }{\mathbb{X}}\otimes_{\Lambda}{\mathbb{X}} \xrightarrow[]{\
\eta\otimes\xi\ }\Lambda\otimes_{\Lambda}\Lambda \xrightarrow[]{\
\nu\ }\Lambda$, we have $\eta\sqcup\xi=0$ if $i'\neq i-m$, and if
$i'=i-m$, then

$(1)$ $(\eta\sqcup\xi)\left(\ok(f^{m+l}_{(i'',j'')})\otimes
\tk(f^{m+l}_{(i'',j'')})\right)=0$, if $i''\neq i$,
or $j''\neq j+j'$;

$(2)$ if $i''= i$ and $j''=j+j'$, we have
$$
\begin{aligned}
&(\eta\sqcup\xi)\left(\ok(f^{m+l}_{(i,j'')})\otimes
\tk(f^{m+l}_{(i,j'')})\right)\\
=&\nu\Bigg(\sum_{s=0}^{m+l}\sum_{t=0}^{T}\q^{(s-t)(j''-t)}
\eta\left(\ok(f^{s}_{(i,t)})\otimes
\tk(f^{s}_{(i,t)})\right)\widehat{\otimes}
\xi\left(\ok(f^{m+l-s}_{(i-s,j''-t)})\otimes
\tk(f^{m+l-s}_{(i-s,j''-t)})\right)\Bigg)\\
=&\q^{(m-j)j'}aa',
\end{aligned}
$$
where $T:=\min\{s, j''\}$. The proof is finished.
\end{proof}

Now using the basis of $HH^{m}(A_{\q})$ in the
pervious section and the description of cup product in
Proposition \ref{pro5.2}, we can give the ring structure of
$HH^{\ast}(A_{\q})$. For convenience,
we denote $\w(y_{1}, y_{2})$ the
exterior algebra generated by $y_{1}, y_{2}$,
and for any ring homomorphisms $R\rightarrow\kk$ and
$S\rightarrow\kk$, denote the pullback by $R\times_{\kk}S$.
Let's first consider the case where $\q$ is not a root of unity.
In this case, the ring structure of $HH^{\ast}(A_{\q})$ is
relatively simple.

\begin{theorem} \label{thm5.3}
If $\q$ is not a root of unity. Then,
as graded $\kk$-algebra, we have the following isomorphism
$$
\theta: \quad HH^{\ast}(A_{\q})\longrightarrow
\kk[z_{1}, z_{2}]/(z_{1}^{2}, z_{1}z_{2}, z_{2}^{2})\times_{\kk}
\w(u_{1}, u_{2}),
$$
which is given by $\widehat{1}:=(e_{1}, f^{0}_{(1,0)})+(e_{2},
f^{0}_{(2,0)})\mapsto 1$, $\widehat{z}_{1}:=(\alpha_{1}\beta_{1},
f^{0}_{(1,0)})+(\alpha_{2}\beta_{2}, f^{0}_{(2,0)})$ $\mapsto z_{1}$,
$\widehat{z}_{2}:=(\alpha_{1}\beta_{1}, f^{0}_{(1,0)})
-(\alpha_{2}\beta_{2}, f^{0}_{(2,0)})\mapsto z_{2}$,
$\widehat{u}_{1}:=(\beta_{1}, f^{1}_{(2,0)})+(\beta_{2}, f^{1}_{(1,0)})
\mapsto u_{1}$, $\widehat{u}_{2}:=(\alpha_{1}, f^{1}_{(1,1)})+(\alpha_{2},
f^{1}_{(2,1)})\mapsto u_{2}$, where $z_{1}, z_{2}$ are in degree 0,
and $u_{1}, u_{2}$ are in degree 1.
\end{theorem}

\begin{proof}
By using the formula given in Proposition \ref{pro5.2}, we can directly
calculate that $\widehat{1}$ is the unit under the cup product,
$\widehat{z}_{i}\sqcup\widehat{z}_{j}=0$,
$\widehat{z}_{i}\sqcup\widehat{u}_{j}=0$,
$\widehat{u}_{i}\sqcup\widehat{u}_{i}=0$, for $i, j=1,2$,
and $\widehat{u}_{2}\sqcup\widehat{u}_{1}=(\alpha_{1}\beta_{1}, f^{2}_{(1,1)})
+(\alpha_{2}\beta_{2}, f^{2}_{(2,1)})=-\widehat{u}_{1}\sqcup\widehat{u}_{2}$.
Hence the correspondence in the theorem gives an isomorphism
between graded algebras.
\end{proof}

Next, let's consider the case of $\q=\pm1$.

\begin{theorem} \label{thm5.4}
If $\q=\pm1$. Then as graded $\kk$-algebra,
we have the following isomorphism
$$
\theta: \quad HH^{\ast}(A_{\q})\longrightarrow
\left(\kk[z_{1}, z_{2}]/(z_{1}^{2}, z_{1}z_{2}, z_{2}^{2})\times_{\kk}
\w(u_{1}, u_{2}, u_{3}, u_{4})\right)[w_{0}, w_{1}, w_{2}]/I,
$$
which is given by $\widehat{1}:=(e_{1}, f^{0}_{(1,0)})+(e_{2},
f^{0}_{(2,0)})\mapsto 1$, $\widehat{z}_{1}:=(\alpha_{1}\beta_{1},
f^{0}_{(1,0)})+\q(\alpha_{2}\beta_{2}, f^{0}_{(2,0)})$ $\mapsto z_{1}$,
$\widehat{z}_{2}:=(\alpha_{1}\beta_{1}, f^{0}_{(1,0)})
-\q(\alpha_{2}\beta_{2}, f^{0}_{(2,0)})\mapsto z_{2}$,
$\widehat{u}_{1}:=\q(\alpha_{1}, f^{1}_{(1,0)})+(\alpha_{2}, f^{1}_{(2,0)})
\mapsto u_{1}$, $\widehat{u}_{2}:=(\beta_{1}, f^{1}_{(2,0)})+(\beta_{2},
f^{1}_{(1,0)})\mapsto u_{2}$, $\widehat{u}_{3}:=(\alpha_{1}, f^{1}_{(1,1)})
+(\alpha_{2}, f^{1}_{(2,1)})\mapsto u_{3}$,
$\widehat{u}_{4}:=(\beta_{1}, f^{1}_{(2,1)})+\q(\beta_{2},
f^{1}_{(1,1)})\mapsto u_{4}$, $\widehat{w}_{j}:=(e_{1}, f^{2}_{(1,j)})
+\q^{j}(e_{2}, f^{2}_{(2,j)})\mapsto w_{j}$,
$j=0, 1, 2$, where the ideal $I$ is given by
$$
\begin{aligned}
&\big(u_{1}u_{3},\; u_{2}u_{4},\; z_{2}w_{0},\; z_{2}w_{1},\;
u_{1}u_{2}-\q z_{1}w_{0},\; u_{1}u_{4}-z_{1}w_{1},\; u_{1}u_{4}-u_{3}u_{2},\;
u_{3}u_{4}-z_{1}w_{2},\\
&\quad z_{2}w_{2},\; u_{1}w_{1}-\q u_{3}w_{0},\; u_{1}w_{2}-u_{3}w_{1},\;
u_{2}w_{1}-\q u_{4}w_{0},\; u_{2}w_{2}-u_{4}w_{1},\;
w_{1}^{2}-\q w_{0}w_{2}\big),
\end{aligned}
$$
and $z_{1}, z_{2}$ are in degree 0, $u_{1}, u_{2}, u_{3}, u_{4}$ are
in degree 1, $w_{0}, w_{1}, w_{2}$ are in degree 2.
\end{theorem}

\begin{proof}
Firstly, using the same analysis as in Theorem \ref{thm5.3}, we get that
$\widehat{z}_{1}, \widehat{z}_{2}, \widehat{u}_{1}, \widehat{u}_{2},
\widehat{u}_{3}, \widehat{u}_{4}$ generate a subalgebra $\Lambda$ of
$HH^{\ast}(A_{\q})$, which is isomorphic to $\kk[z_{1}, z_{2}]/(z_{1}^{2},
z_{1}z_{2}, z_{2}^{2})\times_{\kk}\w(u_{1}, u_{2}, u_{3}, u_{4})$.

Secondly, in order to prove that $HH^{\ast}(A_{\q})$ is generated by
$\widehat{w}_{1}, \widehat{w}_{2}, \widehat{w}_{3}$ over $\Lambda$, we
need show
$$
HH^{m}(A_{\q})=HH^{m-2}(A_{\q})\sqcup\{\widehat{w}_{0},
\widehat{w}_{1}, \widehat{w}_{2}\},
$$
for any $m\geq2$. Here we shall to discuss on the parity of $m$.
If $m$ is odd, then
$$
\begin{aligned}
\q(\alpha_{1}, f^{m}_{(1,0)})+(\alpha_{2}, f^{m}_{(2,0)})
&=[\q(\alpha_{1}, f^{m-2}_{(1,0)})+(\alpha_{2}, f^{m-2}_{(2,0)})]
\sqcup\widehat{w}_{0},\\
(\beta_{1}, f^{m}_{(2,m)})+\q(\beta_{2}, f^{m}_{(1,m)})
&=[(\beta_{1}, f^{m-2}_{(2,m)})+\q(\beta_{2}, f^{m-2}_{(1,m)})]
\sqcup\widehat{w}_{2},\\
\q^{m-j-1}[(\beta_{1}, f^{m}_{(2,j)})+\q^{j}(\beta_{2}, f^{m}_{(1,j)})]
&=[(\beta_{1}, f^{m-2}_{(2,j-1)})+\q^{j-1}(\beta_{2}, f^{m-2}_{(1,j-1)})]
\sqcup\widehat{w}_{1},\\
\q^{m-j-2}[\q^{j}(\alpha_{1}, f^{m}_{(1,j+1)})+(\alpha_{2}, f^{m}_{(2,j+1)})]
&=[\q^{j-1}(\alpha_{1}, f^{m-2}_{(1,j)})+(\alpha_{2}, f^{m-2}_{(2,j)})]
\sqcup\widehat{w}_{1},
\end{aligned}
$$
for any $0\leq j\leq m-1$.
Hence in this case, $HH^{m}(A_{\q})=HH^{m-2}(A_{\q})\sqcup
\{\widehat{w}_{0},\widehat{w}_{1}, \widehat{w}_{2}\}$.
If $m$ is even, then
$$
\begin{aligned}
&\q^{m-j-1}[(\alpha_{1}\beta_{1}, f^{m}_{(1,j)})
+\q^{j-1}(\alpha_{2}\beta_{2}, f^{m}_{(2,j)})]\\
=&[(\alpha_{1}\beta_{1},f^{m-2}_{(1,j-1)})+
\q^{j-2}(\alpha_{2}\beta_{2}, f^{m}_{(2,j-1)})\sqcup\widehat{w}_{1},\\
\q^{m-j-1}[(e_{1}, f^{m}_{(1,j)})&+\q^{j}(e_{2}, f^{m}_{(2,j)})]
=[(e_{1}, f^{m-2}_{(1,j-1)})+\q^{j-1}(e_{2}, f^{m-2}_{(2,j-1)})]
\sqcup\widehat{w}_{1},
\end{aligned}
$$
for any $0\leq j\leq m$.
That is, $HH^{m}(A_{\q})=HH^{m-2}(A_{\q})\sqcup\widehat{w}_{1}$.

Finally, one can check that the generators satisfy the following relations:
$$
\begin{aligned}
&\widehat{z}_{2}\sqcup\widehat{w}_{0}=\widehat{z}_{2}\sqcup\widehat{w}_{1}
=\widehat{z}_{2}\sqcup\widehat{w}_{2}=\widehat{u}_{1}\sqcup\widehat{u}_{3}
=\widehat{u}_{2}\sqcup\widehat{u}_{4}=0, \\
&\widehat{u}_{1}\sqcup\widehat{u}_{2}
=\q(\alpha_{1}\beta_{1}, f^{2}_{(1,0)})+(\alpha_{2}\beta_{2}, f^{2}_{(2,0)})
=\q(\widehat{z}_{1}\sqcup\widehat{w}_{0}),\\
&\widehat{u}_{1}\sqcup\widehat{u}_{4}
=(\alpha_{1}\beta_{1}, f^{2}_{(1,1)})+(\alpha_{2}\beta_{2}, f^{2}_{(2,1)})
=-(\widehat{u}_{2}\sqcup\widehat{u}_{3})
=(\widehat{z}_{1}\sqcup\widehat{w}_{1}),\\
&\widehat{u}_{3}\sqcup\widehat{u}_{4}
=(\alpha_{1}\beta_{1}, f^{2}_{(1,2)})+\q(\alpha_{2}\beta_{2}, f^{2}_{(2,2)})
=\widehat{z}_{1}\sqcup\widehat{w}_{2},\\
&\widehat{u}_{1}\sqcup\widehat{w}_{1}
=\q(\alpha_{1}, f^{3}_{(1,1)})+\q(\alpha_{2}, f^{3}_{(2,1)})
=\q(\widehat{u}_{3}\sqcup\widehat{w}_{0}),\\
&\widehat{u}_{1}\sqcup\widehat{w}_{2}
=\q(\alpha_{1}, f^{3}_{(1,2)})+(\alpha_{2}, f^{3}_{(2,2)})
=\widehat{u}_{3}\sqcup\widehat{w}_{1}, \\
&\widehat{u}_{2}\sqcup\widehat{w}_{1}
=\q(\beta_{1}, f^{3}_{(2,1)})+(\beta_{2}, f^{3}_{(1,1)})
=\q(\widehat{u}_{4}\sqcup\widehat{w}_{0}),\\
&\widehat{u}_{2}\sqcup\widehat{w}_{2}
=(\beta_{1}, f^{3}_{(2,2)})+(\beta_{2}, f^{3}_{(1,2)})
=\widehat{u}_{4}\sqcup\widehat{w}_{1}, \\
&\widehat{w}_{1}\sqcup\widehat{w}_{1}
=\q(e_{1}, f^{4}_{(1,2)})+\q(e_{2}, f^{4}_{(2,2)})
=\q(\widehat{w}_{0}\sqcup\widehat{w}_{2}).
\end{aligned}
$$
Therefore, the correspondence in the theorem gives
an isomorphism between graded algebras.
\end{proof}

Similar to the discussion of Theorem \ref{thm5.4},
we can get the ring structure of $HH^{\ast}(A_{\q})$ when
$\q$ is a primitive $s$-th $(s>2)$ root of unity.

\begin{theorem} \label{thm5.5}
If $\q$ is a primitive $s$-th $(s>2)$ root of unity. Then,
as graded $\kk$-algebra, we have the following isomorphisms.

(1) Whenever $s$ is odd,
$$
\theta_{1}: \quad HH^{\ast}(A_{\q})\longrightarrow
\kk[z_{1}, z_{2}]/(z_{1}^{2}, z_{1}z_{2}, z_{2}^{2})\times_{\kk}
\left(\w(u_{1}, u_{2})[w_{0}, w_{1}, w_{2}]/(w_{1}^{2}-w_{0}w_{2})\right),
$$
which is given by $(e_{1}, f^{0}_{(1,0)})+(e_{2}, f^{0}_{(2,0)})\mapsto
1$, $(\alpha_{i}\beta_{i}, f^{0}_{(i,0)})\mapsto z_{i}$, $i=1, 2$,
$(\beta_{1}, f^{1}_{(2,0)})+(\beta_{2}, f^{1}_{(1,0)})\mapsto u_{1}$,
$(\alpha_{1}, f^{1}_{(1,1)})+(\alpha_{2}, f^{1}_{(2,1)})\mapsto u_{2}$,
$(e_{1}, f^{2s}_{(1,js)})+(e_{2}, f^{2s}_{(2,js)})\mapsto w_{j}$,
$j=0, 1, 2$.

(2) Whenever $s$ is even,
$$
\theta_{2}: HH^{\ast}(A_{\q}^{1})\rightarrow
\kk[z_{1}, z_{2}]/(z_{1}^{2}, z_{1}z_{2}, z_{2}^{2})\times_{\kk}
\left(\w(u_{1}, u_{2})[w_{0}, w_{1}, w_{2}]/(w_{1}^{2}
-\q^{\frac{s^{2}}{4}}w_{0}w_{2})\right),
$$
which is given by $(e_{1}, f^{0}_{(1,0)})+(e_{2}, f^{0}_{(2,0)})\mapsto
1$, $(\alpha_{i}\beta_{i}, f^{0}_{(i,0)})\mapsto z_{i}$, $i=1, 2$,
$(\beta_{1}, f^{1}_{(2,0)})+(\beta_{2}, f^{1}_{(1,0)})\mapsto u_{1}$,
$(\alpha_{1}, f^{1}_{(1,1)})+(\alpha_{2}, f^{1}_{(2,1)})\mapsto u_{2}$,
$(e_{1}, f^{s}_{(1,\frac{js}{2})})+\q^{\frac{js}{2}}(e_{2},
f^{s}_{(2,\frac{js}{2})})\mapsto w_{j}$, $j=0, 1, 2$.
\end{theorem}

\begin{proof}
Here we only give the proof of (1). First, it is easy to see that
$\widehat{z}_{1}, \widehat{z}_{2}, \widehat{u}_{1}, \widehat{u}_{2}$ generate a
subalgebra $\Lambda$ of $HH^{\ast}(A_{\q})$, which is isomorphic to $\kk[z_{1},
z_{2}]/(z_{1}^{2}, z_{1}z_{2}, z_{2}^{2})\times_{\kk}\w(u_{1}, u_{2})$.
Second, in order to prove that $HH^{\ast}(A_{\q})$ is generated by
$\widehat{w}_{0}, \widehat{w}_{1}, \widehat{w}_{2}$ over $\Lambda$, we
need show
$$
HH^{m}(A_{\q})=HH^{m-2}(A_{\q})\sqcup\{\widehat{w}_{0},
\widehat{w}_{1}, \widehat{w}_{2}\},
$$
for any $m\geq2$. Here, we shall to discuss it in terms of the values of $m$.
If $m=2ls$, then
$$
\begin{aligned}
(e_{1}, f^{m}_{(1,j)})+(e_{2}, f^{m}_{(2,j)})
&=[(e_{1}, f^{m-2s}_{(1,j)})+(e_{2}, f^{m-2s}_{(2,j)})]
\sqcup\widehat{w}_{0},\\
(e_{1}, f^{m}_{(1,2ls-s)})+(e_{2}, f^{m}_{(2,2ls-s)})
&=[(e_{1}, f^{m-2s}_{(1,2ls-2s)})+(e_{2}, f^{m-2s}_{(2,2ls-2s)})]
\sqcup\widehat{w}_{1},\\
(e_{1}, f^{m}_{(1,2ls)})+(e_{2}, f^{m}_{(2,2ls)})
&=[(e_{1}, f^{m-2s}_{(1,2ls-2s)})+(e_{2}, f^{m-2s}_{(2,2ls-2s)})]
\sqcup\widehat{w}_{2},\end{aligned}$$
for any $j=ts,$ $0\leq t\leq 2l-1$. Hence, in this case, $HH^{m}(A_{\q})=HH^{m-2}(A_{\q})
\sqcup\{\widehat{w}_{0},\widehat{w}_{1}, \widehat{w}_{2}\}$.
If $m=2ls+1$,
$$
\begin{aligned}
(\beta_{1}, f^{m}_{(2,j)})+(\beta_{2}, f^{m}_{(1,j)})
&=[(\beta_{1}, f^{m-2s}_{(2,j)})+(\beta_{2}, f^{m-2s}_{(1,j)})]
\sqcup\widehat{w}_{0},\\
(\beta_{1}, f^{m}_{(2,m-s)})+(\beta_{2}, f^{m}_{(1,m-s)})
&=[(\beta_{1}, f^{m-2s}_{(2,m-2s)})+(\beta_{2}, f^{m-2s}_{(1,m-2s)})]
\sqcup\widehat{w}_{1},\\
(\beta_{1}, f^{m}_{(2,m)})+(\beta_{2}, f^{m}_{(1,m)})
&=[(\beta_{1}, f^{m-2s}_{(2,m-2s)})+(\beta_{2}, f^{m-2s}_{(1,m-2s)})]
\sqcup\widehat{w}_{2},\\
(\alpha_{1}, f^{m}_{(1,j+1)})+(\alpha_{2}, f^{m}_{(2,j+1)})
&=[(\alpha_{1}, f^{m-2s}_{(1,j+1)})+(\alpha_{2}, f^{m-2s}_{(1,j+1)})]
\sqcup\widehat{w}_{0},\\
(\alpha_{1}, f^{m}_{(1,m-s)})+(\alpha_{2}, f^{m}_{(2,m-s)})
&=[(\alpha_{1}, f^{m-2s}_{(1,m-2s)})+(\alpha_{2}, f^{m-2s}_{(1,m-2s)})]
\sqcup\widehat{w}_{1},\\
(\alpha_{1}, f^{m}_{(1,m)})+(\alpha_{2}, f^{m}_{(2,m)})
&=[(\alpha_{1}, f^{m-2s}_{(1,m-2s)})+(\alpha_{2}, f^{m-2s}_{(1,m-2s)})]
\sqcup\widehat{w}_{2},\\
\end{aligned}
$$
for any $j=ts$, $0\leq t\leq 2l-1$. Hence, in this case, $HH^{m}(A_{\q})=HH^{m-2}(A_{\q})
\sqcup\{\widehat{w}_{0},\widehat{w}_{1}, \widehat{w}_{2}\}$.
If $m=2ls+2$, then$$\begin{aligned}
(\alpha_{1}\beta_{1}, f^{m}_{(1,j)})+(\alpha_{2}\beta_{2}, f^{m}_{(2,j)})
&=[(\alpha_{1}\beta_{1}, f^{m-2s}_{(1,j+1)})+(\alpha_{2}\beta_{2}, f^{m-2s}_{(1,j)})]
\sqcup\widehat{w}_{0},\\
(\alpha_{1}\beta_{1}, f^{m}_{(1,m-s)})+(\alpha_{2}\beta_{2}, f^{m}_{(2,m-s)})
&=[(\alpha_{1}\beta_{1}, f^{m-2s}_{(1,m-2s)})+(\alpha_{2}\beta_{2}, f^{m-2s}_{(1,m-2s)})]
\sqcup\widehat{w}_{1},\\
(\alpha_{1}\beta_{1}, f^{m}_{(1,m)})+(\alpha_{2}\beta_{2}, f^{m}_{(2,m)})
&=[(\alpha_{1}\beta_{1}, f^{m-2s}_{(1,m-2s)})+(\alpha_{2}\beta_{2}, f^{m-2s}_{(1,m-2s)})]
\sqcup\widehat{w}_{2},\\
\end{aligned}
$$
for any $j=ts$, $0\leq t\leq 2l-1$. Hence, in this case, $HH^{m}(A_{\q})=HH^{m-2}(A_{\q})
\sqcup\{\widehat{w}_{0},\widehat{w}_{1}, \widehat{w}_{2}\}$.

Finally, one can check that the generators satisfy the following relations:
$$
\begin{aligned}
z_{1}\sqcup z_{1}=z_{1}\sqcup z_{2}=z_{2}\sqcup z_{2},\quad
\widehat{w}_{1}\sqcup\widehat{w}_{1}
=(e_{1}, f^{4s}_{(1,2s)})+(e_{2}, f^{4s}_{(2,2s)})
=(\widehat{w}_{0}\sqcup\widehat{w}_{2}).
\end{aligned}
$$
Therefore, the correspondence in the theorem gives
an isomorphism between graded algebras.
\end{proof}

The support variety of a module over a group algebra is an affine variety
that encodes many of the homological properties of the module.
For any finite-dimensional $\kk$-algebra $\Lambda$, let ${\mathcal{N}}$
be the ideal of $HH^{\ast}(\Lambda)$ generated by all the homogeneous
nilpotent elements. If $HH^{\ast}(\Lambda)/{\mathcal{N}}$ is a finite-dimensional
commutative $\kk$-algebra, then it is used to define the support
varieties for $\Lambda$-modules \cite{SS}. Moreover, Snashall and Solberg in
\cite{SS} conjectured that $HH^{\ast}(\Lambda)/{\mathcal{N}}$ is
finitely generated for any finite-dimensional $\kk$-algebra
$\Lambda$. At the end of this section, let us consider the quotient ring
$HH^{\ast}(\Lambda)/{\mathcal{N}}$.
The ring structure of $HH^{\ast}(A_{\q})/{\mathcal{N}}$ is given
in \cite{ST1} by considering the graded center of the Koszul dual of $A_{\q}$.
Here, using the generators of $HH^{\ast}(A_{\q})$
given in the theorem above, note that
$\widehat{z}_{i}, \widehat{u}_{i}$ are nilpotence, for all $i=1,2$,
we can give the ring structure of $HH^{\ast}(A_{\q})/{\mathcal{N}}$ directly.

\begin{corollary}[\cite{ST1} Theorem 2.6] \label{cor5.6}
For the quotient algebra $HH^{\ast}(A_{\q})/{\mathcal{N}}$, we have

(1) if $\q$ is not a root of unity, then
$HH^{\ast}(A_{\q})/{\mathcal{N}}\cong\kk$;

(2) if $\q=\pm1$, then as graded $\kk$-algebra,
$HH^{\ast}(A_{\q})/{\mathcal{N}}\cong\kk[w_{0}, w_{1}, w_{2}]/
(w_{1}^{2}-\q w_{0}w_{2})$;

(3) if $\q$ is a primitive $s$-th $(s>2)$ root
of unity and $s$ is odd, then as graded $\kk$-algebra,
$HH^{\ast}(A_{\q})/{\mathcal{N}}\cong\kk[w_{0}, w_{1}, w_{2}]
/(w_{1}^{2}-w_{0}w_{2})$;

(4) if $\q$ is a primitive $s$-th $(s>2)$ root
of unity and $s$ is even, then as graded $\kk$-algebra,
$HH^{\ast}(A_{\q})/{\mathcal{N}}\cong\kk[w_{0}, w_{1}, w_{2}]/
(w_{1}^{2}-\q^{\frac{s^{2}}{4}}w_{0}w_{2})$.
\end{corollary}

Hence the Snashall-Solberg conjecture is true for the
quantum zigzag algebras $A_{\q}$ of type $\widetilde{\mathbf{A}}_{1}$.
This conclusion is very useful for us to understand the representation
theory of this kind of algebra. In fact, we can use the
support varieties to give the complexity of $A_{\q}$-modules.
See literature \cite{EHSST} for details.
In \cite{Her}, the notion of Gerstenhaber ideal of Gerstenhaber algebra is introduced.
Let $(\Lambda^{\bullet},\; \sqcup,\; [\ \,,\,\ ])$ be a Gerstenhaber algebra,
$S\subseteq\Lambda^{\bullet}$ be a subset of homogeneous elements.
Recall that the Gerstenhaber ideal $\mathcal{G}(S)$ of
$(\Lambda^{\bullet},\; \sqcup,\; [\ \,,\,\ ])$ generated by $S$ is the small
subset of $\Lambda^{\bullet}$ containing $S$ and being both an ideal with
respect to $\sqcup$ and $[\ \,,\,\ ]$. For the Gerstenhaber algebras
of quantum zigzag algebras, we have given a detailed characterization,
so it is easy to obtain its Gerstenhaber ideal which generated by all
nilpotent homogeneous elements.

\begin{corollary} \label{cor5.7}
Let $A_{\q}$ be the quantum zigzag algebra. Denote by $\mathcal{G}$ the
Gerstenhaber ideal of $(HH^{\ast}(A_{\q}),\; \sqcup,\; [\ \,,\,\ ])$
generated by all nilpotent homogeneous elements. Then $HH^{\ast}(A_{\q})/\mathcal{G}
\cong\kk$.
\end{corollary}

\bigskip
\section{Batalin-Vilkovisky algebraic structure on $HH^{\ast}(A_{\q})$}\label{BV}

In this section, we construct two comparison morphisms between the
minimal projective bimodule resolution ${\mathbb{P}}=(P_{n},d_{n})$
and the reduced bar resolution $\bar{{\mathbb{B}}}=(\bar{B}_{m},
\bar{d}_{m})$ of $A_{\q}$ by using the weak self-homotopy.
By these comparison morphisms, and applying  the bilinear form constructed
by Tradler and Volkov, we give the Batalin-Vilkovisky algebraic structure
on $HH^{\ast}(A_{\q})$ for all $\q\neq0$.

Let $\Lambda$ be an algebra over field $\kk$. Given two left
$\Lambda$-modules $M$ and $N$, let $\mathbb{C}$ (resp. $\mathbb{D}$)
be a projective resolution of $M$ (resp. $N$). Then, for each
morphism of $\Lambda$-modules $f: M\rightarrow N$, there exists a
chain map $\bar{f}=(f_{i})_{i\geq0}: {\mathbb{C}}\rightarrow{\mathbb{D}}$
lifting $f$. Where $\bar{f}$ is called comparison morphism.
Here, we will use the method in \cite{IIVZ} to construct the comparison
morphisms between the minimal projective bimodule resolution and the
reduced bar resolution of $A_{\q}$.

For any $\kk$-algebra $\Lambda$, let
$$
\xymatrix@C=2em{
\cdots \ar[r]& Q_{m}\ar[r]^{d_{m}^{Q}} & Q_{m-1}\ar[r]&\cdots\ar[r]
& Q_{2}\ar[r]^{d_{2}^{Q}}& Q_{1}\ar[r]^{d_{1}^{Q}}&
Q_{0}\ar[r]^{d_{0}^{Q}} & N\ar[r]& 0  }
$$
be a complex of left $\Lambda$-modules. Recall that a weak self-homotopy
of this complex is a collection of $\kk$-linear maps
$t_{m}: Q_{m}\rightarrow Q_{m+1}$ for each $m\geq 0$ and
$t_{-1}: N\rightarrow Q_{0}$ such that $d_{0}^{Q}\circ t_{-1}=\I_{N}$,
and $t_{m-1}\circ d_{m}^{Q}+d_{m+1}^{Q}\circ t_{m}=\I_{Q_{m}}$ for $m\geq 0$
(see \cite{BZZ}). It is shown in \cite{IIVZ} that, each exact complex of
left $\Lambda$-modules has a weak self-homotopy $\{t_{i}\}_{i\geq -1}$
such that $t_{i+1}\circ t_{i}=0$ for any $i\geq -1$.

For that $\kk$-algebra $A_{\q}$, in section \ref{homology}, we have constructed a
minimal projective bimodule resolution ${\mathbb{P}}=(P_{m},d_{m})$
of $A_{\q}$.  Since this resolution splits as complexes of one-sided
modules, one can even choose a weak self-homotopy $\{t_{i}\}_{i\geq -1}$
which are right module homomorphisms.
Define $t_{-1}: A_{\q}\rightarrow P_{0}$,
$t_{-1}(e_{i})=e_{i}\otimes e_{i}$, for $i=1, 2$;
and if $m\geq 0$, $t_{m}: P_{m}\rightarrow P_{m+1}$ is given
as following: for any $i=1, 2$,

\begin{itemize}
\item[(1)] $t_{m}\left(\ok(f^{m}_{(i,j)})\otimes\tk(f^{m}_{(i,j)})\right)=0$,
     \qquad $0\leq j\leq m$;

\item[(2)] $t_{m}\left(\beta_{i}\otimes\tk(f^{m}_{(i,j)})\right)=
     \q^{-j}\ok(f^{m+1}_{(i+1,j)})\otimes
     \tk(f^{m+1}_{(i+1,j)})$,\qquad $0\leq j\leq m$;

\item[(3)] $t_{m}\left(\alpha_{i-1}\otimes\tk(f^{m}_{(i,j)})\right)=
     \left\{\begin{array}{ll} 0, &\quad \mbox{if } 0\leq j\leq m-1,\\
     \ok(f^{m+1}_{(i-1,j+1)})\otimes\tk(f^{m+1}_{(i-1,j+1)}),
     &\quad \mbox{if } j=m, \mbox{ or }m=0;\end{array}\right.$

\item[(4)] $t_{m}\left(\alpha_{i}\beta_{i}\otimes\tk(f^{m}_{(i,j)})\right)$
\end{itemize}
\vspace{-1.6mm}
\hspace{0.5cm}$=\left\{\begin{array}{ll}
     \q^{-j}\alpha_{i}\otimes\tk(f^{m+1}_{(i+1,j)}),
     &\quad \mbox{if } 0\leq j\leq m-1,\\
     \q^{-m}\alpha_{i}\otimes\tk(f^{m+1}_{(i+1,m)})
     +(-1)^{m}\q^{-m}\ok(f^{m+1}_{(i,m+1)})\otimes\beta_{i+m},
     &\quad \mbox{if } j=m, \mbox{ or }m=0.\end{array}\right.$\\
Then, it is easy to see that $\{t_{i}\}_{i\geq -1}$ satisfying
$t_{i+1}\circ t_{i}=0$ for any $i\geq -1$.
Moreover, we have the following lemma.

\begin{lemma}\label{lem6.1}
The defined maps $\{t_i\}_{i\geq -1}$ above form a weak self-homotopy
over the minimal projective resolution ${\mathbb{P}}=(P_{m},d_{m})$.
\end{lemma}

\begin{proof}
Firstly, we have $d_{0}\circ t_{-1}(e_{i})=d_{0}(e_{i}\otimes e_{i})
=e_{i}$, for $i=1, 2$. That is
$d_{0}\circ t_{-1}=\I_{A_{\q}}$.
Secondly, for any $m>0$, $i=1, 2$, we have
$$
\begin{aligned}
&d_{m+1}\circ t_{m}\left(\alpha_{i}\beta_{i}\otimes\tk(f^{m}_{(i,m)})\right)\\
=&d_{m+1}\left(\q^{-m}\alpha_{i}\otimes\tk(f^{m+1}_{(i+1,m)})
+(-1)^{m}\q^{-m}\ok(f^{m+1}_{(i,m+1)})\otimes\beta_{i+m}\right)\\
=&\q^{-m}\alpha_{i}\Big(\alpha_{i+1}\otimes\tk(f^{m}_{(i+2,m-1)})
+(-1)^{m+1}\q\ok(f^{m}_{(i+1,m-1)})\otimes\alpha_{i+m-1}\\
&\hspace{0.5cm}+\q^{m}\beta_{i}\otimes\tk(f^{m}_{(i,m)})
+(-1)^{m+1}\ok(f^{m}_{(i+1,m)})\otimes\beta_{i+m}\Big)\\
&\hspace{0.5cm}+(-1)^{m}\q^{-m}\left(\alpha_{i}\otimes\tk(f^{m}_{(i+1,m)})
+(-1)^{m+1}\ok(f^{m}_{(i,m)})\otimes\alpha_{i+m}\right)\beta_{i+m}\\
=&(-1)^{m+1}\q^{1-m}\alpha_{i}\otimes\alpha_{i+m-1}
+\alpha_{i}\beta_{i}\otimes\tk(f^{m}_{(i,m)})
-\q^{-m}\ok(f^{m}_{(i,m)})\otimes\alpha_{i+m}\beta_{i+m},
\end{aligned}
$$
and
$$
\begin{aligned}
&t_{m-1}\circ d_{m}\left(\alpha_{i}\beta_{i}\otimes\tk(f^{m}_{(i,m)})\right)\\
=&t_{m-1}\left(\alpha_{i}\beta_{i}
\left(\alpha_{i}\otimes\tk(f^{m-1}_{(i+1,m-1)})
+(-1)^{m}\ok(f^{m-1}_{(i,m-1)})\otimes\alpha_{i+m-1}\right)\right)\\
=&(-1)^{m}t_{m-1}\left(\alpha_{i}\beta_{i}\otimes
\tk(f^{m-1}_{(i,m-1)})\right)\alpha_{i+m-1}\\
=&(-1)^{m}\q^{1-m}\alpha_{i}\otimes\alpha_{i+m-1}
-\q^{-m}\ok(f^{m}_{(i,m)})\otimes\beta_{i+m-1}\alpha_{i+m-1}.
\end{aligned}
$$
Then, it is easy to see $[d_{m+1}\circ t_{m}+t_{m-1}\circ d_{m}]
\left(\alpha_{i}\beta_{i}\otimes\tk(f^{m}_{(i,m)})\right)
=\alpha_{i}\beta_{i}\otimes\tk(f^{m}_{(i,m)})$.
Similarly, one can check that the mapping $d_{m+1}\circ t_{m}+t_{m-1}\circ d_{m}$
remain unchanged on other generators of $P_{m}$ as right
$A_{\q}$-module. That is to say, $d_{m+1}\circ t_{m}+t_{m-1}\circ
d_{m}=\I_{P_{m}}$.

In conclusion, $\{t_{i}\}_{i\geq -1}$ form a weak self-homotopy.
\end{proof}

Consider the reduced bar resolution $\bar{{\mathbb{B}}}
=(\bar{B}_{m}, \bar{d}_{m})$ of $A_{\q}$.
Recall that $\bar{B}_{m}=A_{\q}\otimes_{E}
\bar{A}_{\q}^{\otimes_{E} m}\otimes_{E} A_{\q}$,
where $E$ is the subalgebra of $A_{\q}$
generated by $\{e_{1}, e_{2}\}$.
For convenience, we write a element in $\bar{B}_{m}$ as
$a_{0}\boti a_{1}\boti\dots\boti a_{m}\boti a_{m+1}$, where
$\boti:=\otimes_{E}$. Then
$\bar{{\mathbb{B}}}=(\bar{B}_{m}, \bar{d}_{m})$
has a family of right module homomorphisms $\{s_{i}\}_{i\geq -1}$ as a weak
self-homotopy, which is defined by the formula
$$
s_{m}(a_{0}\boti a_{1}\boti\dots\boti a_{m}\boti\tk(a_{m}))
=\ok(a_{0})\boti a_{0}\boti a_{1}\boti\dots\boti a_{m}\boti\tk(a_{m}).
$$
And it is easy to see that $s_{m+1}\circ s_{m}=0$ for $m\geq -1$.

Now, let us consider the comparison morphisms between ${\mathbb{P}}=(P_{m},d_{m})$
and $\bar{{\mathbb{B}}}=(\bar{B}_{m}, \bar{d}_{m})$.
Firstly, we shall construct a family of morphisms $\Phi=(\Phi_{m})_{m\geq0}$
from ${\mathbb{P}}=(P_{m},d_{m})$ to $\bar{{\mathbb{B}}}=(\bar{B}_{m},
\bar{d}_{m})$ as follows:
\begin{itemize}
\item[(1)] $\Phi_{0}: P_{0}=\bigoplus_{i=1}^{n+1}A_{\q}e_{i}\otimes
    e_{i}A_{\q}\longrightarrow \bar{B}_{0}
    =A_{\q}\otimes_{E} A_{\q}$ is the bimodule isomorphism
    $(e_{i}\otimes e_{i})\mapsto e_{i}\boti e_{i}$;

\item[(2)] for $m\geq1$, the bimodule morphism $\Phi_{m}$ is defined inductively
     by the map $s_{m-1}\circ\Phi_{m-1}\circ d_{m}$ acting on the free basis
     elements of $P_{m}$ as bimodule.
\end{itemize}
We now define $g^{1}_{(i+1,0)}=\beta_{i}$, and $g^{1}_{(i,1)}=\alpha_{i}$,
for $i=1, 2$;
for $m\geq2$, define inductively $g^{m}_{(i,j)}$ by
setting
$$
g^{m}_{(i,j)}=\alpha_{i}\boti g^{m-1}_{(i+1,j-1)}
+\q^{j}\beta_{i-1}\boti g^{m-1}_{(i-1,j)},
$$
for any $i=1, 2$ and $0\leq j\leq m$,
where $g^{m-1}_{(i,j)}=g^{m-1}_{(i',j)}$ if $i\equiv i' \mod2$,
$\beta_{j}=\beta_{j'}$ if $j\equiv j' \mod2$, and $g^{m-1}_{(i,j)}=0$
if $j<0$ or $j>m-1$.
Then we have

\begin{lemma}\label{lem6.2}
The morphism $\Phi=(\Phi_{m})_{m\geq0}: {\mathbb{P}}=(P_{m},d_{m})
\rightarrow \bar{{\mathbb{B}}}=(\bar{B}_{m},
\bar{d}_{m})$ is a chain map, and for $m\geq1$,
$$
\Phi_{m}\left(\ok(f^{m}_{(i,j)})\otimes\tk(f^{m}_{(i,j)})\right)
=\ok(g^{m}_{(i,j)})\boti g^{m}_{(i,j)}\boti\tk(g^{m}_{(i,j)}),
$$
for any $i=1, 2$ and $0\leq j\leq m$, where
$\ok(g^{m}_{(i,j)})=\ok(f^{m}_{(i,j)})$,
$\tk(g^{m}_{(i,j)})=\tk(f^{m}_{(i,j)})$.
\end{lemma}

\begin{proof}
For $m=0$, it is easy to see that $d_{0}(e_{i}\otimes e_{i})=e_{i}
=\bar{d}_{0}\circ\Phi_{0}(e_{i}\otimes e_{i})$. Suppose that
$\bar{d}_{i}\circ\Phi_{i}=\Phi_{i-1}\circ d_{i}$ for
$1\leq i\leq m-1$. Then
$$
\begin{aligned}
&\bar{d}_{m}\circ\Phi_{m}\left(\ok(f^{m}_{(i,j)})\otimes
\tk(f^{m}_{(i,j)})\right)\\
=&\bar{d}_{m}\circ s_{m-1}\circ\Phi_{m-1}\circ d_{m}
\left(\ok(f^{m}_{(i,j)})\otimes\tk(f^{m}_{(i,j)})\right)\\
=&[\Phi_{m-1}\circ d_{m}-s_{m-2}\circ\bar{d}_{m-1}\circ\Phi_{m-1}\circ
d_{m}]\left(\ok(f^{m}_{(i,j)})\otimes\tk(f^{m}_{(i,j)})\right)\\
=&[\Phi_{m-1}\circ d_{m}-s_{m-2}\circ\Phi_{m-2}\circ d_{m-1}\circ
d_{m}]\left(\ok(f^{m}_{(i,j)})\otimes\tk(f^{m}_{(i,j)})\right)\\
=&\Phi_{m-1}\circ d_{m}\left(\ok(f^{m}_{(i,j)})\otimes\tk(f^{m}_{(i,j)})\right),
\end{aligned}
$$
for any $f^{m}_{(i,j)}\in F^{m}$. That is, $\Phi=(\Phi_{m})_{m\geq0}$
is a chain map.

For the morphism $\Phi=(\Phi_{m})_{m\geq0}$, one can check that
$\Phi_{1}\left(\ok(\alpha_{i})\otimes\tk(\alpha_{i})\right)
=s_{0}\circ\Phi_{0}\circ d_{1}\left(\ok(\alpha_{i})\otimes\tk(\alpha_{i})\right)
=\ok(\alpha_{i})\boti\alpha_{i}\boti\tk(\alpha_{i})$, and
$\Phi_{1}\left(\ok(\beta_{i})\otimes\tk(\beta_{i})\right)
=\ok(\beta_{i})\boti\beta_{i}\boti\tk(\beta_{i})$ for $i=1, 2$.
Now suppose that $\Phi_{m}\left(\ok(f^{m}_{(i,j)})\otimes\tk(f^{m}_{(i,j)})\right)
=\ok(g^{m}_{(i,j)})\boti g^{m}_{(i,j)}\boti\tk(g^{m}_{(i,j)})$.
Then
$$
\begin{aligned}
&\Phi_{m+1}\left(\ok(f^{m+1}_{(i,j)})\otimes\tk(f^{m+1}_{(i,j)})\right)\\
=&s_{m}\circ\Phi_{m}\circ d_{m+1}
\left(\ok(f^{m+1}_{(i,j)})\otimes\tk(f^{m+1}_{(i,j)})\right)\\
=&s_{m}\circ\Phi_{m}\Big(\alpha_{i}\otimes\tk(f^{m}_{(i+1,j-1)})
+(-1)^{m+1}\q^{m+1-j}\ok(f^{m}_{(i,j-1)})\otimes\alpha_{i+2j-m-2}\\
&\hspace{1cm}+\q^{j}\beta_{i-1}\otimes\tk(f^{m}_{(i-1,j)})
+(-1)^{m+1}\ok(f^{m}_{(i,j)})\otimes\beta_{i+2j-m-1}\Big)\\
=&s_{m}\Big(\alpha_{i}\boti g^{m}_{(i+1,j-1)}\boti\tk(g^{m}_{(i+1,j-1)})
+\q^{j}\beta_{i-1}\boti g^{m}_{(i-1,j)}\boti\tk(g^{m}_{(i-1,j)})\\
&\hspace{1cm}+(-1)^{m+1}\q^{m+1-j}\ok(g^{m}_{(i,j-1)})\boti g^{m}_{(i,j-1)}\boti
\alpha_{i+2j-m-2}\\
&\hspace{1cm}+(-1)^{m+1}\ok(g^{m}_{(i,j)})\boti g^{m}_{(i,j)}\boti\beta_{i+2j-m-1}\Big)\\
=&\ok(\alpha_{i})\boti\alpha_{i}\boti g^{m}_{(i+1,j-1)}\boti
\tk(g^{m}_{(i+1,j-1)})+\q^{j}\ok(\beta_{i-1})\boti\beta_{i-1}\boti
g^{m}_{(i-1,j)}\boti\tk(g^{m}_{(i-1,j)})\\
=&\ok(g^{m+1}_{(i,j)})\boti g^{m+1}_{(i,j)}\boti\tk(g^{m+1}_{(i,j)}).
\end{aligned}
$$
This completes the proof.
\end{proof}

Secondly, using similar methods, we construct another family of
morphisms $\Psi=(\Psi_{m})_{m\geq0}$ from
$\bar{{\mathbb{B}}}=(\bar{B}_{m}, \bar{d}_{m})$
to ${\mathbb{P}}=(P_{m},d_{m})$ as follows:
\begin{itemize}
\item[(1)] $\Psi_{0}: \bar{B}_{0}=A_{\q}\otimes_{E} A_{\q}
    \longrightarrow P_{0}=\bigoplus_{i=1}^{n+1}A_{\q}e_{i}\otimes
    e_{i}A_{\q}$ is the bimodule isomorphism
    $e_{i}\boti e_{i}$ $\mapsto (e_{i}\otimes e_{i})$;

\item[(2)] for $m\geq1$, the bimodule morphism $\Psi_{m}$ is defined inductively
     by the map $t_{m-1}\circ\Psi_{m-1}\circ\bar{d}_{m}$ acting on the
     free basis elements of $\bar{B}_{m}$.
\end{itemize}
Then, similar to the proof of Lemma \ref{lem6.2}, we have the following lemma.

\begin{lemma}\label{lem6.3}
The morphism $\Psi=(\Psi_{m})_{m\geq0}: \bar{{\mathbb{B}}}=
(\bar{B}_{m}, \bar{d}_{m})\rightarrow{\mathbb{P}}=(P_{m},d_{m})$
is a chain map.
\end{lemma}

For the morphism $\Psi=(\Psi_{m})_{m\geq0}$, we can't get a unified formula.
But we can give a concrete description of the mapping $\Psi$ on each basis element
as following: if $m=1, 2$,
\begin{itemize}
\item[$\bullet$] $\Psi_{1}(\ok(a)\boti a\boti\tk(a))=\left\{\begin{array}{ll}
     \ok(a)\otimes\tk(a), &\quad \mbox{if } a=\alpha_{i}\mbox{ or }\beta_{i};\\
     \alpha_{i}\otimes\tk(a)+\ok(a)\otimes\beta_{i},
     &\quad \mbox{if } a=\alpha_{i}\beta_{i};\end{array}\right.$

\item[$\bullet$] $\Psi_{2}(\ok(a_{1})\boti a_{1}\boti a_{2}\boti\tk(a_{2}))\\
     =\left\{\begin{array}{llllllllll}
     \ok(f^{2}_{(i,2)})\otimes\tk(f^{2}_{(i,2)}), &\quad \mbox{if } a_{1}=\alpha_{i},\;
     a_{2}=\alpha_{i+1};\\
     0, &\quad \mbox{if } a_{1}=\alpha_{i},\; a_{2}=\beta_{i};\\
     \ok(f^{2}_{(i,2)})\otimes\beta_{i+1},
     &\quad \mbox{if } a_{1}=\alpha_{i},\; a_{2}=\alpha_{i+1}\beta_{i+1};\\
     \q^{-1}\ok(f^{2}_{(i+1,1)})\otimes\tk(f^{2}_{(i+1,1)}),
     &\quad \mbox{if } a_{1}=\beta_{i},\;
     a_{2}=\alpha_{i};\\
     \ok(f^{2}_{(i+1,0)})\otimes\tk(f^{2}_{(i+1,0)}), &\quad \mbox{if }
     a_{1}=\beta_{i},\;
     a_{2}=\beta_{i-1};\\
     -\q^{-1}\alpha_{i+1}\otimes\tk(f^{2}_{(i+2,0)})
     +\q^{-1}\ok(f^{2}_{(i+1,1)})\otimes\beta_{i},
     &\quad \mbox{if }a_{1}=\beta_{i},\; a_{2}=\alpha_{i}\beta_{i};\\
     \q^{-1}\alpha_{i}\otimes\tk(f^{2}_{(i+1,1)})
     -\q^{-1}\ok(f^{2}_{(i,2)})\otimes\beta_{i+1},
     &\quad \mbox{if }a_{1}=\alpha_{i}\beta_{i},\; a_{2}=\alpha_{i};\\
     \alpha_{i}\otimes\tk(f^{2}_{(i+1,0)}),
     &\quad \mbox{if }a_{1}=\alpha_{i}\beta_{i},\; a_{2}=\beta_{i-1};\\
     \q^{-1}\alpha_{i}\otimes\beta_{i},
     &\quad \mbox{if }a_{1}=\alpha_{i}\beta_{i},\; a_{2}=\alpha_{i}\beta_{i}.
     \end{array}\right.$
\end{itemize}
More general, for $Y_{m}:=\ok(a_{1})\boti a_{1}\boti a_{2}\boti\cdots\boti
a_{m}\boti\tk(a_{m})$, we denote $X_{m}:=a_{1}\boti a_{2}\boti$
$\cdots\boti a_{m}$, $m\geq 3$, we have
\begin{itemize}
\item[$\bullet$] if $X_{m}=\beta_{i}\boti\beta_{i+1}\boti\cdots\boti\beta_{m+i-1}$, then
     $$
     \Psi_{m}(Y_{m})=\ok(f^{m}_{(i+1, 0)})\otimes\tk(f^{m}_{(i+1, 0)});
     $$

\item[$\bullet$] if $X_{m}=\beta_{i}\boti\beta_{i-1}\boti\cdots\boti\beta_{i-j}\boti
    \alpha_{i-j}\boti\alpha_{i-j+1}\boti\cdots\boti\alpha_{m+i-2j-2}$, then
     $$
     \Psi_{m}(Y_{m})
     =\q^{-(j+1)(m-j-1)}\ok(f^{m}_{(i+1, m-j-1)})\otimes\tk(f^{m}_{(i+1, m-j-1)});
     $$

\item[$\bullet$] if $X_{m}=\beta_{i}\boti
     \beta_{i-1}\boti\cdots\boti\beta_{i-l}\boti\alpha_{i-l}\beta_{i-l}
     \boti\beta_{i-l-1}\boti\cdots\boti\beta_{i-l-j}\boti\alpha_{i-l-j}\boti
     \alpha_{i-l-j+1}\boti$ $\cdots\boti\alpha_{m+i-2l-2j-3}$, then

(1) $j>0$,
     $$
     \Psi_{m}(Y_{m})=(-1)^{l+1}\q^{-(l+1)-(l+j+2)(m-l-j-2)}\alpha_{i+1}\otimes
     \tk(f^{m}_{(i+2, m-l-j-2)});
     $$

(2) $j=0$,
     $$
     \begin{aligned}
     \Psi_{m}(Y_{m})=&\q^{-(l+1)-(l+j+2)(m-l-j-2)}\Big[(-1)^{l+1}\alpha_{i+1}\otimes
     \tk(f^{m}_{(i+2, m-l-j-2)})\\
     &\qquad\qquad+(-1)^{m-l-j}\ok(f^{m}_{(i+1, m-l-j-1)})\otimes
     \beta_{m+i-2l-2j-2}\Big];
     \end{aligned}
     $$

\item[$\bullet$] if $X_{m}=\beta_{i}\boti
     \beta_{i-1}\boti\cdots\boti\beta_{i-l}\boti\alpha_{i-l}\boti
     \cdots\boti\alpha_{i-l+j-1}\boti\alpha_{i-l+j}\beta_{i-l+j}\boti\alpha_{i-l+j}$
     $\boti\alpha_{i-l+j+1}$ $\boti\cdots\boti\alpha_{m+i-2l-3}$, then

(1) $j>0$,
     $$
     \Psi_{m}(Y_{m})
     =(-1)^{m-l-j}\q^{-(m-l-j-2)-(l+1)(m-l-1)}\ok(f^{m}_{(i+1, m-l-1)})
     \otimes\beta_{m+i-2l-2};
     $$

(2) $j=0$,
     $$
     \begin{aligned}
     \Psi_{m}(Y_{m})=&\q^{-(m-l-j-2)-(l+1)(m-l-1)}\Big[(-1)^{l+1}\alpha_{i+1}\otimes
     \tk(f^{m}_{(i+2, m-l-2)})\\
     &\qquad\qquad+(-1)^{m-l-j}\ok(f^{m}_{(i+1, m-l-1)})\otimes
     \beta_{m+i-2l-2}\Big];
     \end{aligned}
     $$

\item[$\bullet$] if $X_{m}=\beta_{i}\boti
     \beta_{i-1}\boti\cdots\boti\beta_{i-l}\boti\alpha_{i-l}\beta_{i-l}
     \boti\beta_{i-l-1}\boti\cdots\boti\beta_{i-l-j}\boti\alpha_{i-l-j}
     \boti\cdots\boti\alpha_{i-l-j+s-1}$ $\boti\alpha_{i-l-j+s}\beta_{i-l-j+s}
     \boti\alpha_{i-l-j+s}\boti\cdots\boti\alpha_{m+i-2l-2j-4}$, then
     $$
     \Psi_{m}(Y_{m})
     =(-1)^{m-j-s}\q^{-(l+1)-(j+l+2)(m-j-l-2)-(m-l-j-s-3)}
     \alpha_{i+1}\otimes\beta_{m+i-2j-2l-3};
     $$

\item[$\bullet$] if $X_{m}=\alpha_{i}\boti
     \alpha_{i+1}\boti\cdots\boti\alpha_{m+i-1}$, then
     $$
     \Psi_{m}(Y_{m})=\ok(f^{m}_{(i, m)})\otimes\tk(f^{m}_{(i, m)});
     $$

\item[$\bullet$] if $X_{m}=\alpha_{i}\boti\cdots\boti\alpha_{i+j}\boti\alpha_{i+j+1}
    \beta_{i+j+1}\boti\alpha_{i+j+1}\cdots\alpha_{m+i-2}$, then
     $$
     \Psi_{m}(Y_{m})=(-1)^{(m-j-2)}\q^{-(m-j-2)}\ok(f^{m}_{(i,m)})\otimes\beta_{m+i+1};
     $$

\item[$\bullet$] if $X_{m}=\alpha_{i}
     \beta_{i}\boti\alpha_{i}\boti\cdots\boti\alpha_{m+i-2}$, then
     $$
     \Psi_{m}(Y_{m})=\q^{-(m-1)}\alpha_{i}\otimes\tk(f^{m}_{(i+1, m-1)})
     +(-1)^{(m-1)}q^{-(m-1)}\ok(f^{m}_{(i, m)})\otimes\beta_{m+i+1};
     $$

\item[$\bullet$] if $X_{m}=\alpha_{i}
     \beta_{i}\boti\alpha_{i}\boti\cdots\boti\alpha_{i-l-1}\boti\alpha_{i-l}
     \beta_{i-l}\boti\alpha_{i-l}\boti\cdots\alpha_{m+i-2l-3}$, then
     $$
     \Psi_{m}(Y_{m})=(-1)^{(m-l-2)}\q^{-(m-2)-(m-l-2)-1}\alpha_{i}\otimes\beta_{m+i};
     $$

\item[$\bullet$] if $X_{m}=\alpha_{i}
     \beta_{i}\boti\beta_{i-1}\boti\cdots\boti\beta_{i-l-1}\boti
     \alpha_{i-l-1}\boti\cdots\boti\alpha_{i-l-j}\boti\alpha_{i-l-j+1}
     \beta_{i-l-j+l}\boti\alpha_{i-l-j+1}$ $\boti\cdots\boti\alpha_{m+i-2l-2j-3}$,
     then
     $$
     \Psi_{m}(Y_{m})
     =(-1)^{(m-l-j-3)}\q^{-(m-l-3)-(m-l-j-3)-(l+1)-(m-l-3)(l+1)-1}\alpha_{i}\otimes
     \beta_{m+i};
     $$

\item[$\bullet$] if $X_{m}=\alpha_{i}
     \beta_{i}\boti\beta_{i-1}\boti\cdots\boti\beta_{i-l-1}\boti
     \alpha_{i-l-1}\boti\cdots\boti\alpha_{m+i-2l-4}$, then
     $$
     \Psi_{m}(Y_{m})=\q^{-(m-l-2)(l+2)}\alpha_{i}\otimes\tk(f^{m}_{(i+1,m-l-2)}).
     $$
\end{itemize}
For the other case, $\Psi_{m}(\ok(a_{1})\boti a_{1}\boti\cdots
\boti a_{m}\boti\tk(a_{m}))=0$.

Now, let us consider the BV operator on Hochschild cohomology
ring $HH^{\ast}(A_{\q})$. Here we consider it in different cases,
that is, when $A_{\q}$ is a symmetric algebra and $A_{\q}$ is a
Frobenius algebra with semisimple Nakayama automorphism.


\subsection{Symmetric case}\label{subsec1}
If $\q=-1$, the algebra $A_{\q}$ is a symmetric algebra.
In this case, we apply Tradler's construction to the zigzag
algebra $A_{\q}$, and get the BV algebraic structure of Hochschild cohomology
ring of $A_{\q}$. Let us review Tradler's construction. Let $\Lambda$ be an
associative $\kk$-algebra. It is well-known that
there is a Connes' $\mathfrak{B}$-operator
$\mathfrak{B}: HH_{m}(\Lambda)\rightarrow HH_{m+1}(\Lambda)$
in the Hochschild homology of $\Lambda$ (see \cite{Lo}).
If $\Lambda$ is symmetric, that is, there exists a symmetric
associative non-degenerate bilinear form $\langle\ \;,\;\ \rangle:
\Lambda\times\Lambda\rightarrow\kk$. This bilinear form induces a
duality between the Hochschild cohomology and the Hochschild homology
of $\Lambda$. Via this duality, for $m\geq1$, there is an
operator $\Delta: HH^{m}(\Lambda)\rightarrow HH^{m-1}(\Lambda)$,
which corresponds to the Connes's $\mathfrak{B}$-operator on the
Hochschild homology. Tradler has given the following theorem.

\begin{theorem}(\cite{T}) \label{thm7.1}
Let $\Lambda$ be a symmetric $\kk$-algebra. There exists a
symmetric associative non-degenerate bilinear form
$\langle\ \;,\;\ \rangle: \Lambda\times\Lambda\rightarrow\kk$.
For any $f\in\Hom_{\kk}(\Lambda^{\otimes m}, \Lambda)$, define
$\Delta(f)\in\Hom_{\kk}(\Lambda^{\otimes(m-1)}, \Lambda)$ by
$$
\begin{aligned}
&\langle\Delta(f)(a_{1}\otimes\cdots\otimes a_{m-1}),\; a_{m}\rangle\\
=&\sum_{i=1}^{m}(-1)^{i(m-1)}\langle f(a_{i}\otimes\cdots\otimes a_{m-1}
\otimes a_{m}\otimes a_{1}\otimes\cdots\otimes a_{i-1},\;
1_{\Lambda}\rangle,
\end{aligned}
$$
where $a_{i}\in\Lambda$, $1\leq i\leq m$. Then $\Delta$ induces a
differential $\Delta: HH^{m}(\Lambda)\rightarrow HH^{m-1}(\Lambda)$.
And $(HH^{\ast}(\Lambda),\; \sqcup,\; [\ \;,\;\ ],\; \Delta)$ is
a BV algebra.
\end{theorem}

Recall that $\B=\{e_{i},\, \alpha_{i},\, \beta_{i},\,
\alpha_{i}\beta_{i}\mid i=1, 2\}$ is a $\kk$-basis of $A_{-1}$.
It is well-known that the zigzag algebra $A_{-1}$ is symmetric
with respect to the symmetrizing form
$$
\langle a,\; b\rangle=\left\{\begin{array}{lll}
1,  &\qquad\mbox{ if } b=a^{\ast};\\
0,  &\qquad\mbox{ otherwise},
\end{array}\right.
$$
for any $a, b\in A_{-1}$, where
$$
\begin{tabular*}{10cm}{@{\extracolsep{\fill}}l|llllllllllllr}
$a$ & \; $e_{1}$ & \; $e_{2}$ & $\alpha_{1}$ & $\alpha_{2}$ & $\beta_{1}$
& $\beta_{2}$ & $\alpha_{1}\beta_{1}$ & $\alpha_{2}\beta_{2}$\\
\hline $a^{\ast}$ & $\alpha_{1}\beta_{1}$ & $\alpha_{2}\beta_{2}$
& $\beta_{1}$ & $\beta_{2}$ & $\alpha_{1}$ & $\alpha_{2}$
& \; $e_{1}$ & \; $e_{2}$
\end{tabular*}
$$
Thus, there exists a BV operator $\Delta$
such that $(HH^{\ast}(A_{-1}),\; \sqcup,\; [\ \;,\;\ ],\; \Delta)$
is a BV algebra. Now we use Tradler's construction
to give the BV operator $\Delta$ on
$$
HH^{\ast}(A_{-1})\cong\left(\kk[z_{1}, z_{2}]/(z_{1}^{2},
z_{1}z_{2}, z_{2}^{2})\times_{\kk}\w(u_{1}, u_{2}, u_{3},
u_{4})\right)[w_{0}, w_{1}, w_{2}]/I.
$$
Thanks to the formulas $[f\sqcup g,\; h]=[f,\; h]\sqcup g+
(-1)^{|f|(|h|-1)}f\sqcup[h,\; g]$ and $[f,\, g]=
-(-1)^{(|f|-1)|g|}(\Delta(f\sqcup g)-\Delta(f)\sqcup g
-(-1)^{|f|}f\sqcup\Delta(g))$, we have
$$
\begin{aligned}
\Delta(f\sqcup g\sqcup h)
&=\Delta(f\sqcup g)\sqcup h
+(-1)^{|g||h|}\Delta(f\sqcup h)\sqcup g
+(-1)^{|f|+|g|+|h|}f\sqcup\Delta(g\sqcup h)\\
&\ \ -\Delta(f)\sqcup g\sqcup h
-(-1)^{|f||g|}f\sqcup h\sqcup\Delta(g)
-(-1)^{|f|+|g|+|h|}f\sqcup\Delta(h)\sqcup g.
\end{aligned}
$$
This means that to determine operator $\Delta$,
we only need to calculate $\Delta(a)$ and $\Delta(a\sqcup b)$
for all the generators $a, b$ of $HH^{\ast}(A_{-1})$.
Moreover, using the comparison morphisms $\Psi$ and $\Phi$,
we compute $\Delta(f)$ by formula
$$
\Delta(f)=\Delta(f\circ\Psi_{m})\Phi_{m-1},
$$
for any $f\in HH^{m}(A_{-1})$, where, we equate the elements
in $HH^{\ast}(A_{-1})$ with their representatives for convenience.
Note that the formula in Theorem \ref{thm7.1} is also hold for the
complex induced by the reduced bar resolution, we have
$$
\begin{aligned}
&\Delta(f)(a_{1}\boti\cdots\boti a_{m-1})\\
=&\sum_{b\in\B_{0}}\sum_{i=1}^{m}(-1)^{i(m-1)}\langle f(a_{i}\boti\cdots\boti a_{m-1}
\boti b \boti a_{1}\boti\cdots\boti a_{i-1},\; 1\rangle b^{\ast},
\end{aligned}
$$
for any $f\in HH^{m}(A_{-1})$, $b, a_{1},\cdots,a_{m-1}\in A_{-1}$,
where $\B_{0}=\{\alpha_{i},\, \beta_{i},\, \alpha_{i}\beta_{i}\mid i=1, 2\}$
and $1=e_{1}+e_{2}$.

Next, we will discuss $\Delta(f)$ for different degree of $f$.
First, $\Delta(f)=0$ for $f\in\{z_{1}, z_{2}\}$, since $\Delta$
is degree $-1$. Second, if $f$ is degree 1, that is,
$f\in\{u_{i}\mid 1\leq i\leq 4\}\cup\{z_{i}u_{j}\mid i=1, 2, 1\leq j\leq 4\}$,
we have the following lemma.

\begin{lemma} \label{lem6.2}
For the BV operator $\Delta$ on $HH^{\ast}(A_{-1})$, we have $\Delta(f)=0$
for any $f\in HH^{\ast}(A_{-1})$ with degree 1, except
$\Delta(u_{2})=\Delta(u_{3})=1$.
\end{lemma}

\begin{proof}
Note that
$$
\begin{aligned}
\Delta(f)(e_{i}\otimes e_{i})
&=\Delta(f\circ\Psi_{1})\circ\Phi_{0}(e_{i}\otimes e_{i})\\
&=\sum_{a\in\B_{0}}\langle f\circ\Psi_{1}(\ok(a)\boti a\boti\tk(a)),\;
1\rangle a^{\ast}\\
&=\langle f(\ok(\alpha_{i})\otimes\tk(\alpha_{i})),\; 1\rangle\beta_{i}
+\langle f(\ok(\beta_{i})\otimes\tk(\beta_{i})),\; 1\rangle\alpha_{i}\\
&+\langle f(\alpha_{i}\otimes\tk(\beta_{i})+\ok(\alpha_{i})
\otimes\beta_{i}),\;  1\rangle e_{i}
\end{aligned}
$$
for any $f\in HH^{\ast}(A_{-1})$, we get $\Delta(u_{1})(e_{i}\otimes
e_{i})=\Delta(u_{4})(e_{i}\otimes e_{i})=0$, $\Delta(u_{2})
(e_{i}\otimes e_{i})=\Delta(u_{3})(e_{i}\otimes e_{i})=e_{i}$.
Thus, $\Delta(u_{1})=\Delta(u_{4})=0$ and
$\Delta(u_{2})=\Delta(u_{3})=1$. Moreover, we have
$\Delta(z_{i}u_{j})=0$ since $z_{i}u_{j}=0$ for all $i=1, 2$, $1\leq j\leq 4$.
\end{proof}

Third, whenever $f\in HH^{\ast}(A_{-1})$ is degree 2, that is,
$f\in\{w_{i}\mid i=0, 1, 2\}\cup\{z_{i}w_{j}\mid i=1, 2, j=0, 1, 2\}
\cup\{u_{i}u_{j}\mid 1\leq i, j\leq 4\}$, we have the following lemma.

\begin{lemma} \label{lem6.3}
For the BV operator $\Delta$ on $HH^{\ast}(A_{-1})$, we have $\Delta(f)=0$
for any $f\in HH^{\ast}(A_{-1})$ with degree 2, except
\begin{alignat*}{3}
&\Delta(z_{1}w_{0})=2u_{1},&\qquad &\Delta(z_{1}w_{1})=u_{2}-u_{3},
&\qquad &\Delta(z_{1}w_{2})=2u_{4},\\
&\Delta(u_{1}u_{2})=-2u_{1}, & &\Delta(u_{1}u_{4})=u_{2}-u_{3}, &
&\Delta(u_{2}u_{3})=u_{3}-u_{2},\\
&\Delta(u_{3}u_{4})=2u_{4}.& & & &
\end{alignat*}
\end{lemma}

\begin{proof}
Using the formula above, for any $f\in HH^{2}(A_{-1})$, we get
$$
\begin{aligned}
&\Delta(f)(\ok(\alpha_{i})\otimes\tk(\alpha_{i}))\\
=&\Delta(f\circ\Psi_{2})(\ok(\alpha_{i})\boti\alpha_{i}
\boti\tk(\alpha_{i}))\\
=&\sum_{a\in\B_{0}}\langle f\circ\Psi_{2}(\ok(a)\boti a\boti\alpha_{i}
\boti\tk(\alpha_{i})-\ok(\alpha_{i})\boti\alpha_{i}
\boti a\boti\tk(a)),\; 1\rangle a^{\ast}\\
=&\langle f(\ok(f^{2}_{(i+1,2)})\otimes\tk(f^{2}_{(i+1,2)})),\;
1\rangle\beta_{i+1}-\langle f(\ok(f^{2}_{(i+1,1)})\otimes
\tk(f^{2}_{(i+1,1)})),\; 1\rangle\alpha_{i}\\
&-\langle f(\alpha_{i}\otimes
\tk(f^{2}_{(i+1,1)})),\; 1\rangle e_{i}+\langle f(\ok(f^{2}_{(i,2)})
\otimes\beta_{i+1}),\; 1\rangle e_{i}\\
&-\langle f(\ok(f^{2}_{(i,2)})\otimes\tk(f^{2}_{(i,2)})),\;
1\rangle\beta_{i+1}-\langle f(\ok(f^{2}_{(i,2)})\otimes
\beta_{i+1}),\; 1\rangle e_{i+1}.
\end{aligned}
$$
Similar calculations we yield
$$
\begin{aligned}
&\Delta(f)(\ok(\beta_{i})\otimes\tk(\beta_{i}))\\
=&\langle f(\ok(f^{2}_{(i,0)})\otimes\tk(f^{2}_{(i,0)})),\;
1\rangle\alpha_{i+1}+\langle f(\alpha_{i+1}
\otimes\tk(f^{2}_{(i,0)})),\; 1\rangle e_{i+1}\\
&+\langle f(\ok(f^{2}_{(i+1,1)})\otimes\tk(f^{2}_{(i+1,1)})),\;
1\rangle\beta_{i}-\langle f(\ok(f^{2}_{(i+1,0)})\otimes
\tk(f^{2}_{(i+1,0)})),\; 1\rangle\alpha_{i+1}\\
&-\langle f(\alpha_{i+1}\otimes
\tk(f^{2}_{(i,0)})),\; 1\rangle e_{i}+\langle f(\ok(f^{2}_{(i+1,1)})
\otimes\beta_{i}),\; 1\rangle e_{i}.
\end{aligned}
$$
Taking $f=u_{1}u_{2}$, we get $\Delta(u_{1}u_{2})(\ok(\alpha_{i})
\otimes\tk(\alpha_{i}))=0$ and $\Delta(u_{1}u_{2})(\ok(\beta_{i})
\otimes\tk(\beta_{i}))=(-1)^{i}2\alpha_{i+1}$. That is to say,
$\Delta(u_{1}u_{2})=-2u_{1}$. Similarly, we have $\Delta(z_{1}w_{0})=2u_{1}$.
Taking $f=u_{1}u_{4}$, we have $\Delta(u_{1}u_{4})(\ok(\alpha_{i})
\otimes\tk(\alpha_{i}))=-\alpha_{i}$ and $\Delta(u_{1}u_{4})
(\ok(\beta_{i})\otimes\tk(\beta_{i}))=\beta_{i}$, and so that
$\Delta(u_{1}u_{4})=u_{2}-u_{3}$. Similarly, $\Delta(u_{2}u_{3})=u_{3}-u_{2}$,
$\Delta(z_{1}w_{1})=u_{2}-u_{3}$, $\Delta(u_{3}u_{4})=2u_{4}$,
$\Delta(z_{1}w_{2})=2u_{4}$. Moreover, by direct calculation,
we have $\Delta(u_{1}u_{3})=\Delta(u_{2}u_{4})=\Delta(z_{2}w_{0})
=\Delta(z_{2}w_{1})=\Delta(z_{2}w_{2})=0$ and $\Delta(w_{i})=0$ for $i=0, 1, 2$.
\end{proof}

Fourth, whenever $f\in HH^{\ast}(A_{-1})$ is degree 3, that is,
$f\in\{u_{i}w_{j}\mid 1\leq i\leq 4, j=0, 1, 2\}$, we have the following lemma.

\begin{lemma} \label{lem6.4}
For the BV operator $\Delta$ on $HH^{\ast}(A_{-1})$, we have $\Delta(f)=0$
for any $f\in HH^{\ast}(A_{-1})$ with degree 3, except
\begin{alignat*}{3}
&\Delta(u_{1}w_{1})=-w_{0}, &\qquad &\Delta(u_{1}w_{2})=2w_{1},
& \qquad &\Delta(u_{2}w_{0})=3w_{0},\\
&\Delta(u_{2}w_{1})=2w_{1}, & &\Delta(u_{2}w_{2})=w_{2}, &
&\Delta(u_{3}w_{0})=w_{0},\\
&\Delta(u_{3}w_{1})=2w_{1}, & &\Delta(u_{3}w_{2})=3w_{2}, &
&\Delta(u_{4}w_{0})=-2w_{1},\\
&\Delta(u_{4}w_{1})=w_{2}. & & & &
\end{alignat*}
\end{lemma}

\begin{proof}
If $f\in HH^{3}(A_{-1})$, we have
$$
\begin{aligned}
&\Delta(f)(\ok(f^{2}_{(i,0)})\otimes\tk(f^{2}_{(i,0)}))\\
=&\Delta(f\circ\Psi_{3})(\ok(\beta_{i+1})\boti\beta_{i+1}
\boti\beta_{i}\boti\tk(\beta_{i}))\\
=&\sum_{a\in\B_{0}}\langle f\circ\Psi_{3}(\ok(a)\boti a\boti
\beta_{i+1}\boti\beta_{i}\boti\tk(\beta_{i}),\; 1\rangle a^{\ast}\\
&+\sum_{a\in\B_{0}}\langle f\circ\Psi_{3}(\ok(\beta_{i})
\boti\beta_{i}\boti a\boti\beta_{i+1}\boti\tk(\beta_{i+1})),\;
1\rangle a^{\ast}\\
&+\sum_{a\in\B_{0}}\langle f\circ\Psi_{3}(\ok(\beta_{i+1})
\boti\beta_{i+1}\boti\beta_{i}\boti a\boti\tk(a)),\;
1\rangle a^{\ast}\\
=&\langle f(\ok(f^{3}_{(i+1,0)})\otimes\tk(f^{3}_{(i+1,0)})),\; 1\rangle\alpha_{i}
+\langle f(\alpha_{i}\otimes\tk(f^{3}_{(i+1,0)})),\; 1\rangle e_{i}\\
&+\langle f(\alpha_{i+1}\otimes\tk(f^{3}_{(i,0)})),\; 1\rangle e_{i}
+\langle f(\ok(f^{3}_{(i,1)})\otimes\tk(f^{3}_{(i,1)})),\; 1\rangle\beta_{i}\\
&+\langle f(\ok(f^{3}_{(i,0)})\otimes\tk(f^{3}_{(i,0)})),\;
1\rangle\alpha_{i+1}
+\langle f(\alpha_{i}\otimes\tk(f^{3}_{(i+1,0)})),\; 1\rangle e_{i}\\
&+\langle f(\ok(f^{3}_{(i,1)})\otimes\beta_{i}),\; 1\rangle e_{i},
\end{aligned}
$$
$$
\begin{aligned}
&\Delta(f)(\ok(f^{2}_{(i,1)})\otimes\tk(f^{2}_{(i,1)}))\\
=&\q\langle f(\ok(f^{3}_{(i+1,1)})\otimes\tk(f^{3}_{(i+1,1)})),\; 1\rangle\alpha_{i}
+\q\langle f(\alpha_{i}\otimes\tk(f^{3}_{(i+1,1)})),\; 1\rangle e_{i}\\
&+\langle f(\alpha_{i+1}\otimes\tk(f^{3}_{(i,1)})),\; 1\rangle e_{i}
+\langle f(\ok(f^{3}_{(i+1,2)})\otimes\beta_{i+1}),\; 1\rangle e_{i}\\
&+\q\langle f(\ok((f^{3}_{(i,2)})\otimes\tk(f^{3}_{(i,2)})),\; 1\rangle\beta_{i}
+\q\langle f(\ok(f^{3}_{(i,2)})\otimes\beta_{i}),\; 1\rangle e_{i},
\end{aligned}
$$
and
$$
\begin{aligned}
&\Delta(f)(\ok(f^{2}_{(i,2)})\otimes\tk(f^{2}_{(i,2)}))\\
=&\langle f(\ok(f^{3}_{(i+1,3)})\otimes\tk(f^{3}_{(i+1,3)})),\; 1\rangle\beta_{i+1}
+\langle f(\ok(f^{3}_{(i+1,2)})\otimes\tk(f^{3}_{(i+1,2)})),\; 1\rangle\alpha_{i}\\
&+\langle f(\ok(f^{3}_{(i+1,3)})\otimes\beta_{i+1}),\; 1\rangle e_{i}+\langle
f(\alpha_{i}\otimes\tk(f^{3}_{(i+1,2)})),\; 1\rangle e_{i} \\
&+\langle f(\ok(f^{3}_{(i,3)})\otimes\tk(f^{3}_{(i,3)})),\;
1\rangle\beta_{i}+\langle f(\ok(f^{3}_{(i,3)})\otimes\beta_{i}),\; 1\rangle e_{i}\\
&+\langle f(\ok(f^{3}_{(i+1,3)})\otimes\beta_{i+1}),\; 1\rangle e_{i}.
\end{aligned}
$$
Taking $f=u_{1}w_{1}$, we have $\Delta(u_{1}w_{1})(\ok(f^{2}_{(i,0)})
\otimes\tk(f^{2}_{(i,0)})=-e_{i}$, $\Delta(u_{1}w_{1})
(\ok(f^{2}_{(i,1)})\otimes\tk(f^{2}_{(i,1)}))=0$ and $\Delta(u_{1}w_{1})(\ok(f^{2}_{(i,2)})
\otimes\tk(f^{2}_{(i,2)}))=0$. That is to say, $\Delta(u_{1}w_{1})=-w_{0}$.
Similarly, we also get $\Delta(u_{1}w_{2})=\Delta(u_{3}w_{1})=\Delta(u_{2}w_{1})=2w_{1}$,
$\Delta(u_{2}w_{2})=\Delta(u_{4}w_{1})=w_{2}$, $\Delta(u_{2}w_{0})=
3\Delta(u_{3}w_{0})=3w_{0}$, $\Delta(u_{3}w_{2})=3w_{2}$, $\Delta(u_{4}w_{0})=-2w_{1}$
and $\Delta(u_{1}w_{0})=\Delta(u_{4}w_{2})=0$.
\end{proof}

Fourth, whenever $f\in HH^{\ast}(A_{-1})$ is degree 3, let
$f\in\{u_{i}w_{j}\mid 1\leq i\leq 4, j=0, 1, 2\}$, we have the following lemma.

\begin{lemma} \label{lem6.5}
For the BV operator $\Delta$ on $HH^{\ast}(A_{-1})$, we have $\Delta(f)=0$
for all $f\in HH^{\ast}(A_{-1})$ with degree 4.
\end{lemma}

Now we can get the BV operator on Hochschild cohomology ring of the zigzag algebra
$A_{-1}$ completely.

\begin{theorem}\label{thm6.6}
Let $A_{-1}$ be the zigzag algebra of type $\widetilde{\mathbf{A}}_{1}$.
Then the BV operator on Hochschild cohomology ring
$$
HH^{\ast}(A_{-1})\cong
\left(\kk[z_{1}, z_{2}]/(z_{1}^{2}, z_{1}z_{2}, z_{2}^{2})\times_{\kk}
\w(u_{1}, u_{2}, u_{3}, u_{4})\right)[w_{0}, w_{1}, w_{2}]/I,
$$
is zero on homogeneous generators and their product except:
\begin{alignat*}{3}
&\Delta(u_{2})=\Delta(u_{3})=1,&\qquad &\Delta(z_{1}w_{0})=2u_{1},
&\qquad &\Delta(z_{1}w_{1})=u_{2}-u_{3},\\
&\Delta(z_{1}w_{2})=2u_{4}, & &\Delta(u_{1}u_{2})=-2u_{1}, &
&\Delta(u_{1}u_{4})=u_{2}-u_{3},\\
&\Delta(u_{2}u_{3})=u_{3}-u_{2}, & &\Delta(u_{3}u_{4})=2u_{4}, &
&\Delta(u_{1}w_{1})=-w_{0},\\
&\Delta(u_{1}w_{2})=2w_{1}, & &\Delta(u_{2}w_{0})=3w_{0}, &
&\Delta(u_{2}w_{1})=2w_{1},\\
&\Delta(u_{2}w_{2})=w_{2}, & &\Delta(u_{3}w_{0})=w_{0}, &
&\Delta(u_{3}w_{1})=2w_{1},\\
&\Delta(u_{3}w_{2})=3w_{2}, & &\Delta(u_{4}w_{0})=-2w_{1}, &
&\Delta(u_{4}w_{1})=w_{2},
\end{alignat*}
where the ideal $I$ is given by
$$
\begin{aligned}
&\big(u_{1}u_{3},\; u_{2}u_{4},\; z_{2}w_{0},\; z_{2}w_{1},\;
u_{1}u_{2}+z_{1}w_{0},\; u_{1}u_{4}-z_{1}w_{1},\; u_{1}u_{4}-u_{3}u_{2},\;
u_{3}u_{4}-z_{1}w_{2},\\
&\qquad z_{2}w_{2},\; u_{1}w_{1}+u_{3}w_{0},\; u_{1}w_{2}-u_{3}w_{1},\;
u_{2}w_{1}+u_{4}w_{0},\; u_{2}w_{2}-u_{4}w_{1},\;
w_{1}^{2}+w_{0}w_{2}\big).
\end{aligned}
$$
\end{theorem}

\begin{proof}
The theorem follows from
Lemmas \ref{lem6.2}$-$\ref{lem6.5}.
\end{proof}

Using the BV operator $\Delta$ on $HH^{\ast}(A_{-1})$,
we can determine the Gerstenhaber bracket $[\ \,,\,\ ]$ on
$HH^{\ast}(A_{-1})$ by setting
$$
[\alpha,\, \beta]=(-1)^{|\alpha||\beta|+|\alpha|+|\beta|}
\left((-1)^{|\alpha|+1}\Delta(\alpha\sqcup\beta)
+(-1)^{|\alpha|}\Delta(\alpha)\sqcup\beta+\alpha\sqcup\Delta(\beta)\right),
$$
for any homogeneous elements $\alpha, \beta\in HH^{\ast}(A_{-1})$.
Then the Gerstenhaber algebraic structure on $HH^{\ast}(A_{-1})$
can be induced.

\begin{corollary}\label{cor6.7}
Let $A_{-1}$ be the zigzag algebra of type $\widetilde{\mathbf{A}}_{1}$.
The Gerstenhaber algebra
$(HH^{\ast}(A_{-1}),\, \sqcup,\, [\ \,,\,\ ])$ is isomorphic to
$$
\left(\kk[z_{1}, z_{2}]/(z_{1}^{2},
z_{1}z_{2}, z_{2}^{2})\times_{\kk}\w(u_{1}, u_{2}, u_{3},
u_{4})\right)[w_{0}, w_{1}, w_{2}]/I,
$$
where the Gerstenhaber bracket is zero for all pairs of
homogeneous generators except:
\begin{alignat*}{3}
&[z_{1}, u_{2}]=[z_{1}, u_{3}]=-z_{1},&\qquad &[z_{2}, u_{2}]=[z_{2}, u_{3}]=-z_{2},&
\qquad&[z_{1}, w_{0}]=-2u_{1},\\
&[z_{1}, w_{1}]=u_{3}-u_{2},& &[z_{1}, w_{2}]=-2u_{4},&
&[u_{1}, u_{2}]=u_{1}, \\
&[u_{1}, u_{3}]=-u_{1}, & &[u_{1}, u_{4}]=u_{3}-u_{2},&
&[u_{2}, u_{4}]=u_{4},\\
&[u_{3}, u_{4}]=-u_{4},& &[u_{1}, w_{1}]=w_{0},&
&[u_{1}, w_{2}]=-2w_{1},\\
&[u_{2}, w_{0}]=-2w_{0},& &[u_{2}, w_{1}]=-w_{1},&
&[u_{3}, w_{1}]=-w_{1},\\
&[u_{3}, w_{2}]=-2w_{2},& &[u_{4}, w_{0}]=2w_{1},&
&[u_{4}, w_{1}]=-w_{2},
\end{alignat*}
the ideal $I$ is given by
$$
\begin{aligned}
&\big(u_{1}u_{3},\; u_{2}u_{4},\; z_{2}w_{0},\; z_{2}w_{1},\;
u_{1}u_{2}+z_{1}w_{0},\; u_{1}u_{4}-z_{1}w_{1},\; u_{1}u_{4}-u_{3}u_{2},\;
u_{3}u_{4}-z_{1}w_{2},\\
&\qquad z_{2}w_{2},\; u_{1}w_{1}+u_{3}w_{0},\; u_{1}w_{2}-u_{3}w_{1},\;
u_{2}w_{1}+u_{4}w_{0},\; u_{2}w_{2}-u_{4}w_{1},\;
w_{1}^{2}+w_{0}w_{2}\big).
\end{aligned}
$$
\end{corollary}


\subsection{Frobenius case}\label{subsec2}

If $\q\neq-1$, then $A_{\q}$ is not symmetric. But $A_{\q}$ is a Frobenius
algebra. We define bilinear form by
$$
\langle a,\; b\rangle=\left\{\begin{array}{ll}
1,& \qquad\mbox{if }ab=\alpha_{1}\beta_{1} \mbox{ or } \alpha_{2}\beta_{2};\\
0,& \qquad\mbox{otherwise}.
\end{array}\right.
$$
Then the corresponding semisimple Nakayama automorphism $\nu$ is given by
$$
\begin{tabular*}{10cm}{@{\extracolsep{\fill}}l|lllllllr}
$a\in\B$ & $e_{1}$ &  $e_{2}$ & $\quad \alpha_{1}$ & $\quad \alpha_{2}$ &
$\quad \beta_{1}$ & $\quad \beta_{2}$ & $\alpha_{1}\beta_{1}$  & $\alpha_{2}\beta_{2}$\\
\hline $\; \nu(a)$ & $e_{1}$ &  $e_{2}$ & $-\q\alpha_{1}$ & $-\q\alpha_{2}$ &
$-\q^{-1}\beta_{1}$ & $-\q^{-1}\beta_{2}$ & $\alpha_{1}\beta_{1}$  & $\alpha_{2}\beta_{2}$
\end{tabular*}
$$
That is $\langle a,\; b\rangle=\langle b,\; \nu(a)\rangle$, for any
$a, b\in A_{\q}$.
In \cite{LZZ} and \cite{Vo}, the authors proved that the Hochschild
cohomology ring of a Frobenius algebra with semisimple Nakayama
automorphism is a BV algebra in different ways. For the algebra
$A_{\q}$, we can define an automorphism $\tilde{(\;)}$ by
$$
\begin{tabular*}{10cm}{@{\extracolsep{\fill}}l|lllllllr}
$a\in\B$ &\; $e_{1}$ &  $e_{2}$ & $\alpha_{1}$ & $\alpha_{2}$ &
$\quad \beta_{1}$ & $\quad \beta_{2}$ & $\alpha_{1}\beta_{1}$  & $\alpha_{2}\beta_{2}$\\
\hline $\; \quad\tilde{a}$ & $\alpha_{1}\beta_{1}$ &  $\alpha_{2}\beta_{2}$ &
$\beta_{1}$ & $\beta_{2}$ & $-\q\alpha_{1}$ & $-\q\alpha_{2}$ &\; $e_{1}$  & $e_{2}$
\end{tabular*}
$$
then
$$
\langle a,\; b\rangle=\left\{\begin{array}{ll}
1,& \qquad\mbox{if } b=\tilde{a};\\
0,& \qquad\mbox{otherwise},
\end{array}\right.
$$
for $a, b\in \B$. Thus, we can calculate $\Delta(\alpha)$ by
$$
\begin{aligned}
&\Delta(\alpha)(a_1\boti\dots\boti a_{n-1})\\
&=\sum_{b\in\B_{1}}\Big\langle\sum_{i=1}^n(-1)^{i(n-1)}
\alpha(a_{i}\boti\dots\boti a_{n-1}\boti\tilde{b}\boti\nu(a_1)\boti
\dots\boti\nu(a_{i-1})),\; 1\Big\rangle b,
\end{aligned}
$$
for any $\alpha\in HH^{n}(A_{\q})$, where $\B_{1}=\{e_{i},\, \alpha_{i},\,
\beta_{i}\mid i=1, 2\}$.

Now, we can give the BV operator on $HH^{\ast}(A_{\q})$ for
Frobenius algebra $A_{\q}$. Firstly, if $\q$ is not a root of unity,
we get the BV algebraic structure on $HH^{\ast}(A_{\q})$ as the following
theorem and corollary by direct calculation.

\begin{theorem}\label{thm6.9}
Let $A_{\q}$ be the quantum zigzag algebra, where $\q$ is not a root of unity.
Then the BV operator $\Delta$ on Hochschild cohomology ring
$$
HH^{\ast}(A_{\q})\cong\kk[z_{1}, z_{2}]/(z_{1}^{2}, z_{1}z_{2}, z_{2}^{2})
\times_{\kk}\w(u_{1}, u_{2}),
$$
is zero on homogeneous generators and their product except:
$$
\Delta(u_{1})=\Delta(u_{2})=1, \qquad\qquad
\Delta(u_{1}u_{2})=u_{2}-u_{1}.
$$
\end{theorem}

\begin{proof}
Since $\Delta$ is degree $-1$, $\Delta(f)=0$ for $f\in\{z_{1}, z_{2}\}$.
If $f\in\{u_{j}, z_{i}u_{j}\mid i, j=1, 2\}$, note that
$$
\begin{aligned}
\Delta(f)(e_{i}\otimes e_{i})
&=\Delta(f\circ\Psi_{1})\Phi_{0}(e_{i}\otimes e_{i})\\
&=\sum_{b\in\B_{1}}\langle f\circ\Psi_{1}(\ok(\tilde{b})\boti
\tilde{b}\boti\tk(a)),\; 1\rangle b\\
&=-\q\langle f(\ok(\alpha_{i})\otimes\tk(\alpha_{i})),\; 1\rangle\beta_{i}
+\langle f(\ok(\beta_{i})\otimes\tk(\beta_{i})),\; 1\rangle\alpha_{i}\\
&\qquad+\langle f(\alpha_{i}\otimes\tk(\beta_{i})+\ok(\alpha_{i})
\otimes\beta_{i}),\;  1\rangle e_{i},
\end{aligned}
$$
we get $\Delta(u_{1})(e_{i}\otimes e_{i})=\Delta(u_{2})(e_{i}\otimes e_{i})=e_{i}$,
i.e., $\Delta(u_{1})=\Delta(u_{2})=1$.
Similarly, one can check that $\Delta(z_{i}u_{j})=0$ for any $i, j=1, 2$.
If $f\in HH^{\ast}(A_{\q})$ with degree 2, that is,
$f=u_{1}u_{2}$, then
$$
\begin{aligned}
&\Delta(f)(\ok(\alpha_{i})\otimes\tk(\alpha_{i}))\\
=&\Delta(f\circ\Psi_{2})(\ok(\alpha_{i})\boti\alpha_{i}
\boti\tk(\alpha_{i}))\\
=&\sum_{b\in\B_{1}}\langle f\circ\Psi_{2}(\ok(\tilde{b})\boti\tilde{b}\boti\nu
(\alpha_{i})\boti\tk(\nu(\alpha_{i}))-\ok(\alpha_{i})\boti\alpha_{i}
\boti\tilde{b}\boti\tk(\tilde{b})),\; 1\rangle b\\
=&-\langle f(\alpha_{i}\otimes\tk(f^{2}_{(i+1,1)})-\ok(f^{2}_{(i,2)})
\otimes\beta_{i+1}),\;  1\rangle e_{i}\\
&-\langle f(\ok(f^{2}_{(i+1,1)}\otimes\tk(f^{2}_{(i+1,1)})),\;
1\rangle\alpha_{i}+\q^{2}\langle f(\ok(f^{2}_{(i+1,2)}\otimes
\tk(f^{2}_{(i+1,2)})),\; 1\rangle\beta_{i+1}\\
&-\langle f(\ok(f^{2}_{(i,2)}\otimes\beta_{i+1}),\; 1\rangle e_{i+1}
+\q\langle f(\ok(f^{2}_{(i,2)}\otimes\tk(f^{2}_{(i,2)})),\; 1\rangle\beta_{i+1},
\end{aligned}
$$
and
$$
\begin{aligned}
&\Delta(f)(\ok(\beta_{i})\otimes\tk(\beta_{i}))\\
=&-\q^{-1}\langle f(\alpha_{i-1}\otimes\tk(f^{2}_{(i,0)}),\;  1\rangle
e_{i}-\q^{-1}\langle f(\ok(f^{2}_{(i,0)}\otimes\tk(f^{2}_{(i,0)}),\;  1\rangle \alpha_{i+1}\\
&+\q^{-1}\langle f(\ok(\alpha_{i+1}\otimes\tk(f^{2}_{(i+2,0)})
-\ok(f^{2}_{(i+1,1)}\otimes\beta_{i})),\; 1\rangle e_{i}\\
&-\langle f(\ok(f^{2}_{(i+1,0)}\otimes\tk(f^{2}_{(i+1,0)}),\; 1\rangle\alpha_{i+1}
+\langle f(\ok(f^{2}_{(i+1,1)}\otimes\tk(f^{2}_{(i+,1)})),\; 1\rangle\beta_{i}.
\end{aligned}
$$
That is, $\Delta(u_{1}u_{2})(\ok(\alpha_{i})\otimes\tk(\alpha_{i}))=\alpha_{i}$ and
$\Delta(u_{1}u_{2})(\ok(\beta_{i})\otimes\tk(\beta_{i}))=-\beta_{i}$. Thus,
$\Delta(u_{1}u_{2})=u_{2}-u_{1}$. The proof is complete.  \qed
\end{proof}

\begin{corollary} \label{cor6.10}
Let $A_{\q}$ be the quantum zigzag algebra, where $\q$ is not root of unity.
Then the BV algebra
$(HH^{\ast}(A_{\q})$,\, $\sqcup,\, [\ \,,\,\ ],\, \Delta)$ is isomorphic to
$$
\kk[z_{1}, z_{2}]/(z_{1}^{2}, z_{1}z_{2}, z_{2}^{2})
\times_{\kk}\w(u_{1}, u_{2}),
$$
where the Gerstenhaber bracket is zero for all pairs of
homogeneous generators except:
$$
[z_{i}, u_{j}]=-z_{i},   \qquad\mbox{ for any } 1\leq i, j\leq 2,
$$
and the BV operator is zero on homogeneous generators and their product except:
$$
\Delta(u_{1})=\Delta(u_{2})=1, \qquad\qquad
\Delta(u_{1}u_{2})=u_{2}-u_{1}.
$$
\end{corollary}

Similarly, for the case the case of $\q=1$, we have the following theorem.

\begin{theorem}\label{thm6.11}
The BV algebra
$(HH^{\ast}(A_{1})$,\, $\sqcup,\, [\ \,,\,\ ],\, \Delta)$ is isomorphic to
$$
\left(\kk[z_{1}, z_{2}]/(z_{1}^{2},
z_{1}z_{2}, z_{2}^{2})\times_{\kk}\w(u_{1}, u_{2}, u_{3},
u_{4})\right)[w_{0}, w_{1}, w_{2}]/I,
$$
where the Gerstenhaber bracket is zero for all pairs of
homogeneous generators except:
\begin{alignat*}{3}
&[z_{1}, u_{2}]=[z_{1}, u_{3}]=-z_{1},&\qquad &[z_{2}, u_{2}]=[z_{2}, u_{3}]=-z_{2},&
\qquad&[z_{1}, w_{0}]=2u_{1},\\
&[z_{1}, w_{1}]=u_{3}-u_{2},& &[z_{1}, w_{2}]=-2u_{4},&
&[u_{1}, u_{2}]=u_{1}, \\
&[u_{1}, u_{3}]=-u_{1}, & &[u_{1}, u_{4}]=u_{3}-u_{2},&
&[u_{2}, u_{4}]=u_{4},\\
&[u_{3}, u_{4}]=-u_{4},& &[u_{1}, w_{1}]=-w_{0},&
&[u_{1}, w_{2}]=-2w_{1},\\
&[u_{2}, w_{0}]=-2w_{0},& &[u_{2}, w_{1}]=-w_{1},&
&[u_{3}, w_{1}]=-w_{1},\\
&[u_{3}, w_{2}]=-2w_{2},& &[u_{4}, w_{0}]=-2w_{1},&
&[u_{4}, w_{1}]=-w_{2},
\end{alignat*}
the BV operator is zero on homogeneous generators and their product except:
\begin{alignat*}{3}
&\Delta(u_{2})=\Delta(u_{3})=1, &\qquad &\Delta(z_{1}w_{0})=-2u_{1},
&\qquad &\Delta(z_{1}w_{1})=u_{2}-u_{3},\\
&\Delta(z_{1}w_{2})=2u_{4},& &\Delta(u_{1}u_{2})=-2u_{1},&
&\Delta(u_{1}u_{4})=u_{2}-u_{3},\\
&\Delta(u_{2}u_{3})=u_{3}-u_{2}, & &\Delta(u_{3}u_{4})=2u_{4}, &
&\Delta(u_{1}w_{1})=w_{0},\\
&\Delta(u_{1}w_{2})=2w_{1},& &\Delta(u_{2}w_{0})=3w_{0},&
&\Delta(u_{2}w_{1})=2w_{1},\\
&\Delta(u_{2}w_{2})=w_{2}, & &\Delta(u_{3}w_{0})=w_{0}, &
&\Delta(u_{3}w_{1})=2w_{1},\\
&\Delta(u_{3}w_{2})=3w_{2}, & &\Delta(u_{4}w_{0})=2w_{1}, &
&\Delta(u_{4}w_{1})=w_{2},
\end{alignat*}
where the ideal $I$ is given by
$$
\begin{aligned}
&\big(u_{1}u_{3},\; u_{2}u_{4},\; z_{2}w_{0},\; z_{2}w_{1},\;
u_{1}u_{2}-z_{1}w_{0},\; u_{1}u_{4}-z_{1}w_{1},\; u_{1}u_{4}-u_{3}u_{2},\;
u_{3}u_{4}-z_{1}w_{2},\\
&\qquad z_{2}w_{2},\; u_{1}w_{1}-u_{3}w_{0},\; u_{1}w_{2}-u_{3}w_{1},\;
u_{2}w_{1}-u_{4}w_{0},\; u_{2}w_{2}-u_{4}w_{1},\;
w_{1}^{2}-w_{0}w_{2}\big).
\end{aligned}
$$
\end{theorem}

Finally, for the case where $\q$ a primite $s$-th $(s\geq3)$ root of unity, similar to
the discussion and calculation of case $\q=1$, we can obtain the following
theorem.

\begin{theorem}\label{thm6.13}
Let $A_{\q}$ be the quantum zigzag algebra, where $\q$ is a primite $s$-th
$(s>2)$ root of unity.

(1) If $s$ is odd, then the BV algebra
$(HH^{\ast}(A_{\q})$,\, $\sqcup,\, [\ \,,\,\ ],\, \Delta)$ is isomorphic to
$$
\kk[z_{1}, z_{2}]/(z_{1}^{2}, z_{1}z_{2}, z_{2}^{2})\times_{\kk}
\left(\w(u_{1}, u_{2})[w_{0}, w_{1}, w_{2}]/(w_{1}^{2}-w_{0}w_{2})\right),
$$
where the Gerstenhaber bracket is zero for all pairs of
homogeneous generators except:
\begin{alignat*}{3}
&[z_{1}, u_{1}]=[z_{1}, u_{2}]=-z_{1},&\qquad &[z_{2}, u_{1}]=[z_{2}, u_{2}]=-z_{2},&
\qquad&[u_{1}, w_{0}]=-2s w_{0},\\
&[u_{1}, w_{1}]=-s w_{1},& &[u_{2}, w_{1}]=-s w_{1},&
&[u_{2}, w_{2}]=-2s w_{2},
\end{alignat*}
and the BV operator is zero on homogeneous generators and their product except:
\begin{alignat*}{3}
&\Delta(u_{1})=1, &\qquad &\Delta(u_{2})=1, &
\qquad &\Delta(u_{1}u_{2})=u_{2}-u_{1},\\
&\Delta(u_{1}w_{0})=(2s+1)w_{0}, & &\Delta(u_{1}w_{1})=(s+1)w_{1}, &
&\Delta(u_{1}w_{2})=w_{2},\\
&\Delta(u_{2}w_{0})=w_{0}, & &\Delta(u_{2}w_{1})=(s+1)w_{1}, &
&\Delta(u_{2}w_{2})=(2s+1)w_{2}.
\end{alignat*}

(2) If $s$ is even, then the BV algebra
$(HH^{\ast}(A_{\q})$,\, $\sqcup,\, [\ \,,\,\ ],\, \Delta)$ is isomorphic to
$$
\kk[z_{1}, z_{2}]/(z_{1}^{2}, z_{1}z_{2}, z_{2}^{2})\times_{\kk}
\left(\w(u_{1}, u_{2})[w_{0}, w_{1}, w_{2}]/(w_{1}^{2}
-\q^{\frac{s^{2}}{4}}w_{0}w_{2})\right),
$$
where the Gerstenhaber bracket is zero for all pairs of
homogeneous generators except:
\begin{alignat*}{3}
&[z_{1}, u_{1}]=[z_{1}, u_{2}]=-z_{1},&\qquad &[z_{2}, u_{1}]=[z_{2}, u_{2}]=-z_{2},&
\qquad&[u_{1}, w_{0}]=-s w_{0},\\
&[u_{1}, w_{1}]=-\frac{s}{2} w_{1},& &[u_{2}, w_{1}]=-\frac{s}{2}w_{1},&
&[u_{2}, w_{2}]=-s w_{2},
\end{alignat*}
and the BV operator is zero on homogeneous generators and their product except:
\begin{alignat*}{3}
&\Delta(u_{1})=1, &\qquad &\Delta(u_{2})=1, &
\qquad &\Delta(u_{1}u_{2})=u_{2}-u_{1},\\
&\Delta(u_{1}w_{0})=(s+1)w_{0}, & &\Delta(u_{1}w_{1})=(\frac{s}{2}+1)w_{1}, &
&\Delta(u_{1}w_{2})=w_{2},\\
&\Delta(u_{2}w_{0})=w_{0}, & &\Delta(u_{2}w_{1})=(\frac{s}{2}+1)w_{1}, &
&\Delta(u_{2}w_{2})=(s+1)w_{2}.
\end{alignat*}
\end{theorem}

\begin{proof}
Suppose that $s$ is odd. Then $\Delta(f)=0$ for $f\in\{z_{1}, z_{2}\}$.
If $f\in\{u_{j}, z_{i}u_{j}\mid i, j=1, 2\}$, note that
$$
\begin{aligned}
\Delta(f)(e_{i}\otimes e_{i})
&=\Delta(f\circ\Psi_{1})\Phi_{0}(e_{i}\otimes e_{i})\\
&=\sum_{b\in\B_{1}}\langle f\circ\Psi_{1}(\ok(\tilde{b})\boti
\tilde{b}\boti\tk(\tilde{b})),\; 1\rangle b\\
&=-\q\langle f(\ok(\alpha_{i})\otimes\tk (\alpha_{i})),\; 1\rangle\beta_{i}
+\langle f(\ok(\beta_{i})\otimes\tk (\beta_{i})),\; 1\rangle\alpha_{i}\\
&\qquad+\langle f(\alpha_{i}\otimes\tk(\beta_{i})+\ok(\alpha_{i})
\otimes\beta_{i}),\;  1\rangle e_{i},
\end{aligned}
$$
we get $\Delta(u_{1})(e_{i}\otimes e_{i})=\Delta(u_{2})(e_{i}\otimes e_{i})=e_{i}$,
i.e., $\Delta(u_{1})=\Delta(u_{2})=1$.
Similarly, one can check that $\Delta(z_{i}u_{j})=0$ for any $i, j=1, 2$.

If $f\in HH^{\ast}(A_{\q})$ with degree 2, that is, $f=u_{1}u_{2}$, then
$$
\begin{aligned}
&\Delta(f)(\ok(\alpha_{i})\otimes\tk (\alpha_{i}))\\
=&\Delta(f\circ\Psi_{2})(\ok(\alpha_{i})\boti\alpha_{i}
\boti\tk (\alpha_{i}))\\
=&\sum_{b\in\B_{1}}\langle f\circ\Psi_{2}( \ok
(\tilde{b})\boti\tilde{b}\boti\nu
(\alpha_{i})\boti\tk (\nu(\alpha_{i}))- \ok (\alpha_{i})\boti\alpha_{i}
\boti\tilde{b}\boti\tk (\tilde{b})),\; 1\rangle b\\
=&-\langle f(\alpha_{i}\otimes\tk (f^{2}_{(i+1,1)})- \ok (f^{2}_{(i,2)})
\otimes\beta_{i+1}),\;  1\rangle e_{i}\\
&-\langle f( \ok (f^{2}_{(i+1,1)}\otimes\tk (f^{2}_{(i+1,1)})),\;
1\rangle\alpha_{i}+\q^{2}\langle f( \ok (f^{2}_{(i+1,2)}\otimes
\tk (f^{2}_{(i+1,2)})),\; 1\rangle\beta_{i+1}\\
&-\langle f( \ok (f^{2}_{(i,2)}\otimes\beta_{i+1}),\; 1\rangle e_{i+1}
+\q\langle f( \ok (f^{2}_{(i,2)}\otimes\tk (f^{2}_{(i,2)})),\; 1\rangle\beta_{i+1},
\end{aligned}
$$
and
$$
\begin{aligned}
&\Delta(f)( \ok (\beta_{i})\otimes\tk (\beta_{i}))\\
=&-\q^{-1}\langle f(\alpha_{i-1}\otimes\tk (f^{2}_{(i,0)}),\;  1\rangle
e_{i}-\q^{-1}\langle f( \ok (f^{2}_{(i,0)}\otimes\tk (f^{2}_{(i,0)}),\;  1\rangle \alpha_{i+1}\\
&+\q^{-1}\langle f( \ok (\alpha_{i+1}\otimes\tk (f^{2}_{(i+2,0)})
- \ok (f^{2}_{(i+1,1)}\otimes\beta_{i})),\; 1\rangle e_{i}\\
&-\langle f( \ok (f^{2}_{(i+1,0)}\otimes\tk (f^{2}_{(i+1,0)}),\; 1\rangle\alpha_{i+1}
+\langle f( \ok (f^{2}_{(i+1,1)}\otimes\tk (f^{2}_{(i+,1)})),\; 1\rangle\beta_{i}.
\end{aligned}
$$
That is, $\Delta(u_{1}u_{2})(\ok(\alpha_{i})\otimes\tk (\alpha_{i}))=\alpha_{i}$
and $\Delta(u_{1}u_{2})( \ok (\beta_{i})\otimes\tk (\beta_{i}))=-\beta_{i}$.
Hence, $\Delta(u_{1}u_{2})=u_{2}-u_{1}$.
If $f\in\{z_{i}w_{j},w_{j}\mid i=1, 2,\, j=0, 1, 2\}$ or
$f\in\{w_{i}w_{j}\mid i,\, j=0,1, 2\}$, by direct calculation, we get $\Delta(f)=0$.

If $f\in HH^{\ast}(A_{\q})$ with degree $2s+1$, that is, $f\in\{u_{i}w_{j}\mid i=1,2,\,
j=0,1, 2\}$. Denote $\chi_{k}:=a_{k}\otimes\cdots\otimes a_{2s}\otimes\tilde{b}\otimes
\nu(a_{1})\otimes\cdots\otimes\nu(a_{k-1})$ for $1\leq k\leq 2s$. Note that
$\Big\langle u_{i}w_{j}\circ\Psi_{2s+1}(\ok(a_{k})\otimes\chi_{k}\otimes
\tk(\nu(a_{k-1}))),\; 1\Big\rangle=0$ except
\begin{itemize}
\item[(i)] if $\chi_{k}=\beta_{k}\otimes\cdots\otimes\beta_{k-l+1}\otimes
    \alpha_{k-1+1}\beta_{k-1+1}\otimes\nu(\beta_{k-l})\otimes\cdots\otimes
    \nu(\beta_{k+2s-js-2l-1})\otimes\nu(\alpha_{k+2s-js-2l-1})\otimes\cdots
    \otimes\nu(\alpha_{k+2s-2l-2})$,
$$
\Big\langle
u_{1}w_{j}\circ\Psi_{2s+1}(\ok(a_{l})\otimes\chi_{l}\otimes\tk(\nu(a_{l-1}))),\;
1\Big\rangle=1;
$$
\item[(ii)] if $\chi_{k}=\beta_{k}\otimes\cdots\otimes\beta_{js+k-1}\otimes
    \alpha_{js+k-1}\otimes\cdots\cdots\otimes\alpha_{js+k-l}\otimes\alpha_{js+k-l-1}
    \beta_{js+k-l-1}\otimes$
    $\nu(\alpha_{js+k-l-1})\otimes\cdots\otimes\nu(\alpha_{2s+k-2l-2})$,
$$
\Big\langle
u_{2}w_{j}\circ\Psi_{2s+1}(\ok(a_{k})\otimes\chi_{k}\otimes
\tk(\nu(a_{k-1}))),\; 1\Big\rangle=1,
$$
\end{itemize}
we get
$$
\begin{aligned}
\Delta(u_{1}w_{j})(\ok(f^{2s}_{(i,j')})\otimes\tk(f^{2s}_{(i,j')}))
&=\left\{\begin{array}{lll}
(2s-js+1)e_{i},& j'=js;\\
0,& \mbox{otherwise},
\end{array}\right.\\
&=\left\{\begin{array}{lll}
e_{i},& j'=j=0;\\
(s+1)e_{i},& j=1,j'=s;\\
(2s+1)e_{i},& j=2,j'=2s;\\
0,& \mbox{otherwise},
\end{array}\right.\\
\end{aligned}
$$
and
$$\begin{aligned}
\Delta(u_{2}w_{j})(\ok(f^{2s+1}_{(i,j')})\otimes\tk(f^{2s+1}_{(i,j')}))
&=\left\{\begin{array}{lll}
(js+1)e_{i},& j'=2s-js;\\
0,& \mbox{otherwise},
\end{array}\right.\\
&=\left\{\begin{array}{lll}
e_{i},& j'=2s,j=0;\\
(s+1)e_{i},& j'=s,j=1;\\
(2s+1)e_{i},& j'=0,j=2;\\
0,& \mbox{otherwise}.
\end{array}\right.
\end{aligned}
$$
That is to say,
$$
\begin{aligned}
&\Delta(u_{1}w_{0})=(2s+1)w_{0},\qquad &\Delta(u_{1}w_{1})=(r+1)w_{1},\qquad
&\Delta(u_{1}w_{2})=w_{2},\\
&\Delta(u_{2}w_{0})=w_{0}, &\Delta(u_{2}w_{1})=(s+1)w_{1}, &\Delta(u_{2}w_{2})=(2s+1)w_{2}.
\end{aligned}
$$
Using the BV operator $\Delta$, we obtain the Gerstenhaber bracket which is given in
the theorem. If $s$ is even, the proof is similar and will not be repeated here.
\end{proof}


\begin{thebibliography}{101}


\bibitem{AD} Angel A., Duarte D.:
           A BV-algebra Structure on Hochschild Cohomology of the
           Group Ring of Finitely Generated Abelian Groups.
           arXiv:1704.03075.

\bibitem{BGS} Beilinson A., Ginsburg V., Soergel W.:
           Koszul duality patterns in representation theory.
           J. Amer. Math. Soc. {\bf 9}, 473--527 (1996)

\bibitem{BZZ} Bian N., Zhang G., Zhang P.:
           Setwise homotopy category.
           Appl. Categ. Structures {\bf 17}, 561--565 (2009)

\bibitem{BGMS} Buchweitz R.O., Green E.L., Madsen D., Solberg {\O}.:
           Finite Hochschild cohomology without finite global dimension.
           Math. Res. Lett. {\bf 359}, 805--816 (2005)

\bibitem{BK} Butler M.C.R., King A.D.:
           Minimal resolutions of algebras.
           J. Algebra {\bf 212}, 323--362 (1999)

\bibitem{CE} Cartan H., Eilenberg S.:
           Homological algebra. pp. 171-174.
           Princeton University Press, Princeton, New Jersey, (1956)

\bibitem{ET} Ehrig M., Tubbenhauer D.:
           Algebraic properties of zigzag algebras.
           comm. Algebra {\bf 48}, 11--36 (2020)

\bibitem{EHSST} Erdmann K., Holloway M., Snashall N., Solberg {\O}. Taillefer R.:
            Support varieties for selfinjective algebras.
            K-Theory {\bf 33}, 67--87 (2004)

\bibitem{EK} Evseev A., Kleshchev A.:
            Blocks of symmetric groups, semicuspidal KLR algebras
            and zigzag Schur-Weyl duality.
            Ann. Math. {\bf 188}, 453--512 (2018).

\bibitem{Ger} Gerstenhaber M.:
            The cohomology structure of an associative ring.
            Ann. Math. {\bf 78}, 267--288 (1963)

\bibitem{Ge} Getzler E.:
            Batalin-Vilkovisky algebras and two-dimensional topological
            field theories.
            Comm. Math. Phys. {\bf 159}, 265--285 (1994)

\bibitem{Gi} Ginzburg V.:
            Calabi-Yau algebras.
            arXiv: math.AG/0612139.

\bibitem{GH} Green E., Huang R.Q.:
           Projective resolution of straightening closed algebras
           generated by minors.
           Adv. Math. {\bf 110}, 314--333 (1995)

\bibitem{GHM} Green E.L., Hartman G., Marcos E.N., et al:
            Resolution over Koszul algebras.
            Arch. Math. {\bf 85}, 118--127 (2005)

\bibitem{Han} Han Y.:
             Hochschild (co)homology dimension.
             J. London Math. Soc. {\bf 73}, 657-668 (2006).

\bibitem{Ha} Happel D.:
            Hochschild cohomology of finite-dimensional algebras.
            Lecture Notes in Math. Vol {\bf 1404}, pp.108--126, Springer, (1989)

\bibitem{Her} Hermann R.:
            Homolohical epimorphisms, recallements and Hochschild cohomology-with
            a conjecture by Snashall-Solberg in view.
            Adv. Math. {\bf 299}, 687--759 (2016)

\bibitem{Ho} Hochschild G.:
            On the cohomology groups of an associative algebra.
            Ann. Math. {\bf 46}, 58--67 (1945)

\bibitem{Hou} Hou B., Gao J.:
            Hochschild cohomology of zigzag algebras of type $\widetilde{\mathbf{A}}$.
            in preparasion.

\bibitem{HK} Huerfano R.S., Khovanov M.:
            A catefory for the adjoint representation.
            J. algebra {\bf 246}, 514--542 (2001)

\bibitem{Iv} Ivanov A.:
            BV-algebra structure on Hochschild cohomology of
            local algebras of quaternion type in characteristic 2.
            Zap. Nauch Sem. POMI {\bf 430}, 136--185 (2014)

\bibitem{IIVZ} Ivanov A., Ivanov S., Volkov Y., Zhou G.:
            BV structure on Hochschild cohomology of the group
            ring of quaternion group of order eight in characteristic two.
            J. Algebra {\bf 435}, 174--203 (2015)

\bibitem{KMS} Khovanov M., Mazorchuk V., Stroppel C.:
             A brief review of abelian categorifications.
             Theory Appl. Categ. {\bf 22}, 479--508 (2009)

\bibitem{KS} Khovanov M., Seidel P.:
             Quivers, Floer cohomology, and braid group actions.
             J. Amer. Math. Soc. {\bf 15}, 203--271 (2002)

\bibitem{KK} Kowalzig N., Kr$\ddot{a}$hmer U.:
             Batalin-Vilkovisky structures on Ext and Tor.
             J. Reine Angew. Math. {\bf 697}, 159--219 (2014)

\bibitem{LZZ} Lambre Th., Zhou G., Zimmermann A.:
              The Hochschild cohomology ring of a Frobenius algebra with
              semisimple Nakayama automorphism is a Batalin-Vilkovisky algebra.
              J. Algebra {\bf 446}, 103--131 (2016)

\bibitem{LZ} Liu Y., Zhou G.:
            The Batalin-Vilkovisky structure over the Hochschild
            cohomology ring of a group algebra,
            J. Noncommut. Geom. {\bf 10}, 811--858 (2016)

\bibitem{Lo}  Loday J.L.:
             Cyclic homology. Second Edition. pp. 139--153.
             Grundlehren 301, Springer, Berlin, (1998)

\bibitem{PS} Parker A., Snashall N.:
            A family of Koszul self-injective algebras with
            finite Hochschild cohomology.
            J. Pure Appl. Algebra {\bf 216}, 1245--1252 (2012)

\bibitem{SW} Siegel S.F., Witherspoon S.J.:
            The Hochschild cohomology ring of a group algebra.
            Proc. London Math. Soc. {\bf 79}, 131--157 (1999)

\bibitem{SS} Snashall N., Solberg {\O}.:
            Support varieties and Hochschild cohomology rings.
            Proc. London Math. Soc. {\bf 88}, 705--732 (2004)

\bibitem{ST} Snashall N., Taillefer R.:
            The Hochschild cohomology ring of a class of special
            biserial algebras.
            J. Algebra Appl. {\bf 9}, 73--122 (2010)

\bibitem{ST1} Snashall N., Taillefer R.:
            Hochschild cohomology of socle deformations of a class
            of Koszul self-injective algebras.
            Colloq. Math. {\bf 119}, 79--93 (2010)

\bibitem{Su} Suter R.:
            Modules for $\mathfrak{U}_{q}(\mathfrak{sl}_{2})$.
            Comm. Math. Phys. {\bf 163}, 359--393 (1994)

\bibitem{T} Tradler T.:
            The Batalin-Vilkovisky algebra on Hochschild cohomology
            induced by infinity inner products.
            Ann. Inst. Fourier {\bf 58}, 2351--2379 (2008)

\bibitem{Vo} Volkov Y.:
            BV-differential on Hochschild cohomology of Frobenius algebras.
            J. Pure Appl. Algebra {\bf 220}, 3384--3402 (2016)

\bibitem{V} Volkov Y.:
            Gerstenhaber bracket on the Hochschild cohomology via
            an arbitrary resolution.
            P. Edinburgh Math. Soc. {\bf 62}, 817--836 (2019)

\bibitem{X} Xiao J.:
           Finite-dimensional representations of $U_{t}(sl(2))$
           at roots of unity.
           Can. J. Math. {\bf 49}, 772--787 (1997)

\bibitem{Y} Yang T.:
            A Batalin-Vilkovisky algebra structure on the Hochschild
            cohomology of truncated polynomials.
            Topology Appl. {\bf 160}, 1633--1651 (2013)

\bibitem{ZH} Zhao D., Han Y.:
             Koszul algebras and finite Galois covering.
             Sicence in China A {\bf 52}, 2145--2153 (2009)


\end{thebibliography}
\end{document}